\documentclass[a4paper,11pt,final,oneside]{amsart}			
\usepackage[utf8]{inputenc}						
\usepackage[T1]{fontenc}		  			    
\usepackage[danish,english]{babel}			
\usepackage{amsmath,amsfonts,amssymb,amsthm}	
\usepackage{enumerate}				
\usepackage{graphicx}			
\usepackage{tikz}
	\usetikzlibrary{matrix,arrows}
\usepackage{comment}
\usepackage{mathrsfs}
\usepackage{fixme}
\usepackage{mathtools}
\usepackage{amsthm}
\usepackage[final,colorlinks=true,linkcolor=black,citecolor=black]{hyperref}

\usepackage[%
	hmargin={34mm,34mm},%
	vmargin={41.6mm,41.6mm}%
]{geometry}

\theoremstyle{plain}
\newtheorem{theorem}{Theorem}[section]
\newtheorem{lemma}[theorem]{Lemma}

\newtheorem{prop}[theorem]{Proposition}
\newtheorem{corollary}[theorem]{Corollary}

\theoremstyle{definition}
\newtheorem{definition}[theorem]{Definition}
\newtheorem{example}[theorem]{Example}
\newtheorem{remark}[theorem]{Remark}
\newtheorem{ass}[theorem]{Assumptions}

\theoremstyle{plain}   %% This is the default, anyway
\newtheorem{bigthm}{Theorem}   % Numbered separately, as A, B, etc.
\newcommand{\adhoc}[2]{%
	\begin{center}
	\scalebox{#1}{\hbox{#2}}
	\end{center}
	}

\makeatletter
\def\ps@plain{\ps@empty
  \def\@oddfoot{\normalfont\scriptsize \hfil\thepage\hfil}%
  \let\@evenfoot\@oddfoot}
\def\ps@headings{\ps@empty
  \def\@evenhead{%
    \setTrue{runhead}%
    \normalfont\scriptsize
    \hfil
    \def\thanks{\protect\thanks@warning}%
    \leftmark{}{}\hfil}%
  \def\@oddhead{%
    \setTrue{runhead}%
    \normalfont\scriptsize \hfil
    \def\thanks{\protect\thanks@warning}%
    \rightmark{}{}\hfil}%
  \def\@oddfoot{\normalfont\scriptsize \hfil\thepage\hfil}%
  \let\@evenfoot\@oddfoot
}
\makeatother
\author{Amalie Høgenhaven}
\date{\today}
\thanks{Assistance from DNRF Niels Bohr Professorship of Lars Hesselholt is
gratefully acknowledged}

\newlength{\TMP}

\newcommand{\interior}[1]{%
  {\kern0pt#1}^{\mathrm{o}}%
}
\newcommand{\Acirc}{\,\circ\,}
\newcommand{\eqq}[1]{\text{``}#1\text{''}}
\newcommand{\C}{\mathbb{C}}

\newcommand{\R}{\mathbb{R}}
\DeclareMathOperator{\THH}{THH}
\DeclareMathOperator{\TRR}{TRR}
\DeclareMathOperator{\THR}{THR}
\DeclareMathOperator{\TCR}{TCR}
\DeclareMathOperator{\Gal}{Gal}
\DeclareMathOperator{\HE}{HE}
\DeclareMathOperator{\HP}{HP}

\newcommand{\Bdi}{B^{\mathrm{di}}}
\newcommand{\id}{\mathrm{id}}
\DeclareMathOperator*{\hocolim}{hocolim}
\newcommand{\Spe}{\mathsf{Sp}^{\mathrm{O}}}

\DeclareMathOperator*{\holim}{holim}
\DeclareMathOperator*{\colim}{colim}
\DeclareMathOperator{\Map}{Map}
\DeclareMathOperator{\Homeo}{Homeo}
\DeclareMathOperator{\Ob}{Ob}
\newcommand{\ob}{\Ob}
\DeclareMathOperator{\Fun}{Fun}
\DeclareMathOperator{\conn}{conn}
\newcommand{\Top}{\mathsf{Top}}
\DeclareMathOperator{\Fix}{Fix}

\newcommand{\incl}{\mathrm{incl}}
\newcommand{\trf}{\mathrm{trf}}
\newcommand{\proj}{\mathrm{proj}}
\newcommand{\pr}{\mathrm{pr}}
\newcommand{\res}{\mathrm{res}}
\newcommand{\ev}{\mathrm{ev}}
\newcommand{\op}{\mathrm{op}}
\DeclareMathOperator{\ind}{ind}
\DeclareMathOperator{\Arc}{Arc}

\newcommand{\Z}{\mathbb{Z}}

\title{Real Topological Cyclic Homology of Spherical Group Rings}
  
\begin{document}

\begin{abstract}
We compute the $G$-equivariant homotopy type of the real topological cyclic homology of spherical group rings with anti-involution induced by taking inverses in the group, where $G$ denotes the group $\Gal(\mathbb{C}/\mathbb{R})$. The real topological Hochschild homology of a spherical group ring $\mathbb{S}[\Gamma]$, with anti-involution as described, is an $O(2)$-cyclotomic spectrum and we construct a map commuting with the cyclotomic structures from the $O(2)$-equivariant suspension spectrum of the dihedral bar construction on $\Gamma$ to the real topological Hochschild homology of $\mathbb{S}[\Gamma]$, which induce isomorphisms on $C_{p^n}$- and $D_{p^n}$-homotopy groups for all $n\in \Z$ and all primes $p$. Here $C_{p^n}$ is the cyclic group of order $p^n$ and $D_{p^n}$ is the dihedral group of order $2p^n$. Finally, we compute the $G$-equivariant homotopy type of the real topological cyclic homology of $\mathbb{S}[\Gamma]$ at a prime $p$.
\end{abstract}

\maketitle

\tableofcontents

\newpage

\section*{Introduction}
This paper determines the $G$-equivariant homotopy type of the real topological cyclic homology of spherical group rings with anti-involution induced by taking inverses in the group, where $G$ denotes the group $\Gal(\mathbb{C} / \mathbb{R})$ of order 2. Bökstedt-Hsiang-Madsen calculated the topological cyclic homology of spherical group rings in \cite[Section 5]{BHM} and this is a generalization of the classical results. The long term goal of this program is to determine the canonical involution on the stable pseudo-isotopy space of a compact connected topological manifold. The equivariant stable pseudo-isotopy space of a manifold and the real topological cyclic homology of spherical group rings are connected via real algebraic $K$-theory, as explained below.

Recently, Hesselholt and Madsen defined real algebraic $K$-theory in \cite{HM15}, which associates a $G$-spectrum $KR(A,D)$ to a ring spectrum $A$ with anti-involution $D$. Real topological Hochschild homology was also constructed in \cite{HM15} as a generalization of topological Hochschild homology to the $G$-equivariant setting. Real topological Hochschild homology associates an $O(2)$-equivariant orthogonal spectrum $\THR(A, D)$ to a pair $(A,D)$ using a dihedral variant of Bökstedt's model introduced in \cite{Bo}. The generalization leads to a $G$-equivariant version of topological cyclic homology, which we denote $\TCR(A,D)$. Hesselholt and Madsen constructed a $G$-equivariant trace map 
	\[
	\operatorname{tr}: KR(A,D) \to \THR(A,D),
	\]
which factors through $\TCR(A,D)$. 

The real algebraic $K$-theory of spherical group rings have close connections to the geometry of manifolds. Let $M$ be a compact connected topological manifold admitting a smooth structure. A topological pseudo-isotopy of $M$ is a homeomorphism $h: M \times [0,1] \to M \times [0,1]$ which is the identity on $\partial M \times [0,1] \cup M \times \{0\} $. We let $P(M)$ be the space of such homeomorphisms and we note that there is a stabilization map $
P(M) \to P(M \times [0,1])$ 
given by crossing~with the identity. The stable pseudo-isotopy space $\mathscr{P}(M)$ is defined as the homotopy colimit of the stabilization maps. Igusa showed in \cite{Ig}, building on work by Hatcher in \cite{HAT}, that the inclusion $P(M) \to \mathscr{P}(M)$ is $k$-connected if $k$ is less than both $(\dim(M)-7)/2$ and $(\dim(M)-4)/3$. There is a geometric involution on $P(M)$ giving by ``turning a pseudo-istopy upside down,'' see \cite{Ha}, which in turn induces an involution on $\mathscr{P}(M)$. By work of Weiss and Williams~\cite{WWI}, there is map
\[
\widetilde{\operatorname{Homeo}}(M) /\operatorname{Homeo}(M) \to \mathscr{P}(M)_{hG}
\]
which is at least as connected as the stabilization map by Igusa, where the left hand space is the quotient of the space of block homeomorphisms of $M$ by is the space of homeomorphisms of $M$, and $\mathscr{P}(M)_{hG}$ are the homotopy orbits with respect to the involution. Thus information about the involution on the stable pseudo-isotopy yields information about the self-homeomorphisms of the manifold. The stable psuedo-isotopy space can be expressed in terms of Waldhausen algebraic $K$-theory of spaces. If we let $\Gamma$ be the Kan loop group of $M$, (i.e a simplicial group such that $M$ is weakly equivalent to the classifying space $B\Gamma$) then by work of Waldhuasen \cite{Wa},\cite{Wa83} and Waldhausen-Rognes-Jahren \cite{WRJ}, there is a cofibration sequence of spectra
\[
K(\mathbb{S}) \wedge \operatorname{B} \Gamma_+ \to K(\mathbb{S}[\Gamma]) \to \Sigma^2 \mathcal{P}(M) \to \Sigma \big( K(\mathbb{S}) \wedge \operatorname{B} \Gamma_+\big).
\] 
The generalization to real algebraic $K$-theory expresses the equivariant stable pseudo-isotopy space in terms of the real algebraic $K$-theory of spherical group rings with anti-involution and an equivariant understanding of the real topological cyclic homology will via trace methods give information about the real algebraic $K$-theory. In this paper we consider the basic anti-involution on the spherical group rings induced by taking inverses in the group. In order to obtain geometric applications as indicated, it will be necessary to consider more general anti-involutions. 

We proceed to explain the content of this paper. In Section 1 we review the definition of the orthogonal $O(2)$-spectrum $\THR(A,D)$ and we establish an equivariant version of Bökstedt's approximation lemma. 

In Section 2 we observe that the cyclotomic structure of the classical topological Hochschild homology spectrum is compatible with the $G$-action. The compatibility is crucial, if we want a $G$-equivariant version of topological cyclic homology. In order to explain what this means, we let $\mathbb{T}$ denote the multiplicative group of complex numbers of modulus 1. The group $G=\Gal(\mathbb{C}/\mathbb{R})$ acts on $\mathbb{T}$ and $O(2)$ is the semi-direct product $O(2)= \mathbb{T} \rtimes G$. Let
\[
\rho_r: O(2)\to O(2)/C_r 
\] 
be the root isomorphism given by $\rho_r(z)=z^{\frac{1}{r}}C_r$ if $z\in \mathbb{T}$ and $\rho_r(x)=x$ if $x\in G$ and let $\THR(A,D)^{gC_r}$ denote the $C_r$-geometric fixed points.  The $O(2)$-cyclotomic structure is a collection of compatible $O(2)$-equivariant maps
\[T_r:  \rho_r^* \left( \THR(A,D)^{gC_r}\right) \to \THR(A, D)\]
which induce weak equivalences on $H$-fixed points for all finite subgroups $H \leq O(2)$, when pre-composed with the canonical map from the derived $C_r$-geometric fixed points. 

In Section 3 we observe that the $O(2)$-cyclotomic structure on $\THR(A,D)$ gives rise to $G$-equivariant restriction maps 
\[R_n: \THR(A,D)^{C_{p^n}} \to \THR(A,D)^{C_{p^{n-1}}},\] 
and we define the real topological cyclic homology at a prime $p$ as an orthogonal $G$-spectrum $\TCR(A,D; p)$ by mimicking the classical definition by Bökstedt-Hsiang-Madsen in \cite{BHM}. We define a $G$-spectrum $\TRR(A,D;p)$ as the homotopy limit over the $R_n$ maps
\[\TRR(A,D;p):=\holim_{n, R_n} \THR(A, D)^{C_{p^n}} .\]
The Frobenius maps, which are inclusion of fixed points,  induce a self-map of the \mbox{$G$-spectrum} $\TRR(A,D;p)$, which we denote $\varphi$, and the real topological cyclic homology of $(A,D)$ at $p$, $\TCR(A,D;p)$, is the homotopy equalizer of the maps
 \[\begin{tikzpicture}
\setlength{\TMP}{2.5pt}
\node(a){$\TRR(A,D;p)$};
\node(b)[right of=a, node distance = 3.5cm]{$\TRR(A,D;p).$};
\draw[->] ([yshift=\TMP]a.east) to node [above]{$\scriptstyle \varphi$} ([yshift=\TMP]b.west);
\draw[->] ([yshift=-\TMP]a.east) to node [below]{$\scriptstyle\id$} ([yshift=-\TMP]b.west);

\end{tikzpicture}\]

Section 4 computes the real topological Hochschild homology and the real topological cyclic homology of a spherical group ring $\mathbb{S}[\Gamma]$ with anti-involution $\id[\Gamma]$ induced by taking inverses in the group. In particular, we obtain a calculation of the real topological cyclic homology of the sphere spectrum with the identity serving as anti-involution.

In the rest of the introduction, we state our computational results for spherical group rings. The proofs of Theorems A, B and C below can be found in Section~4. 

We let $\Bdi\Gamma$ denote the geometric realization of the dihedral bar construction on $\Gamma$. The space $\Bdi\Gamma$ is weakly equivalent to the free loop-space $\Map(\mathbb{T}, B\Gamma)$ by \mbox{\cite[Theorem 7.3.11]{Lo}}. It follows from \cite[Section 2.1]{Ma} and \cite[Theorem 4.0.5]{SN} that the weak equivalence induces isomorphism on $\pi_*^H(-)$ for all finite subgroups $H\leq O(2)$, if we let $O(2)$ act on the free loop space as follows: The group $O(2)$ acts on $\mathbb{T}$ by multiplication and complex conjugation. Taking inverses in the group induces a \mbox{$G$-action} on $B\Gamma$ and we view $B\Gamma$ as an $O(2)$-space with trivial $\mathbb{T}$-action. Finally, $O(2)$ acts on the free loop space by the conjugation action. There are $O(2)$-equivariant homeomorphisms
\[ p_r: \Bdi\Gamma \xrightarrow{\cong} \rho_{r}^* \left( \Bdi\Gamma^{C_r} \right),\]
which under the identification with the free loop space correspond to the maps, which take a loop to the $r$-fold concatenation with itself. These maps give the \mbox{$O(2)$-equivariant} suspension spectrum of $ \Bdi\Gamma$ an $O(2)$-cyclotomic structure. 

\begin{bigthm}
Let $\Gamma$ be a topological group. There is a map of $O(2)$-orthogonal spectra
\[i: \Sigma^\infty_{O(2)} \Bdi\Gamma_+ \to \THR(\mathbb{S}[\Gamma], \id[\Gamma]),\]
commuting with the cyclotomic structures, which induces isomorphisms on $\pi_*^{C_{p^n}}(-)$ and $\pi_*^{D_{p^n}}(-)$ for all $n\geq 0$ and all primes $p$. 
\end{bigthm}

The $G$-equivariant restriction maps
\[
R_n: \THR(\mathbb{S}[\Gamma], \id[\Gamma])^{C_{p^n}} \to \THR(\mathbb{S}[\Gamma], \id[\Gamma])^{C_{p^{n-1}}}
\]
admits a canonical section, which provides a canonical isomorphism in the $G$-stable homotopy category
\[ \TRR(\mathbb{S}[\Gamma],\id[\Gamma];p) \sim  \prod_{j=0}^{\infty}\operatorname{H}_{\boldsymbol{\cdot}}(C_{p^j};  \Bdi\Gamma),\] 
where $\operatorname{H}_{\boldsymbol{\cdot}}(C_{p^j};  \Bdi\Gamma)$ is the $G$-equivariant Borel group homology spectrum of the subgroup $C_{p^j}$ acting on the $O(2)$-spectrum $\Sigma^\infty_{O(2)}\Bdi\Gamma$. We have a $\pi_*$-isomorphisms of $G$-spectra $c: \operatorname{H}_{\boldsymbol{\cdot}}(1;  \Bdi\Gamma) \xrightarrow{c} \Sigma^{\infty}_{G}\Bdi\Gamma_+$ given by collapsing a certain classifying space in the construction of the homology spectrum. We let $\Delta_p$ denote the $G$-equivariant composition
\[
\Bdi\Gamma \xrightarrow{p_p} \Bdi\Gamma^{C_p} \hookrightarrow \Bdi\Gamma,
\]
and we let $\Sigma^{\infty}_{G} \Bdi\Gamma_+^{\eqq{\Delta_p=\id}} $ denote the homotopy equalizer of $\Sigma^{\infty}{\Delta_p}_+$ and the identity map. We have a canonical inclusion $\iota: \Omega (\Sigma^{\infty}_{G} \Bdi\Gamma_+) \to \Sigma^{\infty}_{G} \Bdi\Gamma_+^{\eqq{\Delta_p=\id}}$. The inclusion and the projection
\[\Sigma^{\infty}_{G} \Bdi\Gamma_+ \xrightarrow{\incl} \Sigma^{\infty}_{G} \Bdi\Gamma_+  \times \prod_{j=1}^{\infty}\operatorname{H}_{\boldsymbol{\cdot}}(C_{p^j};  \Bdi\Gamma) \xrightarrow{\text{proj}}  \prod_{j=1}^{\infty}\operatorname{H}_{\boldsymbol{\cdot}}(C_{p^j};  \Bdi\Gamma),
 \]
induce the maps $I$ and $P$ in the theorem below, where the homotopy limit is constructed with respect to inclusions. 
	\begin{bigthm}
The triangle 
\begin{align*}
\Sigma^{\infty}_{G} \Bdi\Gamma_+^{\eqq{\Delta_p=\id}}  \xrightarrow{I} \TCR(\mathbb{S}[\Gamma], \id[\Gamma]; p) \xrightarrow{P} \holim_{j\geq 1}\operatorname{H}_{\boldsymbol{\cdot}}(C_{p^j};  \Bdi\Gamma) \\
%\xrightarrow{\hbox to 11em{$- \Sigma(\iota) \circ \varepsilon^{-1}\circ c \circ \incl \circ \pr_1$}}
\xrightarrow{{- \Sigma(\iota) \Acirc \varepsilon^{-1}\Acirc c \Acirc \incl \Acirc \pr_1}}
\Sigma \left(\Sigma^{\infty}_{G} \Bdi\Gamma_+^{\eqq{\Delta_p=\id}}\right)
\end{align*}
is distinguished in the $G$-stable homotopy category, where $\varepsilon: \Sigma \Omega Y \to Y$ is the counit of the loop-suspension adjunction.
	\end{bigthm}

After $p$-completion, a non-equivariant identification of the homotopy limit in the triangle appears in \cite{BHM}, and the result generalizes immediately to the equivariant setting. We let $O(2)$ act on  $S(\mathbb{C})$ by multiplication and complex conjugation. The homeomorphism $S(\mathbb{C}) \star S(\mathbb{C}) \star \cdots \cong S(\mathbb{C}^{\infty})$ gives $S(\mathbb{C}^{\infty})$ an $O(2)$-action and there is an isomorphism in the $G$-stable category after $p$-completion 
\[
\Sigma^{1,1}  S(\mathbb{C}^{\infty})_+ \wedge_{\mathbb{T}} \Sigma^{\infty}_{O(2)} \Bdi\Gamma_+  \to  \holim_{j\geq 1}\operatorname{H}_{\boldsymbol{\cdot}}(C_{p^j};  \Bdi\Gamma),
\] 
where $\Sigma^{1,1}$ denotes suspension with respect to the sign representation of $G$. 

The theorem above leads to a calculation of the topological cyclic homology of the sphere spectrum with the identity serving as anti-involution. We let $\mathbb{P}^{\infty}(\mathbb{C})$ denote the infinite complex projective space with $G$ acting by complex conjugation and we let $\Sigma^{1,1}$ denote suspension with respect to the sign representation of $G$.

\begin{bigthm}
After $p$-completion, there is an isomorphism in the $G$-stable homotopy category 
\[
\TCR(\mathbb{S},\id ;p) \sim \Sigma^{1,1} \mathbb{P}^{\infty}_{-1}(\mathbb{C}) \vee \mathbb{S},
\]
where $\Sigma^{1,1} \mathbb{P}^{\infty}_{-1}(\mathbb{C})$ denotes the homotopy fiber of the $\mathbb{T}$-transfer $ \Sigma^{\infty}_{G} \Sigma^{1,1}\mathbb{P}^{\infty}(\mathbb{C})\to \mathbb{S}$. 
\end{bigthm}

\medskip
\noindent\textbf{Acknowledgments.} I would like to express my sincere gratitude to my supervisor Lars Hesselholt for his inspiring guidance and support throughout this project. My thanks are due to Ib Madsen for originally encouraging my interest in topological cyclic homology and for teaching me about the classical results underlying this paper. I would also like to thank Irakli Patchkoria, Kristian Moi, and Cary Malkiewich for several helpful and inspiring conversations that have helped me write this paper.

\newpage
Throughout this paper, $\mathbb{T}$ denotes the multiplicative group of complex numbers of modulus 1 and $G$ is the group $\Gal(\C/\R)=\{1,\omega\}$ of order 2. The group $G$ acts on $\mathbb{T} \subset \mathbb{C}$ and $O(2)$ is the semi-direct product $O(2)= \mathbb{T} \rtimes G$. Let $C_r$ denote the cyclic subgroup of order $r$, generated by the $r$th root of unity $t_r:=e^{\frac{2\pi i}{r}}$. Let $D_{r}$ denote the dihedral subgroup of order $2r$ generated by $t_r$ and $\omega$. The collection $\{ C_r,D_{r}\}_{r\geq 0}$ represents all conjugacy classes of finite subgroups of $O(2)$. 

By a space we will always mean a compactly generated weak Hausdorff space and any construction is always carried out in this category.

\section{Real topological Hochschild homology}
A symmetric ring spectrum $X$ is a sequence of based spaces $X_0, X_1, \dots$ with a left based action of the symmetric group $\Sigma_n$ on $X_n$ and $\Sigma_n \times \Sigma_m$-equivariant maps $\lambda_{n,m} : X_n \wedge S^m  \to X_{n+m}$. Let $A$ be a symmetric ring spectrum with multiplication maps $\mu_{n,m}: A_n \wedge A_m \to A_{n+m}$ and unit maps $1_n: S^n \to A_n$. An anti-involution on $A$ is a self-map of the underlying symmetric spectrum $D: A \to A$, such that 
\[D^2=\id, \quad D_n \circ 1_n=1_n,\] 
and the following diagram commutes:
\begin{center}
\begin{tikzpicture}
  \matrix (m) [matrix of math nodes,row sep=2em,column sep=4em,minimum width=2em] {
 A_m \wedge A_n & A_{m} \wedge A_{n} \\
  & A_{n} \wedge A_m\\
   & A_{n+m} \\
    A_{m+n}    & A_{m+n}.\\};
  \path[-stealth]
      (m-1-1) edge node [above] {$D_m \wedge D_n$} (m-1-2)          
    (m-1-2) edge node [right] {$\gamma$} (m-2-2)
    (m-2-2) edge node [right] {$\mu_{n,m} $} (m-3-2)
    (m-3-2) edge node [right] {$\chi_{n,m}$} (m-4-2)
        (m-1-1) edge node [right] {$\mu_{m,n}$} (m-4-1)
    (m-4-1) edge node [above] {$D_{m+n}$} (m-4-2)
  
    ;
\end{tikzpicture}
\end{center}
Here $\gamma$ is the twist map and $\chi_{n,m}\in \Sigma_{n+m}$ is the shuffle permutation
\[\chi_{n,m}(i)= \left\{ \begin{array}{rl}
 i+m &\mbox{ if $1 \leq i \leq n$} \\
  i-n &\mbox{ if $n+1 \leq i \leq n+m.$}
       \end{array} \right. \]
If $A$ is commutative, then the identity defines an anti-involution on $A$. 

If $\Gamma$ is a topological group, then the spherical group ring $\mathbb{S}[\Gamma]$ is the symmetric suspension spectrum of $\Gamma_+$.  If $m: \Gamma \times \Gamma \to \Gamma$ denotes the multiplication in $\Gamma$ and $1\in \Gamma$ is the unit, then the spherical group ring becomes a symmetric ring spectrum with multiplication maps defined to be the compositions
\begin{align*}
\Gamma_+ \wedge S^n \wedge \Gamma_+ \wedge S^m \xrightarrow{\id \wedge \gamma \wedge \id } (\Gamma \times\Gamma)_+ \wedge S^n  \wedge S^m \xrightarrow{m_+ \wedge \mu_{n,m}}\Gamma_+ \wedge S^{n+m},
\end{align*}
and unit maps $S^n \to \Gamma_+ \wedge S^n$ given by $z \mapsto 
1 \wedge z$. Taking inverses in the group induce an anti-involution $\id[\Gamma]$ on the spherical group ring, where  $\id[\Gamma]_n:  \Gamma_+ \wedge S^n  \to \Gamma_+ \wedge S^n$ is given by
 \[
 \id[\Gamma]_n(g \wedge x)=g^{-1} \wedge x
 .\]

 The real topological Hochschild homology space $\THR (A,D)$ of a symmetric ring spectrum $A$ with anti-involution $D$ was defined in \cite[Sect. 10]{HM15} as the geometric realization of a dihedral space. Before reviewing the definition, we recall the notion of dihedral sets, and their realization.  

\begin{definition}\label{didef}
A dihedral object in a category $\mathcal{C}$ is a simplicial object 
\[
X[-]: \triangle^{\op} \to \mathcal{C}
\] 
together with dihedral structure maps $t_k, w_k: X[k] \to X[k]$ such that $t_k^{k+1}=\id$, $w_k^2=\id$, and $t_k w_k=t_k^{-1}\omega_k$. The dihedral structure maps are required to satisfy the following relation involving the simplicial structure maps:
\begin{align*}
d_l w_k=w_{k-1}d_{k-l}, &\quad s_l w_k=w_{k+1}s_{k-l} \ \text{ for }0\leq l\leq k, \\ 
d_l t_k=t_{k-1}d_{l-1}, &\quad s_l t_k=t_{k+1}s_{l-1} \ \ \ \text{ for } 0<l\leq k,  \\
 d_0t_k=d_k, &\quad s_0t_k=t_{k+1}^2s_k.
\end{align*}
\end{definition}

We let $\mathsf{Sets}$ denote the category of sets and set maps and we let $X[-]$ be a dihedral object in $\mathsf{Sets}$. We use Drinfeld's realization of dihedral sets which is naturally homeomorphic as an $O(2)$-space to the ordinary geometric realization of the underlying simplicial set with $O(2)$-action arising from the dihedral structure. We briefly recall the method; see \cite{Sa} and \cite{Dr} for details.

Let $\mathcal{F}$ denote the category with objects all finite subsets of the circle and morphisms set inclusions. A dihedral set $X[-]$ extends uniquely, up to unique isomorphism, to a functor $X[-]: \mathcal{F} \to \mathsf{Sets}$. Given an inclusion $F \subset F' \subset \mathbb{T}$ the degeneracy maps give rise to a map $s_F^{F'}: X[F] \to X[F']$. As a set 
\[ \vert X[-] \vert:= \colim_{F \in \mathcal{F}} X[F].\]
Given an inclusion $F \subset F' \subset \mathbb{T}$ the face maps give rise to a map $d_F^{F'}:X[F'] \to X[F]$. These maps are used to define a topology on the colimit. Finally, the dihedral structure maps give rise to a continuous action of the homeomorphism group $\Homeo(\mathbb{T})$ on the colimit as follows. A homeomorphism $h: \mathbb{T} \to \mathbb{T}$ induces a functor $ \mathcal{F} \to \mathcal{F}$ given on objects by $F \mapsto h(F)$. The dihedral structure maps give rise to a natural transformations $\phi_h: X[-] \Rightarrow X[-] \circ h$. The action of $h$ on the colimit is given as the composition:
 \[ 
 \colim_{F \in \mathcal{F}} X[F] \xrightarrow{\phi_h}  \colim_{F \in \mathcal{F}} X[h(F)] \xrightarrow{\ind_{h}}  \colim_{F \in \mathcal{F}} X[F] . 
 \]
In particular, the subgroup $O(2)<\Homeo(\mathbb{T})$ acts on the realization. Note that the category $\mathcal{F}$ is filtered, and therefore the realization, as a functor valued in $\mathsf{Sets}$, commutes with finite limits. This is also true as functors valued in $\mathsf{Top}$, the category of compactly generated weak Hausdorff spaces; e.g. \cite[Corollary 1.2]{Sa}. The functor $X[-]: \mathcal{F} \to \mathsf{Sets}$ arising from a dihedral set is a special case of an $O(2)$-diagram, as defined below, and we recall have such diagrams behave.

\begin{definition}
Let $H$ be a group and let $J$ be a small category with a left $H$-action. An $H$-diagram indexed by $J$ is a functor $X: J\to \mathsf{Sets}$ together with a collection of natural transformations
\[
\alpha=\{ h\in H \mid \alpha_h: X \Rightarrow X \circ h \}.
\]
such that $\alpha_e=\id$ and $(\alpha_{h'})_h \circ \alpha_h =\alpha_{h'h}$. Here $(\alpha_{h'})_h$ is the natural transformation obtained by restricting $\alpha_{h'}$ along the functor $h:J \to J$.
\end{definition}

The group $H$ acts on the colimit of the diagram by letting $h\in H$ act as the composition
\[ \colim_{J} X \xrightarrow{\alpha_h} \colim_{J} X \circ h \xrightarrow{\ind_{h}}  \colim_{J} X.\]
A natural transformation of $H$-diagrams indexed by $J$ commuting with the group action induces an $H$-equivariant map of colimits. Let $\phi: K \to J$ be a functor between small categories with $H$-actions, such that $\phi \circ h=h \circ \phi$. If $(X: J \to \mathsf{Sets}, \alpha)$ is an $H$-diagram indexed by $J$, then $(X \circ \phi: K \to \mathsf{Sets}, \alpha_{\phi})$ is an $H$-diagram indexed by $K$ and the canonical map
\[\colim_{K} X \circ \phi \xrightarrow{\ind_\phi} \colim_{J} X \]
is $H$-equivariant. Finally, if $J$ is filtered and $N \leq H$ is a normal subgroup acting trivially on $J$ then $(X^N: J \to \mathsf{Sets}, \alpha^N)$ is an  $H/N$-diagram indexed by $J$ and the canonical inclusion $X(j)^N \hookrightarrow X(j)$ induces an $H/N$-equivariant bijection 
\[
\colim_{J} (X^N) \xrightarrow{\cong} \big(\colim_{J} X\big)^N. 
\]

In the following, we would like to apply Drinfeld's realization to the case of a dihedral pointed space $Y[-]$, but a priori Drinfeld's description of the topology on the realization only apply to a simplical set. We solve this problem as follows. We have a bijection of underlying sets from the ordinary geometric realization to the colimit as specified in \cite{Sa} and \cite{Dr}:
\[\Big( \bigvee_{n=0}^\infty Y[n] \times \Delta^n /\sim\Big) \ \to \ \colim_{F \in \mathcal{F}} Y[F] .\]
The bijection is $O(2)$-equivariant and natural with respect to functors of simplicial pointed spaces. 
We give the right hand colimit the topology which makes the above bijection a homeomorphism. Of course, if $Y[-]$ is discrete, then the topology in the description of Drinfeld's realization makes the bijection into a homeomorphism.

When we describe dihedral spaces in this work, we will give the space $X[F]$ at every object $F\in \mathcal{F}$, for each inclusion $F \subset F' \subset \mathbb{T}$, the maps $d_F^{F'}: X[F] \to X[F']$ and $s_F^{F'}: X[F'] \to X[F]$, and the transformations $X[F] \to X[h(F)]$ for $h \in \Homeo(\mathbb{T})$. 

We are now ready to construct to real topological homology of a ring spectrum $A$ with anti-involution $D$. First let $I$ be the category with objects all non-negative integers. The morphisms from $i$ to $j$ are all injective set maps
\[\{1,\dots,i\}  \to \{1,\dots, j\}. \] 
The category $I$ has a strict monoidal product $+: I\times I \to I$ given on objects by addition and on morphisms by concatenation. Given an object $i$ let $\omega_i: i \to i$ denote the involution given by
\[\omega_i(s)=i-s+1, \]
and we define the conjugate of a morphism $\alpha: i \to j$ by $\alpha^{\omega}:=\omega_j \circ \alpha \circ \omega_i^{-1}$.

Let $F\subset F' $ be finite subsets of the circle. We define $s_F^{F'}: I^F \to I^{F'}$ on objects by $(i_z)_{z \in F}  \mapsto (i_{z'})_{z' \in F'}$ where
\[  i_{z'}= \left\{\begin{array}{rl}
 i_{z'} &\mbox{ if $z' \in F$} \\
  0 &\mbox{ if $z' \notin F.$}
       \end{array} \right.\] 
Hence the functor repeats the initial object $0 \in I$ as pictured in the example:
       
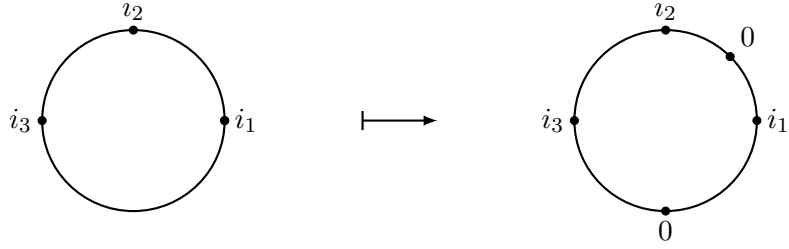
\begin{center}
    \begin{tikzpicture}[scale=1,cap=round,>=latex]
        \draw[thick] (0cm,0cm) circle(1.2cm);

        \foreach \x in {0,90,180,360} {
                \filldraw[black] (\x:1.2cm) circle(1.5pt);   }
 
        \draw[thick,|->] (3cm,0cm) --  (4cm,0cm);
        \draw[thick] (7cm,0cm) circle(1.2cm);

        \foreach \y in {0,45,90,180,270} {
                \filldraw[black](7cm,0cm)++ (\y:1.2cm) circle(1.5pt);   }
                
                        \draw (0:1.2cm) node[right] {$i_{1}$};
                \draw (90:1.2cm) node[above] {$i_{2}$};
                        \draw (180:1.2cm) node[left] {$i_{3}$};         
                 \draw (7cm,0cm)++(0:1.2cm) node[right] {$i_{1}$};
                \draw (7cm,0cm)++(90:1.2cm) node[above] {$i_{2}$};
                        \draw (7cm,0cm)++(180:1.2cm) node[left] {$i_{3}$};             
               \draw (7cm,0cm)++(45:1.2cm) node[above right] {$0$};
                \draw (7cm,0cm)++(270:1.2cm) node[below] {$0$};
              
    \end{tikzpicture}
        \end{center}   
This also defines the functor on morphisms, since $0$ is the initial object in $I$.

In order to define $d_F^{F'}: I^{F'} \to I^{F}$ we introduce to following notation. Given $z=e^{2\pi i s} \in F$, let $\overline{s}=\max_t\{s \leq t < 1+ s \mid  e^{2 \pi i t } \in F\}$. Then we set 
\[\
\Arc_{z}:= \{ e^{2\pi i t}\mid \overline{s} < t \leq 1+ s\} \subset \mathbb{T},
\]
hence $\Arc_z$ is the circle arc starting from the first element counter clockwise of $z$ and ending at $z$. We define $d_F^{F'}: I^{F'} \to I^{F}$ on objects by $ (i_{z'})_{z' \in F'} \mapsto (i_z)_{z \in F} $ where
\[i_z = \sum_{z' \in \Arc_{z}} i_{z'},\]
and similarly on morphisms. Hence the functor adds together the objects counter clockwise around the circle and concatenate morphisms, as pictured in the example:
       
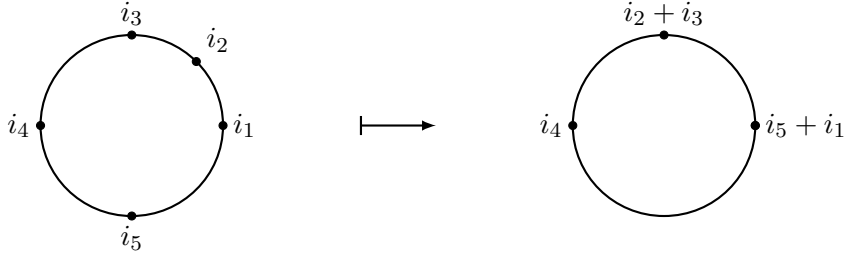
\begin{center}
          \hspace*{0.6cm}
    \begin{tikzpicture}[scale=1,cap=round,>=latex]
        % draw the unit circle
        \draw[thick] (0cm,0cm) circle(1.2cm);

        \foreach \x in {0,90,180,360} {
            % dots at each point
                \filldraw[black](7cm,0cm)++ (\x:1.2cm) circle(1.5pt);   }

        \draw[thick,|->] (3cm,0cm) --  (4cm,0cm);
        % draw the unit circle
        \draw[thick] (7cm,0cm) circle(1.2cm);

        \foreach \y in {0,45,90,180,270} {
            % dots at each point
                \filldraw[black] (\y:1.2cm) circle(1.5pt);   }
                
                        \draw (0:1.2cm) node[right] {$i_{1}$};
                \draw (90:1.2cm) node[above] {$i_{3}$};
                        \draw (180:1.2cm) node[left] {$i_{4}$};
                                                            
                                        \draw (45:1.2cm) node[above right] {$i_2$};
                \draw (270:1.2cm) node[below] {$i_5$};
                        
                                        \draw (7cm,0cm)++(0:1.2cm) node[right] {$i_5 +i_{1}$};
                \draw (7cm,0cm)++(90:1.2cm) node[above] {$i_{2}+i_3$};
                        \draw (7cm,0cm)++(180:1.2cm) node[left] {$i_{4}$};

    \end{tikzpicture}
        \end{center}  
Given a pointed left $O(2)$-space, we define a dihedral space $\THR(A,D;X)[-]$. Let $F \subset \mathbb{T}$ be finite and define a functor $G_X^F: I^F \to \Top_*$ on objects by
\[G_X^F\left((i_z)_{z\in F}\right)= \Map \Big(\bigwedge_{z\in F} S^{i_z}, \bigwedge_{z\in F} A_{i_z}\wedge X  \Big). \]
We define $G^F_X$ on morphisms in the case $\vert F \vert =1$, the general case is similar. 
Let $\alpha: i\to j$ be a morphism in $I$. We write $\alpha$ as a composite $\alpha=\sigma \circ \iota$, where 
\[
\iota: \{1,\dots, i\} \to \{1,\dots,j\} 
\] 
is the standard inclusion and $\sigma \in \Sigma_j$. The map $G^F_X(\alpha)$ is the composite 
\[\begin{tikzpicture}[x=3.5cm,y=1.9cm]
\node(a) at (0,1) {$\Map (S^i,A_i\wedge X) $};
\node(b) at (1,0) {$\Map (S^{j-i} \wedge S^{i}, S^{j-i}\wedge A_i\wedge X)$};
\node(c) at (2,1) {$ \Map (S^j, A_j \wedge X)$};
\node(d) at (3,0) {$\Map (S^j, A_j \wedge X),$};

\draw[->](a) to node [fill=white] {$\scriptstyle S^{j-i} \wedge (-)$} (b);

\draw[->](b) to node [fill=white] {$\scriptstyle (\lambda_{j-i,i}\wedge \id)\circ (-) $} (c);

\draw[->](c) to node [fill=white] {$\scriptstyle\Map (\sigma^{-1}, \sigma\wedge \id)$} (d);

\end{tikzpicture}\]
which is independent of choice of $\sigma$, since $\lambda_{j,i}$ is $\Sigma_j \times \Sigma_i$ - equivariant. We set
\[
\THR(A,D;X)[F] := \hocolim_{I^{F}} G_X^F.
\]
Let $F \subset F'$. We define the map $\THR(A,D;X)[F]\to \THR(A,D;X)[F']$ by first defining a natural transformation of functors $(s')_{F}^{F'}: G^F_X \Rightarrow G^{F'}_X \circ {s}_{F}^{F'}$. At an object $(i_z)_{z\in F} \in I^F$ the natural transformation
 \[
\Map \Big(\bigwedge_{z\in F} S^{i_z} ,\bigwedge_{z\in F} A_{i_z}\wedge X\Big) \to \Map \Big(\bigwedge_{z'\in F'} S^{i_{z'}} ,\bigwedge_{z'\in F'} A_{i_{z'}}\wedge X\Big),
 \]
uses the identity map $1_0: S^0 \to A_0$.  We let the forward map be the composition
 \begin{align*}
\hocolim_{I^{F}} G^F_X  \xrightarrow{(s')_{F}^{F'}} \hocolim_{I^{F}} G^{F'}_X \circ s_{F}^{F'} \xrightarrow{\ind_{s_{F}^{F'}}}
\hocolim_{I^{F'}} G^{F'}_X.
\end{align*}
We define the map $\THR(A,D;X)[F']\to \THR(A,D;X)[F]$ by first defining a natural transformation of functors $ (d')_{F}^{F'}: G^{F'}_X \Rightarrow G^{F}_X \circ d_{F}^{F'}$. At an object $(i_{z'})_{z' \in F'} \in  I^{F'}$ the natural transformation
 \[
 \Map \Big(\bigwedge_{z'\in F'} S^{i_{z'}} ,\bigwedge_{z'\in F'} A_{i_{z'}}\wedge X\Big) \to \Map \Big(\bigwedge_{z\in F} S^{i_{z}} ,\bigwedge_{z\in F} A_{i_{z}}\wedge X\Big),
 \]
uses the multiplication maps in $A$. We let the backward map be the composition
 \begin{align*}
\hocolim_{I^{F'}} G^{F'}_X  \xrightarrow{(d')_{F}^{F'}} \hocolim_{I^{F'}} G^F_X \circ d_{F}^{F'} \xrightarrow{\ind_{d_{F}^{F'}}}
\hocolim_{I^{F}} G^{F}_X.
\end{align*}
 Finally, we describe the action of $\Homeo(\mathbb{T})$. We restrict our attention to the subgroup $O(2)$ and describe the transformations induced by $t \in \mathbb{T}$ and complex conjugation $\omega \in G$. We define functors $t_F :I^F \xrightarrow{} I^{t (F)}$ and $w_F: I^F \xrightarrow{} I^{\omega (F)}$ on objects and morphisms by:
\begin{align*}
t_F: \ \ &(i_z)_{z\in F} \mapsto (i_{t^{-1}(y)})_{y\in t(F)}, \hspace{1cm}  (\alpha_z)_{z\in F} \mapsto (\alpha_{t^{-1}(y)})_{y\in t(F)}, \\
w_F: \ \ &(i_z)_{z\in F} \mapsto (i_{\omega^{-1}(y)})_{y\in \omega(F)}, \hspace{.8cm} (\alpha_z)_{z\in F} \mapsto (\alpha^{\omega}_{\omega^{-1}(y)})_{y\in \omega(F)}.
\end{align*}
Let $\omega_i \in \Sigma_i$ be the permutaion given by $\omega_i(s)=i-s+1$. We define the natural transformations
\[t'_F: G_X^{F} \Rightarrow G_X^{t(F)} \circ t_F, \quad  w'_F: G_X^{F} \Rightarrow G_X^{\omega(F)} \circ w_F,\]
to be the natural transformations which at $(i_z)_{z\in F} \in \Ob(I^F)$ are the unique maps making the diagrams 
 \[
 \begin{tikzpicture}
\node(a){$\displaystyle\bigwedge_{z\in F} S^{i_z}$}; 
\node(b)[right of = a,node distance = 6cm]{$\displaystyle\bigwedge_{z\in F} A_{i_z}\wedge X$};
\node(e)[below of = a,node distance = 2cm]{$\displaystyle\bigwedge_{y\in t(F)} S^{i_{t^{-1}(y)}}$};
\node(f)[right of = e,node distance = 6cm]{$\displaystyle\bigwedge_{y\in t(F)} A_{i_{t^{-1}(y)}}\wedge X$};

\draw[->](a) to node [above]{$g$} (b);

\draw[->](e) to node [left]{} (a);
\draw[->](b) to node [right]{} (f);

\draw[->](e) to node [above]{$t'_F(g)$} (f);

\end{tikzpicture}
 \]
and
 \[\begin{tikzpicture}
 
\node(l){$\displaystyle\bigwedge_{y\in \omega(F)} S^{i_{\omega^{-1}(y)}}$};
\node(m)[right of = l,node distance = 6cm]{$\displaystyle\bigwedge_{y\in \omega(F)} A_{i_{\omega^{-1}(y)}}\wedge X$};

\node(j)[above of = l,node distance = 2cm]{$\displaystyle\bigwedge_{z\in F} S^{i_z}$}; 
\node(k)[right of = j,node distance = 6cm]{$ \displaystyle\bigwedge_{z\in F} A_{i_z}\wedge X$};

\node(h)[above of = j,node distance = 2cm]{$\displaystyle\bigwedge_{z\in F} S^{i_z}$}; 
\node(i)[above of = k,node distance = 2cm]{$ \displaystyle\bigwedge_{z\in F} A_{i_z}\wedge X$};

\draw[->](l) to node [left]{} (j);
\draw[->](j) to node [left]{$\scriptstyle  \underset{z \in F}{\wedge}\omega_{i_z}$} (h);

\draw[->](l) to node [above]{$w_F'(g)$} (m);
\draw[->](h) to node [above]{$g$} (i);

\draw[->](i) to node [right]{$\scriptstyle  \underset{z \in F}{\wedge} D_{i_z} \circ \omega_{i_z}\wedge \id$} (k);

\draw[->](k) to node [left]{} (m);

\end{tikzpicture}\]
commute. The unlabelled vertical maps are appropriate permutations of the smash factors. More precisely,  given families of pointed spaces $\{X_i\}_{i\in I}$ and  $\{X_j\}_{j\in J}$, then a set bijection $\phi: I \to J$ together with a collection of pointed maps $\phi_i: X_i \to X_{\phi(i)}$  induces a map of indexed smash products
\[
\bigwedge_{i\in I} X_i \to \bigwedge_{j\in J} X_j, \quad (x_i)_{i\in I} \mapsto (\phi_i(x_{\phi^{-1}(j)}))_{j \in J}.
\]
In the case at hand, the space maps are all identity maps and the set bijections $t(F) \to F$, $F \to t(F)$, $\omega(F) \to F$ and $F \to \omega(F)$ are given by $t^{-1}$, $t$, $\omega^{-1}$ and $\omega$, respectively. The natural transformations are given at $F\in \mathcal{F}$ as the compositions
 \begin{align*}
t_F: \hocolim_{I^{F}} G_X^{F}  &\xrightarrow{t'_F} \hocolim_{I^{F}} G_X^{t(F)} \circ t_F  \xrightarrow{\ind_{t_F}} \hocolim_{I^{t(F)}} G_X^{t(F)},\\
w_F: \hocolim_{I^{F}} G_X^{F}  &\xrightarrow{w'_F} \hocolim_{I^{F}} G_X^{\omega(F)} \circ w_F  \xrightarrow{\ind_{w_F}}
\hocolim_{I^{\omega(F)}} G_X^{\omega(F)}.
\end{align*}
We have now defined a dihedral space and we let $ \THR(A, D;X)$ be the realization
 \[ 
 \THR(A, D;X):= \colim_{F \in \mathcal{F}} \THR(A,D;X)[F] ,
 \]
with the topology given by identifying the colimit with the ordinary geometric realization. The space $\THR(A, D;X)$ is in fact an $O(2) \times O(2)$-space, where the action by the first factor comes from the dihedral structure and the action by the second factor comes from the $O(2)$-action on $X$. We are interested in $\THR(A,D)$ with the diagonal $O(2)$-action.

\subsection{Fixed points}
Let $\triangle: O(2) \to O(2) \times O(2)$ be the diagonal map.  We wish to study the $O(2)/C_r$-space $\left(\triangle^*\THR(A, D; X)\right)^{C_{r}}$. The image $\triangle(C_r)$ is not normal in $O(2) \times O(2)$, but it is normal in $\triangle(O(2))$, hence we consider $\THR(A, D; X)^{\triangle(C_r)}$ as an $\triangle(O(2))/\triangle(C_r)$-space. 

Let $C_r \mathcal{F}$ denote the full subcategory of $\mathcal{F}$ with objects $C_r\cdot F$ for $F\in \Ob(\mathcal{F})$. This subcategory is both cofinal and stable under the group action, and therefore the inclusion $i: C_r \mathcal{F} \to \mathcal{F} $ induces a bijection of $\triangle(O(2))/\triangle(C_r)$-sets:
\[ \Big(\colim_{ \mathcal{F}}  \THR(A,D;X)[F] \Big)^{\triangle(C_r)} \xleftarrow{\ind_i}\Big(\colim_{ C_r \mathcal{F}}  \THR(A,D;X)[C_r \cdot F]\Big)^{\triangle(C_r)} .\]
This is in fact a homeomorphism, since the corresponding map on classical geometric realizations is the homeomorphism from the realization of the $r$-subdivided simplicial space to the realization of the simplicial space itself. Since the $C_r$-action on the indexing category is trivial, the inclusion 
\[
 \THR(A,D;X)[C_r \cdot F]^{\triangle(C_r)}  \hookrightarrow  \THR(A,D;X)[C_r \cdot F]
\] 
induces an $\triangle(O(2))/\triangle(C_r)$-bijection
\[
\Big(\colim_{ C_r \mathcal{F}}  \THR(A,D;X)[C_r \cdot F]\Big)^{\triangle(C_r)} 
\xleftarrow{\sim} \colim_{ C_r \mathcal{F}} \THR(A,D;X)[C_r \cdot F]^{\triangle(C_r)}.
\]
This is likewise a homeomorphism, since it corresponds to commuting the classical geometric realization and $C_r$-fixed points of a $C_r$-simplicial space, which commutes by \cite[Corollary 11.6]{May72}.  

Note that
\[
 \THR(A,D;X)[C_r \cdot F]=\hocolim_{I^{C_r\cdot F}} G^{C_r \cdot F}_{X}=\Big\vert [k] \mapsto \bigvee_{\underline{i}_0 \to \cdots \to \underline{i}_k} G^{C_r \cdot F}_{X} (\underline{i}_0)\Big\vert.
 \]
 We want to describe the $\triangle(C_r) $-simplicial action. Let $t_r:=e^{\frac{2\pi i }{r}} \in \mathbb{T}$ and recall that we constructed a functor $(t_r)_{C_r \cdot F} :I^{C_r \cdot F} \xrightarrow{} I^{C_r \cdot F}$ and a natural transformation 
\[t'_{C_r \cdot F}: G_X^{C_r \cdot F} \Rightarrow G_X^{C_r \cdot F} \circ t_{C_r \cdot F}. \]
The $t_r$-action on $X$ gives rise to a natural transformation 
\[
X_{t_r}: G_{X}^{C_r \cdot F} \circ (t_r)_{C_r \cdot F}  \Rightarrow G_{X}^{C_r \cdot F} \circ  (t_r)_{C_r \cdot F} .
\] 
The $\triangle(C_r) $-simplicial action is generated by the operator which takes the summand indexed by $\underline{i}_0 \to \cdots \to \underline{i}_k$ to the one indexed by $(t_r)_{C_r \cdot F} (\underline{i}_0) \to \cdots \to (t_r)_{C_r \cdot F}(\underline{i}_k) $ via 
\begin{align}\label{ct1}
G_{X}^{C_r \cdot F} (\underline{i}_0) \xrightarrow{(t_r')_{C_r \cdot F} } G_{X}^{C_r \cdot F} \circ (t_r)_{C_r \cdot F} (\underline{i}_0) \xrightarrow{X_{t_r}}G_{X}^{C_r \cdot F}\circ (t_r)_{C_r \cdot F} (\underline{i}_0).
\end{align}
If a $k$-simplex is fixed, then it must belong to a wedge summand whose index consists of objects and morphisms in $I^{C_r \cdot F}$ which are fixed by $(t_r)_{C_r \cdot F} $. This is exactly the image of the diagonal functor
\[\triangle_r: I^{C_r \cdot F/C_r} \to I^{C_r \cdot F}, \quad (i_{\overline{z}})_{\overline{z}\in C_r \cdot F/C_r} \mapsto (i_{\overline{z}})_{z\in C_r\cdot F},\] 
which is defined similarly on morphism and where $\overline{z}$ denotes the orbit of $z\in C_r \cdot F$. The natural transformation \eqref{ct1} restricts to a natural transformation from $G_{X}^{C_r \cdot F}\circ \triangle_r$ to itself, hence $C_{r}$ acts on $G_{X}^{C_r \cdot F}\circ \triangle_r$ through natural transformations. Since geometric realization commutes with finite limits by \cite[Corollary 11.6]{May72}, we obtain the following lemma:

\begin{lemma}
The canonical map induces a homeomorphism of non-equivariant spaces: 
\[ \hocolim_{I^{C_r \cdot F/C_r}} \left(G_{X}^{C_r \cdot F} \circ \triangle_r\right)^{C_{r}} \xrightarrow{\cong } \left(\underset{I^{C_r \cdot F}} {\hocolim} \ G_{X}^{C_r \cdot F} \right)^{C_{r}}.\]
Furthermore, the maps assemble into an isomorphism of $\triangle(O(2))/\triangle(C_r)$-diagrams $C_r\mathcal{F} \to \mathsf{Top}_*$.

\end{lemma}
Next, we consider the non-equivariant fixed point space $\THR(A, D; X)^{\triangle(D_{r})}$. Let $D_{r} \mathcal{F}$ denote the full subcategory with objects $D_{r}\cdot F$ for $F\in \Ob(\mathcal{F})$. For convenience we restrict further to the cofinal subcategory $D_{r}\mathcal{F}_*$ consisting of objects of $D_r\mathcal{F}$ containing $1$ and $ t_{2r}:= e^{\frac{2 \pi i}{2r}}$. There is a canonical homeomorphism constructed as above
\[
\Big(\colim_{ \mathcal{F}}  \THR(A,D;X)[ F]\Big)^{\triangle(D_{r})} \xleftarrow{\sim} \colim_{ D_{r} \mathcal{F}_*}  \THR(A,D;X)[D_r \cdot F]^{\triangle(D_{r})}.
\]
Note that
\[
\THR(A,D;X)[D_r \cdot F]=\hocolim_{I^{D_{r}\cdot F}} G^{D_{r} \cdot F}_{X}=\Big\vert [k] \mapsto \bigvee_{\underline{i}_0 \to \cdots \to \underline{i}_k} G^{D_{r}\cdot F}_{X} (\underline{i}_0)\Big\vert .
\] 
The $\triangle(D_{r}) $-action is generated by the simplicial operator which takes the summand indexed by $\underline{i}_0 \to \cdots \to \underline{i}_k$ to the one indexed by $w_{D_r \cdot F}(\underline{i}_0) \to \cdots \to w_{D_r \cdot F}(\underline{i}_k)$ via 
\begin{align} \label{t1}
G_{X}^{D_{r}\cdot F} (\underline{i}_0) \xrightarrow{w'_{D_r \cdot F}} G_{X}^{D_{r}\cdot F} \circ w_{D_r \cdot F}(\underline{i}_0) \xrightarrow{X_\omega} G_{X}^{D_{r}\cdot F}\circ  w_{D_r \cdot F} (\underline{i}_0).
\end{align}
and the simplicial operator which takes the summand indexed by $\underline{i}_0 \to \cdots \to \underline{i}_k$ to the one indexed by $(t_r)_{D_r \cdot F}(\underline{i}_0) \to \cdots \to (t_r)_{D_r \cdot F}(\underline{i}_k) $ via 
\begin{align} \label{t2}
G_{X}^{D_{r}\cdot F} (\underline{i}_0) \xrightarrow{(t'_r)_{D_r \cdot F}} G_{X}^{D_{r}\cdot F} \circ (t_r)_{D_r \cdot F}(\underline{i}_0) \xrightarrow{X_{t_r}}G_{X}^{D_{r}\cdot F}\circ (t_r)_{D_r \cdot F}(\underline{i}_0). 
\end{align}
If a $k$-simplex is fixed, then it must belong to a wedge summand whose index consists of objects and morphisms in $I^{D_r \cdot F}$ which are fixed by the $D_r$-action. In order to describe such objects and morphisms, we note that a fundamental domain for the $D_r$-action on $\mathbb{T}$ is given by 
\[
F_{D_r}=\{e^{2\pi i t} \mid  0\leq t \leq 1/2r \}\subset \mathbb{T}
\]
Let $z$ be in the interior of the circle arc $F_{D_r}$. Then $
\overline{z}:= D_r \cdot z= C_r \cdot z \amalg C_r \cdot \omega(z)$ while the orbits of the endpoints $1$ and $t_{2r}$ are of the form
$\overline{1}= C_r \cdot 1=C_r \cdot \omega(1)$ and $\overline{t_{2r}}= C_r \cdot t_{2r}=C_r \cdot \omega(t_{2r}) $.

The objects of $ I^{D_{r} \cdot F}$ are permuted by the generators, while the morphisms are permuted by $t_r$ and both permuted and conjugated by $\omega$.  Let $I^{G}$ denote the subcategory of $I$ with the same objects and all morphisms $\alpha$ which satisfies $\alpha^\omega=\alpha$. Define a functor 
\begin{align*}
\triangle_r^e : (I^{G})^{\{\overline{1}, \overline{t_{2r}}\}} \times I^{D_r F/D_r\setminus \{\overline{1},\overline{t_{2r}} \} } \to I^{D_{r}\cdot F},
\end{align*}
on objects by 
\[
\triangle_r^e((i_{\overline{z}})_{\overline{z}\in D_r \cdot F/D_r})=(i_{\overline{z}})_{z\in D_r\cdot F}
\]
and on morphisms by
\[
\triangle_r^e((\alpha_{\overline{z}})_{\overline{z}\in D_r \cdot F/D_r}) = (\alpha_{z})_{z\in D_r\cdot F},
\]
where
\[  \alpha_{z}= \left\{\begin{array}{rl}
\alpha_{\overline{y}} &\mbox{ if $z\in C_r \cdot y$} \\
( \alpha_{\overline{y}})^{\omega} &\mbox{ if $z \in C_r \cdot \omega(y),$}
       \end{array} \right.\] 
              and $y \in F_{D_r}.$ 
Then a fixed $k$-simplex must have an index which consist of elements and morphisms in the image of $\triangle_r^e$. 

To clarify the ``diagonal'' functor $\triangle_r^e$, we draw an example of the functor 
\[
\triangle_3^e: (I^{G})^{\{\overline{1}, \overline{t_{6}}\}} \times I^{D_3F/D_3\setminus \{\overline{1},\overline{t_{6}} \} } \to I^{D_{3}\cdot F}
\]  
on a tuple of morphisms. In the picture below $D_3\cdot F$ is the subset consisting of all the points on the right hand circle, and the points of the left hand circle are orbit representatives.
              
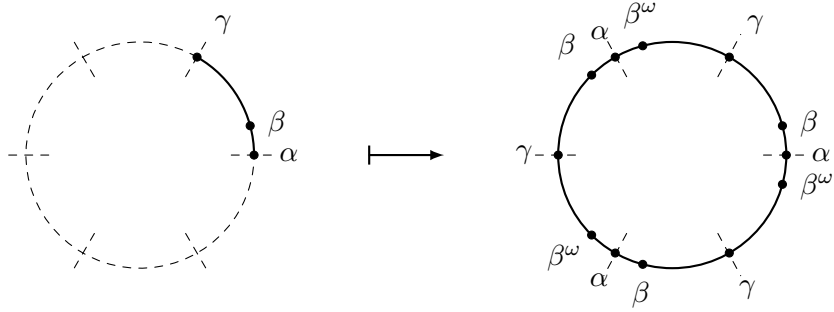
\begin{center}
    \begin{tikzpicture}[scale=1,cap=round,>=latex]
        \draw[thick] (1.5cm,0cm) arc (0:60:1.5cm) ;
             \draw[dashed] (60:1.5cm) arc (60:360:1.5cm) ;

        \foreach \x in {0,15,60} {
            % dots at each point
                \filldraw[black] (\x:1.5cm) circle(1.5pt);   }

        \draw[thick,|->] (3cm,0cm) --  (4cm,0cm);
        
          \draw[dashed] (1.2cm,0cm) --  (1.8cm,0cm);
            \draw[dashed] (60:1.2cm) --  (60:1.8cm);
                   \draw[dashed] (120:1.2cm) --  (120:1.8cm);
                                      \draw[dashed] (180:1.2cm) --  (180:1.8cm);
                   \draw[dashed] (240:1.2cm) --  (240:1.8cm);
                                      \draw[dashed] (300:1.2cm) --  (300:1.8cm);
        % draw the unit circle
        \draw[thick] (7cm,0cm) circle(1.5cm);

        \foreach \y in {0,15,60, 120, 135, 180, 240, 255, 300, 105, 345, 225 } {
            % dots at each point
                \filldraw[black](7cm,0cm)++ (\y:1.5cm) circle(1.5pt);   }

                          \draw[dashed] (8.2cm,0cm) --  (8.8cm,0cm);
                 \draw[dashed] (7cm,0cm)++ (120:1.2cm)--(6.1cm,1.6cm);
                       \draw[dashed] (7cm,0cm)++ (240:1.2cm)--(6.1cm,-1.6cm);
                       
                                           \draw[dashed] (5.2cm,0cm) --  (5.8cm,0cm);
                 \draw[dashed] (7cm,0cm)++ (60:1.2cm)--(7.9cm,1.6cm);
                       \draw[dashed] (7cm,0cm)++ (300:1.2cm)--(7.9cm,-1.6cm);

        % the placement is better this way
            \draw (0:1.7cm) node[right] {$\alpha$};
                \draw (15:1.6cm) node[right] {$\beta$};
                        \draw (60:1.7cm) node[above right] {$\gamma$};
                        
                        \draw (7cm,0cm)++(0:1.7cm) node[right] {$\alpha$};
                \draw (7cm,0cm)++(15:1.6cm) node[right] {$\beta$};
                        \draw (7cm,0cm)++(60:1.7cm) node[above right] {$\gamma$};

          \draw (7cm,0cm)++(124:1.7cm) node[above] {$\alpha$};
                \draw (7cm,0cm)++(135:1.6cm) node[above left] {$\beta$};
                        \draw (7cm,0cm)++(180:1.7cm) node[left] {$\gamma$};
                        
                                  \draw (7cm,0cm)++(244:1.6cm) node[below left] {$\alpha$};
                \draw (7cm,0cm)++(255:1.6cm) node[below] {$\beta$};
                        \draw (7cm,0cm)++(296:1.7cm) node[below right] {$\gamma$};
                        
                                  \draw (7cm,0cm)++(105:1.6cm) node[above] {$\beta^{\omega}$};
                \draw (7cm,0cm)++(345:1.6cm) node[right] {$\beta^{\omega}$};
                        \draw (7cm,0cm)++(225:1.5cm) node[below left] {$\beta^{\omega}$};

    \end{tikzpicture}
        \end{center}
    
        The natural transformations \eqref{t1} and \eqref{t2} restrict to transformations from the functor $G_{X}^{D_{r}\cdot F}\circ \triangle^e_r$ to itself, hence $D_r$ acts on $G_{X}^{D_{r}\cdot F}\circ \triangle^e_r$ trough natural transformation. To ease notation we set $I^{{D_r F/D_r^-} }:=I^{D_r F/D_r\setminus \{\overline{1},\overline{t_{2r}} \} }$. Likewise, we omit the over-lines on the orbits $1$ and $t_{2r}$. We obtain the following lemma: 

\begin{lemma}\label{subcat}
The canonical map induces a homeomorphism: 
\[\hocolim_{ (I^{G})^{\{1, t_{2r}\}} \times I^{{D_r F/D_r^-} } } (G_{X}^{D_{r}\cdot F}\circ \triangle_r^e)^{D_{r}} \xrightarrow{\cong} \left(\hocolim_{I^{D_{r}\cdot F}} G_{X}^{D_{r}\cdot F} \right)^{D_{r}} .\]
Furthermore, the maps assemble into an isomorphism of functors $D_r\mathcal{F}_* \to \Top_*$. 
\end{lemma}

\begin{remark} \label{action}
We describe the $D_{r}$-action on 
\begin{align*}
G_{X}^{D_{r}\cdot F} \circ \triangle^e_r(i) =
\Map\Big(\bigwedge_{z\in D_{r}\cdot F} S^{i_{\overline{z}}} , \bigwedge_{z\in D_{r}\cdot F } A_{i_{\overline{z}}} \wedge X \Big) 
\end{align*}
at an object $i \in \ob\big( (I^{G})^{\{1, t_{2r}\}}\times I^{D_r\cdot F /D_r^-}\big)$ arising from the natural transformations (\ref{t1}) and (\ref{t2}). The spaces
\[\bigwedge_{z\in D_{r}\cdot F} S^{i_{\overline{z}}} \quad \text{ and  \quad} \bigwedge_{z\in D_{r}\cdot F } A_{i_{\overline{z}}} \wedge X  \]
are themselves $D_{r}$-spaces. On the left hand side, the generator $t_r$ permutes the smash factors and the generator $\omega$ permutes the smash factors, then acts by $w_{i} \in \Sigma_{i}$ factor-wise. On the right hand side the generator $t_r$ permutes the smash factors and acts on $X$, and the generator $\omega$ permutes the smash factors, then acts by $w_{i} \in \Sigma_{i}$ and the anti-involution $D$ factor-wise and acts on $X$. The $D_{r}$-action on the mapping space is by conjugation. 
\end{remark}

\subsection{Equivariant Approximation Lemma}
In this section we consider the functor $G^F:=G^F_{S^0}$ for simplicity. The approximation lemma proven below also holds if $S^0$ is replaced by $S^V$ for any finite dimensional real orthogonal $O(2)$-representation $V$. Bökstedt's Approximation Lemma states that the canonical inclusion 
\[G^{F}(i) \hookrightarrow \hocolim_{I^{F}} G^{F}\] 
can be made as connected as desired by choosing $i$ coordinate-wise big enough under some connectivity conditions on $A$. Dotto proves in \cite{Do} a $G$-equivariant version of the approximation lemma under some restricted connectivity conditions on $(A,D)$. We use the same connectivity conditions and extend the result to a $D_r$-equivariant version. For an integer $n$, we let $\lceil \frac{n}{2} \rceil$ denote the ceiling of $\frac{n}{2}$. Throughout this rest of this paper we make the following connectivity assumptions on $(A,D)$, which hold for $(\mathbb{S}[\Gamma], \id[\Gamma])$:

\begin{ass}\label{concond} 
Let $(A,D)$ be a symmetric ring spectrum with anti-involution. We assume that $A_n$ is $(n-1)$-connected and that $\left(A_n\right)^{D_n \circ \omega_n}$ is $\left(\lceil \frac{n}{2} \rceil -1 \right)$-connected. Furthermore we assume that there exists a constant $\epsilon \geq 0$, such that the structure map $\lambda_{n,m}: A_n \wedge S^m \to A_{n+m}$ is $(2n+m- \epsilon)$-connected and the restriction of the structure map $\lambda_{n,m}: A_n^{D_n \circ \omega_n} \wedge \left(S^m\right)^{\omega_m} \to \left(A_{n+m}\right)^{D_{n+m}\circ (\omega_n \times \omega_m)}$ is $(n+\lceil \frac{m}{2} \rceil -\epsilon)$-connected.
\end{ass}

Before we state the Equivariant Approximation Lemma, we introduce some notation. Let $f: Z \to Y$ be a map of pointed left $H$-spaces, where $H$ denotes a compact Lie group. We call $f$ $n$-connected respectively a weak-$H$-equivalence if $f^K: Z^K \to Y^K$ is $n$-connected respectively a weak equivalence for all closed subgroups $K\leq H$. For a natural number $N\in \mathbb{N}$ and $i\in \ob( I^F)$ we say that $i \geq N$ if $i_z \geq N$ for all $z \in F$. We define a partial order on the set $\ob( I^F)$ by declaring $(i)_{z\in F} \leq (j)_{z\in F}$ if $i_z\leq j_z$ for all $z\in F$. Given a partially ordered set $J$ we say that a map $\lambda: J \to \mathbb{N}$ tends to infinity on $J$ if for all $N\in \mathbb{N}$ there is a $j_N\in J$ such that $\lambda(j)\geq N$ for all $j\geq j_N.$ 

\begin{prop}[Equivariant Approximation Lemma] \label{EAL} 
Let $(A,D)$ be a symmetric ring spectrum with anti-involution satisfying Assumptions \ref{concond}. Let $F$ be a finite subset of the circle. For part \eqref{part 2} assume further that $1,t_{2r} \in D_r \cdot F$. For part \eqref{part 3} assume further that $1,t_{4r}\in D_{2r} \cdot F$. Then the following holds:
\begin{enumerate}[i]

\item Given $n\geq 0$, there exists $N \geq 0$ such that the $C_r$-equivariant inclusion 
\[
G^{C_r \cdot F} \circ \triangle_r (i) \hookrightarrow \hocolim_{I^{C_r \cdot F}} G^{C_r \cdot F}
\]
is $n$-connected for all $i \in \ob(I^{C_r \cdot F/C_r})$ such that $i \geq N$. \label{part 1}

\item Let $r$ be odd. Given $n\geq 0$, there exists $ N\geq 0$ such that the $D_{r}$-equivariant inclusion 
\[
G^{D_{r} \cdot F} \circ \triangle^e_r (i) \hookrightarrow \hocolim_{I^{D_{r} \cdot F}} G^{D_{r} \cdot F}
\]
is $n$-connected for all $i \in \ob\big( (I^{G})^{\{1, t_{2r}\}}\times I^{D_r \cdot F /D_r^- }\big) $ such that $i \geq N$.\label{part 2}

\item Let $r$ be even. Given $n\geq 0$, there exists $ N\geq 0$ such that the $D_{2r}$-equivariant inclusion 
\[
G^{D_{2r} \cdot F} \circ \triangle^e_{2r} (i) \hookrightarrow \hocolim_{I^{D_{2r} \cdot F}} G^{D_{2r} \cdot F}
\]
is $n$-connected for all $i \in \ob\big( (I^{G})^{\{1, t_{4r}\}}\times I^{D_{2r} \cdot F /D_{2r}^- }\big) $ such that $i \geq N$, when considered as a map of $D_r$-spaces.\label{part 3}
\end{enumerate}
\end{prop}

The result depends on the categories $I^F$ and $ I^G$ being good indexing categories and on the connectivity of the functor $G^F$. Part \eqref{part 1} is proven in \cite{BHM}. Part \eqref{part 2} is proven in the case $r=1$ in \cite[Prop. 4.3.2]{Do}. We prove part \eqref{part 2} and \eqref{part 3} below.

\begin{definition}
A good indexing category is a triple $(J, \overline{J}, \mu)$ where $J$ is a small category, $\overline{J} \subset J$ is a full subcategory and $\mu=\{\mu_j: J \to J\}_{j\in \ob(\overline{J})}$ is a family of functors. The data is required to satisfy that for every $j \in \ob( \overline{J})$, there exists a natural transformation $U: \id \Rightarrow \mu_j$ such that $\mu_j \circ U = U \circ \mu_j$. 
\end{definition}
Let $I_{\ev}$ denote the full subcategory of $I$ with objects the even non-negative integers. For $j\in \ob(I_{\ev})$ let $\mu_j: I \to I$ denote the functor 
\[\mu_j(i)= \frac{j}{2} + i + \frac{j}{2}, \quad \quad \mu_j(\alpha)= \id_{\frac{j}{2}} + \alpha + \id_{\frac{j}{2}}. \]
There is a natural transformation $U: \id \Rightarrow \mu_j$ defined by the middle inclusion. The triple $(I,I_{\ev}, \mu)$ is a good indexing category and the structure restricts to $I^{G}$.  In \cite[Lemma 4.3.8]{Do} the following lemma is proved:

\begin{lemma}\label{oi}
Let $(J, \overline{J}, \mu)$ be a good indexing category with initial object $0 \in \ob(J)$ such that $\mu_j(0)=j$ for all $j \in \ob(\overline{J})$. Let $X: J \to \Top_*$ be a functor. If for  $j\in \ob(\overline{J})$ 
\[X(j)=X(\mu_j(0)) \to X(\mu_j(i)),\] 
induced by $0 \to i$ is $n$-connected for all $i \in \ob(J)$, then the canonical map 
\[
X(j) \to \hocolim_{J} X
\]
is $n$-connected. 
\end{lemma}

\begin{remark}
If $r$ is odd, then the collection $\{ C_s,D_{s}\}_{s \mid r}$ represents all conjugacy classes of subgroups in $D_{r}$. If $r$ is even, then $D_{r}$ contains 2 conjugacy classes of dihedral subgroups of order $2s$, represented by
\[D_{s}=\left\langle t_s, \omega \right \rangle, \quad D'_{s}=\left\langle t_s,  t_r\omega \right\rangle.\]
Hence the collection $\{ C_s,D_{s}, D'_s\}_{s \mid r}$ represents all conjugacy classes of subgroups in $D_{r}$, when $r$ is even. The subgroups $D_{s}$ and $D'_{s}$ beocmes conjugate when considered as subgroups of the bigger group $D_{2r}$.
\end{remark}

\begin{proof}[Proof of part \eqref{part 2}] 
We let $j \in  \ob\big((I^{G})^{\{1, t_{2r}\}}\times I^{D_r F /D_r^- }\big) $ and consider the inclusion
\[\iota_j:G^{D_{r} \cdot F} \circ \triangle^e_r (j) \hookrightarrow \hocolim_{I^{D_{r} \cdot F}} G^{D_{r} \cdot F}.\]
We must show that the connectivity of the induced map on $H$-fixed points tends to infinity with $j$ for $H \in \{ C_s,D_{s}\}_{s \mid r}$. It suffices by part \eqref{part 1} to prove that the connectivity on $D_{r}$-fixed points tends to infinity with $j$. Indeed, if $s\cdot t=r$, then $D_{r}\cdot F=D_{s}\cdot F'$ and $ D_{r}\cdot F=C_s\cdot F''$ for finite subsets $F', F'' \subset \mathbb{T}$, and
\begin{align*}
G^{D_{r}\cdot F}\circ \triangle^e_r (j) =G^{D_{s}\cdot F'} \circ \triangle_s( j') , \quad
G^{D_{r}\cdot F}\circ \triangle^e_r (j) =G^{C_s\cdot F''} \circ \triangle_s( j'').
\end{align*} 
for some $j' \in \ob\big((I^G)^{\{1, t_{2s}\}}\times I^{D_s \cdot F' /D_s^- }\big) $ and $j'' \in \ob(I^{C_s \cdot F''/C_s})$ both of which tend to infinity with $j$. 

By Lemma \ref{subcat} the restriction of $\iota_j$ to $D_{r}$-fixed points is equal to the inclusion
\[ \left( G^{D_{r} \cdot F} \circ \triangle^e_r (j) \right)^{D_{r}} \hookrightarrow \hocolim_{ (I^{G})^{\{1, t_{2r}\}}\times I^{D_rF/D_r^- }} \left( G^{D_{r} \cdot F} \circ \triangle^e_r \right)^{D_{r}}. \] 
Assume first that $j$ is coordinate-wise even. We let $i\in \ob\big((I^{G})^{\{1, t_{2r}\}}\times I^{D_r\cdot F/D_r^-} \big)$ and define the map $\Lambda$ to be the composite

\[  
\bigwedge_{z\in D_{r}\cdot F } A_{j_{\overline{z}}} \wedge  \bigwedge_{z\in D_{r}\cdot F } S^{i_{\overline{z}}} \xrightarrow{\cong} \bigwedge_{z\in D_{r}\cdot F } A_{j_{\overline{z}}}\wedge S^{i_{\overline{z}}} \xrightarrow{\underset{z\in D_{r}\cdot F }{\bigwedge} \lambda_{j_{\overline{z}},i_{\overline{z}}}} \bigwedge_{z\in D_{r}\cdot F } A_{i_{\overline{z}}+j_{\overline{z}}}.
 \]
If we give the domain the diagonal $D_r$-action and we let $D_{r}$ act on the target as described in Remark \ref{action} but using the permutations $\omega_{j_{\overline{z}}} \times \omega_{i_{\overline{z}}}$ instead of $\omega_{j_{\overline{z}}+i_{\overline{z}}}$, then $\Lambda$ is $D_{r}$-equivariant. The $D_{r}$-map  $G^{D_{r} \cdot F}\circ \triangle_r^e (\mu_j(0)) \to G^{D_{r} \cdot F}\circ \triangle_r^e (\mu_j(i))$ induced by the morphisms $0\to i$ is equal to the composite: 
\[
	\begin{tikzpicture}[x=0cm,y=2cm]
	\node(a) at (0,3) {$\displaystyle\Map\Big(\bigwedge_{z\in D_{r}\cdot F } S^{j_{\overline{z}}},\bigwedge_{z\in D_{r}\cdot F } A_{j_{\overline{z}}}\Big) $};
	\node(b) at (0,2) {$\displaystyle\Map\Big(\bigwedge_{z\in D_{r}\cdot F } S^{j_{\overline{z}}}\wedge \bigwedge_{z\in D_{r}\cdot F } S^{i_{\overline{z}}},\bigwedge_{z\in D_{r}\cdot F } A_{j_{\overline{z}}} \wedge \bigwedge_{z\in D_{r}\cdot F } S^{i_{\overline{z}}} \Big)$};
	\node(c) at (0,1) {$\displaystyle\Map\Big(\bigwedge_{z\in D_{r}\cdot F } S^{j_{\overline{z}}+i_{\overline{z}}} ,\bigwedge_{z\in D_{r}\cdot F } A_{j_{\overline{z}}+i_{\overline{z}}} \Big) $};
	\node(d) at (0,0){$\displaystyle\Map\Big(\bigwedge_{z\in D_{r}\cdot F } S^{j_{\overline{z}}+i_{\overline{z}}} ,\bigwedge_{z\in D_{r}\cdot F } A_{j_{\overline{z}}+i_{\overline{z}}} \Big),$};

	\draw[->](a) to node [right]{$\scriptstyle(-)\wedge (\wedge_{z\in D_{r}\cdot F }S^{i_{{\overline{z}}}})$} (b);

	\draw[->](b) to node [right]{$\scriptstyle\Lambda^*$} (c);

	\draw[->](c) to node [right]{$\scriptstyle\Map(\wedge \alpha_{\overline{z}}^{-1},  \wedge \alpha_{\overline{z}})$} (d);

	\end{tikzpicture}\]
where $\alpha_{\overline{z}}=id_{j_{{\overline{z}}}/2} \times \xi_{j_{{\overline{z}}}/2,i_{{\overline{z}}}}\in \Sigma_{j_{{\overline{z}}}+i_{{\overline{z}}}}$. The first map is induced by the adjunction unit
\[\eta: \bigwedge_{z\in D_{r}\cdot F } A_{j_{\overline{z}}} \to \Map\Big(\bigwedge_{z\in D_{r}\cdot F }S^{i_{{\overline{z}}}}, \bigwedge_{z\in D_{r}\cdot F }S^{i_{{\overline{z}}}}\wedge \bigwedge_{z\in D_{r}\cdot F } A_{j_{\overline{z}}} \Big),\]
followed by the adjunction homeomorphism and a twist homeomorphism. We use the Equivariant Suspension Theorem \ref{suspension} to estimate its connectivity. If $s \mid r$ and we set $A:=\bigwedge_{z\in D_{r}\cdot F } A_{j_{\overline{z}}}$, then we have homeomorphisms induced by the appropriate diagonal maps and connectivity estimates following from Assumptions \ref{concond}:
 \begin{align*}
 A^{C_s} &\cong \Big(A_{j_1} \wedge A_{j_{t_{2r}}} \wedge \Big(\bigwedge_{\overline{z}\in D_r F/D_r^-} A_{j_{\overline{z}}}\Big)^{\wedge 2} \Big)^{\wedge r/s}, \\
 \conn(A^{C_s})& \geq \frac{r}{s}\Big(j_1+j_{t_{2r}} +  \sum_{\overline{z}\in D_r F/D_r^-} 2j_{\overline{z}}\Big)-1, \\
 A^{D_s} &\cong A_{j_1}^{D\circ \omega} \wedge A_{j_{t_{2r}}}^{D \circ \omega} \wedge \left(A_{j_1} \wedge A_{j_{t_{2r}}} \right)^{\wedge \frac{r-s}{2s}} \wedge\Big(\bigwedge_{\overline{z}\in D_r F/D_r^-} A_{j_{\overline{z}}}\Big)^{\wedge \frac{r}{s}}, \\
\conn(A^{D_s} ) &\geq \frac{r}{2s} \Big(j_1+j_{t_{2r}} +\sum_{\overline{z}\in D_r F/D_r^-} 2j_{\overline{z}} \Big) -1.
\end{align*}
By the Equivariant Suspension Theorem \ref{suspension} and Lemma \ref{A2} wee see that
\begin{align*}
&\conn(\eta^{C_s}) \geq  \frac{2r}{s}\Big(j_1+j_{t_{2r}} +  \sum_{\overline{z}\in D_r F/D_r^-} 2j_{\overline{z}}\Big)-1,\\ 
&\conn(\eta^{D_{s}})\geq  \frac{r}{s} \Big(j_1+j_{t_{2r}} + \sum_{\overline{z}\in D_r F/D_r^-} 2j_{\overline{z}}\Big) -1, \\
&\conn((\eta^*)^{D_{r}})\geq \frac{j_1}{2} +\frac{j_{t_{2r}}}{2}  + \sum_{\overline{z}\in D_rF/D_r^-} j_{\overline{z}}.
\end{align*} 
The connectivity of $\Lambda$ on fixed points follows from Assumptions \ref{concond}:
\begin{align*}
\conn(\Lambda^{C_s}) &\geq \min(j) + \frac{r}{s}\Big(j_1+i_1+j_{t_{2r}}+i_{t_{2r}} +  \sum_{\overline{z} \in D_rF/D_r ^-} 2(j_{\overline{z}}+i_{\overline{z}}) \Big) - 1 - \epsilon, \\
\conn(\Lambda^{D_{s}}) &\geq  \text{min}'(j)+ \frac{j_1}{2}+\Big\lceil\frac{i_1}{2}\Big\rceil+\frac{j_{t_{2r}}}{2}+\Big\lceil\frac{i_{t_{2r}}}{2}\Big\rceil + \frac{r-s}{2s}\left(j_1+i_1+j_{t_{2r}}+i_{t_{2r}}\right) \\
&+ \frac{r}{s}\Big(   \sum_{\overline{z} \in D_rF/D_r^- } j_{\overline{z}} +i_{\overline{z}} \Big) -1- \epsilon,
 \end{align*}
where $\min(j)= \min \{ j_{\overline{z}} \mid \overline{z} \in D_rF/D_r \}$ and $\min'(j)= \min \{ \frac{j_1}{2} , \frac{j_{t_{2r}}}{2}, j_{\overline{z}} \mid \overline{z} \in D_rF/D_r \}$. By Lemma \ref{A2} $ \conn((\Lambda^*)^{D_{r}})\geq \min'(j)  - 1 -\epsilon$. Hence by Lemma \ref{oi} there exists $\overline{N}\in \mathbb{N}$ such that the map
\[ \left( G^{D_r \cdot F} \circ \triangle^e_r (j) \right)^{D_{r}} \hookrightarrow  \hocolim_{(I^{G})^{\{1, t_{2r}\} }\times I^{D_r \cdot F/D_r^- } } \left( G^{D_r\cdot F} \circ \triangle^e_r \right)^{D_{r}}\] 
is $n$-connected if $j$ is coordinate-wise even and bigger than $\overline{N}$. 

Let $N:= \overline{N}+1$. Finally let $i \in \ob\big((I^{G})^{\{1, t_{2r}\}}\times I^{D_r \cdot F/D_r^-}\big)$ and assume that $i \geq N$. If $i$ is coordinate-wise even, then we have already seen that the inclusion into the homotopy colimit is $n$-connected. If $i$ is not coordinate-wise even then write $i=i' + j'$, where $i'$ is the coordinate-wise even element given by subtracting 1 from all the odd coordinates of $i$. The unique map $0 \to j'$ induces the horizontal map in the homotopy commutative diagram below: 
\begin{center}
\begin{tikzpicture}[x=3cm,y=2cm]
\node (a) at (0,1) {$\left( G^{D_r F} \circ \triangle^e_r (i') \right)^{D_{r}}$};
\node (c) at (2,1) {$\left( G^{D_r F} \circ \triangle^e_r (i) \right)^{D_{r}}$};
\node (b) at (1,0) {$\displaystyle \hocolim_{ (I^{G})^{\{1, t_{2r}\}}\times I^{D_rF/D_r^- } } \left( G^{D_r F} \circ \triangle^e_r \right)^{D_{r}}$};
\draw[->] (a) to (c);
\draw[->] (a) to (b);
\draw[->] (c) to (b);
\end{tikzpicture}
\end{center}
Since $i' \geq \overline{N}$, the horizontal map and the left diagonal map are $n$-connected, hence so is the right diagonal map as desired.
\end{proof}

\begin{proof}[Proof of part \eqref{part 3}] 
We let $j \in  \ob\big((I^{G})^{\{1, t_{4r}\}}\times I^{D_{2r} F /D_{2r}^- }\big) $ and consider the $D_{2r}$-equivariant inclusion
\[
G^{D_{2r} \cdot F} \circ \triangle^e_{2r} (j) \hookrightarrow \hocolim_{I^{D_{2r} \cdot F}} G^{D_{2r} \cdot F}.
\]
We want to show that the connectivity on $H$-fixed points tends to infinity with $j$ for $H\in \{ C_s,D_{s}, D'_s\}_{s \mid r}$. The subgroups $D_{s}$ and $D'_{s}$ are conjugate inside $D_{2r}$, hence we can reduce to check connectivity on the $H$-fixed points for $H \in \{ C_s,D_{s}\}_{s \mid r}$ and complete the proof as above. 
\end{proof}

Before we conclude this section, we state the following useful lemma, which can be found in \cite[Lemma 4.3.7]{Do}.
\begin{lemma}\label{infinity}
Let $(J, \overline{J}, \mu)$ be a good indexing category, let $X,Y: J \to \Top_*$ be functors, and let $\Phi:  X \Rightarrow Y$ be a natural transformation.  Suppose that for all $j \in \ob(J)$ the map $\Phi_j: X(j) \to Y(j)$ is $ \lambda(j)$-connected for a map $\lambda: \ob(J) \to \mathbb{N}$ that tends to infinity on $\ob(J)$. In this situation,  the induced map on homotopy colimits 
\[\Phi: \hocolim_{J} X \xrightarrow{} \hocolim_{J} Y \]
is a weak equivalence.
\end{lemma}

\section{The cyclotomic structure of $\THR(A,D)$}
The space $\THR(A,D;S^0)$ is the $0$th space of a fibrant orthogonal $O(2)$-spectrum in the model structure based on the family of finite subgroups of $O(2)$, see Proposition \ref{fibrant}, and furthermore the spectrum is cyclotomic. Before we establish these results, we briefly recall the category of equivariant orthogonal spectra and the fixed points functors.

We let $H$ denote a compact Lie group. By an $H$-representation, we will mean a finite dimensional real inner product space on which $H$ acts by linear isometries. We will work in the category of orthogonal $H$-spectra, defined as diagram $H$-spaces as in \cite[Chapter II.4]{MM}. Let $(H\Top_*, \wedge, S^0)$ denote the symmetric monoidal category of based $H$-spaces and continuous based $H$-equivariant maps. The collection of all $H$-spaces together with all based maps gives rise to a category enriched in $(H\Top_*, \wedge, S^0)$, which we denote $\mathscr{T}_{H}$.

We fix a complete $H$-universe $\mathcal{U}$. If $V$ and $W$ are $H$-representations in $\mathcal{U}$, then $L(V,W)$ denotes the $H$-space of linear isometries from $V$ to $W$ with $H$-action by conjugation. The $H$-bundle $E(V,W)$ is the sub-bundle of the product $H$-bundle $L(V,W) \times W \to W$ consisting of those pairs $(\alpha, w)$ such that $w\in W-\alpha(V)$. Let $\mathscr{J}_H(V,W)$ be the Thom $H$-space of $E(V,W)$. We define composition
\[\circ:  \mathscr{J}_H(V',V'') \times \mathscr{J}_H(V,V') \to \mathscr{J}_H(V,V'')\]
by $(\beta, y) \circ (\alpha,x )=(\beta \circ \alpha, \beta(x)+y)$. The point $(\id_V, 0) \in \mathscr{J}_G^{U}(V,V) $ is the identity morphism.  Let $\mathscr{J}_H^{\mathcal{U}}$ be the category enriched over $(H\Top_*, \wedge, S^0)$ with objects all finite dimensional $H$-representation $V \subset \mathcal{U}$ and morphisms the Thom $H$-spaces $\mathscr{J}_H(V,W)$.

\begin{definition}
An orthogonal $H$-spectrum indexed on $\mathcal{U}$ is an enriched functor 
\[X: \mathscr{J}_{H}^{\mathcal{U}} \to \mathscr{T}_{H}.\] 
A morphism of orthogonal $H$-spectra indexed on $\mathcal{U}$ is an enriched natural transformation. Let $H\Spe_{\mathcal{U}}$ denote the category of orthogonal $H$-spectra indexed on $\mathcal{U}$.
\end{definition}
Let $X$ be an orthogonal $H$-spectrum. For a closed subgroup $K \leq H$ and a non-negative integer $q$ the homotopy groups of $X$ are given as follows:
\begin{align*}
\pi_q^K(X)=\colim_{V \subset U} \pi_q^K(\Omega^V X(V)),  \quad \pi_{-q}^K(X)=\colim_{\mathbb{R}^q \subset V \subset U} \pi_0^K(\Omega^{V-\mathbb{R}^q} X(V)).
\end{align*}
The morphisms of $H$-spectra inducing isomorphism on all homotopy groups are referred to as $\pi_*$-isomorphisms. These are the weak equivalences in the stable model structure on orthogonal $H$-spectra indexed on $\mathcal{U}$ given in \cite[Chapter III, 4.1,4.2]{MM}. Throughout this paper, we let $j_f: X \to X_f$ denote a fibrant replacement functor in the stable model structure, such that $j_f$ is an acyclic cofibration, and we let $j^c: X^c \to X$ denote a cofibrant replacement functor in the stable model structure such that $j^c$ is an acyclic fibration. 

Let $\mathcal{S}$ be a family of subgroups of $H$, that is $\mathcal{S}$ is a collection of subgroups closed under taking subgroups and conjugates. We say that a morphism of othorgonal $H$-spectra indexed on $\mathcal{U}$ is an $\mathcal{S}$-equivalence if it induces isomorphisms on $\pi^K_q(-)$ for all subgroups $K\in \mathcal{S}$ and all $q \in \mathbb{Z}$. The $\mathcal{S}$-equivalences constitutes the weak equivalences in the $\mathcal{S}$-model structure, see \cite[Chapter IV.6]{MM}. 

\subsection{Pointset fixed point functors}
Let $K\leq H$ be a closed subgroup. If $\mathcal{U}$ is a complete $H$-universe, then $\mathcal{U}^K$ is a complete $N(K)/K$-universe. There are two fixed point functors which take an orthogonal $H$-spectrum indexed on $U$ and produce an orthogonal $N(K)/K$-spectrum indexed on $U^K$. We recall the fixed point functors in case of a normal subgroup $N$, see \cite[Chapter V.4]{MM} for details.

We define a category $\mathscr{J}_{H,N}^{U}$ with the same objects as $\mathscr{J}_{H}^{\mathcal{U}}$. The morphism spaces are the $H/N$-spaces of $N$-fixed points 
\[\mathscr{J}_{H,N}^{\mathcal{U}}(V,W):= \mathscr{J}_H^{\mathcal{U}}(V,W)^N.\] 
The composition and identity restrict appropriately making $\mathscr{J}_{H,N}^{U}$ into a category enriched over $(H/N\Top_*, \wedge, S^0)$. Taking $N$-fixed points levelwise takes an orthogonal $H$-spectrum $X$ and produces a functor enriched over $(H/N\Top_*, \wedge, S^0)$:
\[\Fix^N(X): \mathscr{J}_{H,N}^{\mathcal{U}} \to \mathscr{T}_{H/N}.\]     
We have two enriched functors comparing the categories $\mathscr{J}_{H,N}^{\mathcal{U}}$ and $\mathscr{J}_{H/N}^{\mathcal{U}^N}$:
\[\mathscr{J}_{H/N}^{U^N} \xrightarrow{\nu}\mathscr{J}_{H,N}^{\mathcal{U}}\xrightarrow{\phi} \mathscr{J}_{H/N}^{\mathcal{U}^N},\]
The functor $\nu$ takes an $H/N$-representation $V$ to the $H$-representation $q^*V$, where $q: H \to H/N$ is the quotient homomorphism. The functor $\phi$ sends a $H$-representation $V$ to the $H/N$-representation $V^N$, and it sends a morphism $(\alpha, x)$ to the fixpoints morphism $(\alpha^N,x)$. The functors  $\phi$ and $\nu$ induces forgetful functors
\begin{align*}
U_{\phi}&: \Fun_{H/N\Top_*}(\mathscr{J}_{H/N}^{\mathcal{U}^N}, \mathscr{T}_{H/N})  \to \Fun_{H/N\Top_*}(\mathscr{J}_{H,N}^{\mathcal{U}}, \mathscr{T}_{H/N}), \\
U_{\nu}&: \Fun_{H/N\Top_*}(\mathscr{J}_{H,N}^{\mathcal{U}},\mathscr{T}_{H/N})  \to \Fun_{H/N\Top_*}(\mathscr{J}_{H/N}^{\mathcal{U}^N},\mathscr{T}_{H/N}),
\end{align*}
where $\Fun_{H/N\Top_*}(-,-)$ is the category of functors enriched in $(H/N\Top_*, \wedge, S^0)$ and enriched natural tranformations.
Note that $\phi \circ \nu=\id$, hence $U_{\nu} \circ U_{\phi}=\id$. Enriched left Kan extension along $\phi$ gives an enriched functor
\[P_{\phi}: \Fun_{H/N\Top_*}(\mathscr{J}_{H,N}^{\mathcal{U}},\mathscr{T}_{H/N})  \to \Fun_{H/N\Top_*}(\mathscr{J}_{H/N}^{\mathcal{U}^N},\mathscr{T}_{H/N}),\]
which is left adjoint to $U_{\phi}$. 
\begin{definition} 
The fixed point functor is the composite functor
\[(-)^N:=U_{\nu} \circ \Fix^N:  H\Spe_{\mathcal{U}} \to H/N\Spe_{\mathcal{U}^N}. \]
\end{definition}

We note that $X^N(V)=X(q^*V)^N$ for a $H/N$-representation $V$. The fixed point functor preserves fibrations, acyclic fibrations and $\pi_*$-isomorphisms between fibrant objects.

\begin{definition}
The geometric fixed point functor is the composite
\[ (-)^{gN}:=  P_{\phi}\circ \Fix^N:  H\Spe_{\mathcal{U}} \to H/N\Spe_{\mathcal{U}^N} .\]
\end{definition}

The geometric fixed point functor preserves cofibrations, acyclic cofibrations, and $\pi_*$-isomorphisms between cofibrant objects.

Let $\overline{\gamma}:\id \to U_{\phi}P_{\phi}$ denote the unit of the adjunction $(P_{\phi},U_{\phi})$. We have a natural transformation of fixed point functors $\gamma: X^N \to X^{gN}$ given as
\[U_{\nu}(\overline{\gamma} ): U_{\nu}(\Fix^N(X)) \to   U_{\nu}U_{\phi}P_{\phi}(\Fix^N(X)) =P_{\phi}(\Fix^N(X)).\]
%
\begin{comment}
\begin{remark}\label{section}
Let $X$ denote an $H$-spectrum indexed on $\mathcal{U}$ and let $Y$ be a $H/N$-spectrum indexed on $\mathcal{U}^N$. Suppose we are given maps 
\[f: \Fix^N(X) \to Y \circ \phi, \ \quad h: Y \circ \phi \to  \Fix^N(X),\]
such that $f \circ h=\id$. Let $\tilde{f}: \phi^N(X) \to Y$ denote the adjoint of $f$. Then we obtain maps 
\[F:=\tilde{f} \circ \gamma= U_{\nu}(f): X^N \to Y, \ \quad H:=U_\nu(h): Y \to X^N,\]
and $F \circ H =\id.$ 
\end{remark}
\end{comment}

\subsection{The orthogonal spectrum $\THR(A, D)$}
We fix a complete $O(2)$-universe
\[\mathcal{U}=\Big(\bigoplus_{\alpha \geq 0}\underset{n \geq 0}{\bigoplus} \ \mathbb{C}(n)\Big)  \bigoplus \Big( \bigoplus_{\alpha \geq 0}\bigoplus_{n \geq 0} \ \mathbb{C}(n) \oplus \mathbb{C}(-n)\Big). \]
Here $\mathbb{C}(n):=\C$ with $\mathbb{T}$-action given by $z\cdot x=z^{n}x$ for $z\in \mathbb{T}$ and $x\in \C$. On the left hand side $\mathbb{C}(n)$ is the $O(2)$-representation with $\omega$ acting by complex conjugation. On the right hand side $\mathbb{C}(n) \oplus \mathbb{C}(-n)$ is the $O(2)$-representation with $\omega$ acting by $\omega(x,y)=(y,x)$. We see that  
\[
\rho_r^*{\mathbb{C}(n)}^{C_r}= \left\{ \begin{array}{rl}
\mathbb{C}(\frac{n}{r}) &\mbox{ if $r \mid n$} \\
0  & \mbox{ otherwise,}
       \end{array} \right. 
         \]
         and
         \[
\rho_r^*\left(\mathbb{C}(n) \oplus \C(-n)\right)^{C_r}= \left\{ \begin{array}{rl}
\mathbb{C}(\frac{n}{r}) \oplus \C(-\frac{n}{r})&\mbox{ if $r \mid n$} \\
0  & \mbox{ otherwise.}
       \end{array} \right. 
          \]      
 It follows that
 \[
 \rho_r^*\mathcal{U}^{C_r}=\Big(\bigoplus_{\alpha \geq 0} \bigoplus_{n \geq 0, r\mid n} \ \mathbb{C}\left(\frac{n}{r}\right)\Big)  \bigoplus \Big( \bigoplus_{\alpha \geq 0} \bigoplus_{n\geq 0, r \mid n} \ \mathbb{C}\left(\frac{n}{r}\right) \oplus \mathbb{C}\left(-\frac{n}{r}\right)\Big). 
 \]
We let $f_r: \mathcal{U} \xrightarrow{} \rho_r^*\mathcal{U}^{C_r}$ be the $O(2)$-equivariant homeomorphism given by mapping the summand $\mathbb{C}(n)$ indexed by $(\alpha, n)$ in $\mathcal{U}$ to the summand $\mathbb{C}(n)$ indexed by $(\alpha, rn)$ in $\rho_r^*\mathcal{U}^{C_r}$ and by mapping the summand $\mathbb{C}(n) \oplus \C(-n)$ indexed by $(\alpha, n)$ in $\mathcal{U}$ to the summand $\mathbb{C}(n)\oplus \C(-n)$ indexed by $(\alpha, rn)$ in $\rho_r^*\mathcal{U}^{C_r}$.

Let $X$ and $Y$ be pointed $O(2)$-spaces. We define a natural transformation of functors $G_X^F \wedge Y \Rightarrow G^F_{X\wedge Y}$ from $ I^F$ to $\Top_*$ by $(f,y) \mapsto \overline{f}$, where $\overline{f}(t)=f(t)\wedge y$. We compose the  induced map on colimits with the canonical homeomorphism commuting the functor that smashes with a fixed pointed $O(2)$-space and the homotopy colimit functor to obtain maps
\[\left(\hocolim_{I^F} G_X^F\right) \wedge Y \xrightarrow{\cong}  \hocolim_{I^F} (G_X^F \wedge Y) \to \hocolim_{I^F} G^F_{X\wedge Y}.\] 
Since these maps commute with the dihedral structure maps and the $O(2)$-action on $X$ and $Y$, we obtain an $O(2)$-equivariant map
\[\sigma_{X,Y}:\triangle^* (\THR(A,D;X))\wedge Y \to \triangle^*\THR(A,D; X\wedge Y),\]
where $\triangle: O(2) \to O(2) \times O(2)$ denotes the diagonal map and $O(2)$ acts diagonally on the domain.
\begin{definition}
Let $V \subset \mathcal{U}$ be an $O(2)$-representation. Let
\[\THR(A,D)(V)=\triangle^* \THR(A,D; S^V). \]
The group $O(V)$ acts on $\THR(A,D)(V)$ through the action on the sphere $S^V$. The family of $O(V) \rtimes O(2)$-spaces $\THR(A,D)(V)$ together with the structure maps 
\[\sigma_{V,W}:= \sigma_{S^V,S^W}: \THR(A,D)(V)  \wedge S^W \to \THR(A,D)(V\oplus W),\]
defines an orthogonal $O(2)$-spectrum indexed on $\mathcal{U}$, which is denoted $\THR(A,D)$. 
\end{definition}

The following result extends the classical result for the $\mathbb{T}$-spectrum $\THH(A)$; see \cite[Prop. 1.4]{HM97}. Let $\mathscr{F}$ denote the family of finite subgroups of $O(2)$.

\begin{prop}\label{fibrant}
Let $H <O(2)$ be a finite subgroup. For all finite dimensional $O(2)$-representations $V \subset W$  
the adjoint of the structure map
\[\widetilde{\sigma}_{V,W-V}: \THR(A,D)(V) \to \Omega^{W-V} \THR(A,D)(W) \]
induces a weak equivalence on $H$-fixed points. In other words, $\THR(A,D)$ is fibrant in the $\mathscr{F}$-model structure.
\end{prop}

\begin{proof}
It suffices to prove the statement for $H \in \{ C_r,D_{r}\}_{r\geq 0}$. The case $H=C_r$ is done in \cite[Prop. 1.4]{HM97}. We assume for convenience that $V=0$, the general case is analogous.

Assume $H=D_{r}$ with $r$ odd.   Let $i \in  \ob\big((I^{G})^{\{1, t_{2r}\}}\times I^{D_r\cdot F /D_r ^- }\big)$. We introduce the notation
\begin{align*}
G^{D_rF/D_r}\circ \triangle^e_r(i) \star  S^{W} :=
\Map\Big(\big(\bigwedge_{z\in D_{r}\cdot F} S^{i_{\overline{z}}}\big)\wedge S^{W} , \big(\bigwedge_{z\in D_{r}\cdot F } A_{i_{\overline{z}}}\big)\wedge S^{W}\Big) .
\end{align*}
The domain and target are given the diagonal $O(2)$-action and the action on the mapping space is by conjugation. The map $\widetilde{\sigma}$ becomes the top row in the commutative diagram below, when we restrict to the cofinal subcategory $D_{r}\mathcal{F}_*$.
\begin{center}
\begin{tikzpicture}
  \matrix (m) [matrix of math nodes,row sep=0.8cm,column sep=1.4cm,minimum width=2em] {
\left(\underset{ D_{r} \mathcal{F}_*}{\colim} \ \underset{I^{D_r F}}{\hocolim} \ G^{D_{r}\cdot F} \right)^{D_r}  &  \left( \Omega^{W}\left(\underset{ D_{r} \mathcal{F}_*}{\colim} \ \underset{I^{D_r F}}{\hocolim} \ G_{S^W}^{D_{r}\cdot F} \right) \right)^{D_{r}} \\
   
   &   \left(\underset{ D_{r} \mathcal{F}_*}{\colim} \ \Omega^{W} \ \underset{I^{D_r F}}{\hocolim} \ G_{S^W}^{D_{r}\cdot F} \right)^{D_{r}}\\
   
   &  \left(\underset{ D_{r} \mathcal{F}_*}{\colim} \  \underset{I^{D_r F}}{\hocolim} \ \Omega^{W} \ G_{S^{W}}^{D_{r}\cdot F} \right)^{D_{r}}\\
   
\underset{ D_{r} \mathcal{F}_*}{\colim} \ \underset{\overline{I}}{\hocolim} \ \left(G^{D_{r}\cdot F} \circ \triangle_r^e \right)^{D_{r}}  &  \underset{ D_{r} \mathcal{F}_*}{\colim} \ \underset{\overline{I}}{\hocolim} \  (G^{D_rF/D_r}\circ \triangle^e_r(i) \star  S^{W} )^{D_{r}}  \\};
  \path[-stealth]
    (m-1-1) edge node [above] {$\widetilde{\sigma}$} (m-1-2)   
     (m-4-1) edge node [left] {$\cong$} (m-1-1)  
        (m-4-1) edge node [right] {$\triangle$} (m-1-1)  
     (m-4-2) edge node [left] {$\cong$} (m-3-2)  
       (m-4-2) edge node [right] {$\triangle$} (m-3-2)  
         (m-3-2) edge node [right] {$\operatorname{can}$} (m-2-2) 
            (m-3-2) edge node [left] {$\sim$} (m-2-2) 
             (m-2-2) edge node [right] {$\gamma^{D_{r}}$} (m-1-2) 
                   (m-2-2) edge node [left] {$\sim$} (m-1-2) 
         (m-4-1) edge node [above] {$(\eta^*)^{D_{r}}$} (m-4-2)  
             (m-4-1) edge node [below] {$\sim$} (m-4-2)

    ;
\end{tikzpicture}
\end{center}
where $\overline{I}:=(I^{G})^{\{1, t_{2r}\}}\times I^{D_r\cdot F /D_r ^- }$. The maps labelled $\triangle$ are homeomorphisms by Lemma \ref{subcat}. The bottom map is induced from the adjunction unit
\[\eta: \bigwedge_{z\in D_{2r}\cdot F } A_{i_{\overline{z}}}  \to \Map \big (S^W,  S^{W} \wedge \bigwedge_{z\in D_{2r}\cdot F } A_{i_{\overline{z}}} \big)\]
followed by a homeomorphism. The connectivity of $(\eta^*)^{D_r}$ can be estimated using the equivariant suspension Theorem \ref{suspension} and Lemma \ref{A2}:
\[\conn((\eta^*)^{D_{r}})\geq \Big\lceil\frac{i_1}{2}\Big\rceil+\Big\lceil\frac{i_{t_{2r}} }{2}\Big\rceil + \sum_{\overline{z} \in D_rF/D_r^-}i_{\overline{z}} .\]
Since the connectivity tends to infinity with $i$, the induced map on homotopy colimits is a weak equivalence by Lemma \ref{infinity}. A similar argument shows that $\operatorname{can}$ is a weak equivalence. Finally $\gamma^{D_{r}}$ is a weak equivalence by \cite[Lemma 1.4]{HM97}.

The case $H=D_r$ with $r$ even, can be done analogously by restricting to the cofinal subcategory $D_{2r}\mathcal{F}_*$, compare Lemma \ref{EAL} part \eqref{part 3}.
\end{proof}

\subsection{The cyclotomic structure}
The $\mathbb{T}$-spectrum underlying $\THR(A, D)$ is cyclotomic; see \cite[Def. 1.2, Prop. 1.5]{HM97}. We prove that the cyclotomic structure is compatible with the $G$-action, i.e $\THR(A,D)$ is an $O(2)$-cyclotomic spectrum. To define this notion, we must first introduce some notation.  Let
\[
\rho_r: O(2)\to O(2)/C_r 
\] 
be the root isomorphism given by $\rho_r(z)=z^{\frac{1}{r}}C_r$ if $z\in \mathbb{T}$ and $\rho_r(x)=x$ if $x\in G$. The isomorphism $\rho_{r}$ induces isomorphisms of enriched categories 
\[ \mathscr{J}_{\rho_r^*}: \mathscr{J}_{O(2)/C_r}^{\mathcal{U}^{C_r}}  \xrightarrow{\cong} \mathscr{J}_{O(2)}^{\rho_r^* U^{C_r}}, \quad \mathscr{T}_{\rho_r^*}:\mathscr{T}_{O(2)/C_r} \xrightarrow{\cong} \mathscr{T}_{O(2)}.\]
The isomorphism of universes $f_r: \mathcal{U} \xrightarrow{\cong } \rho_r^* \mathcal{U}^{C_r}$ defined earlier induces an isomorphism of enriched categories
\[ 
f_r: \mathscr{J}_{O(2)}^{\mathcal{U}}  \xrightarrow{\cong} \mathscr{J}_{O(2)}^{\rho_r^* \mathcal{U}^{C_r}}.
\]
If $Y$ is a $O(2)/C_r$-spectrum indexed on $\mathcal{U}^{C_r}$, then we let $\rho_r^*Y$ be the $O(2)$-spectrum indexed on $\mathcal{U}$ given by the composition
\[ \mathscr{J}_{O(2)}^{\mathcal{U}}  \xrightarrow{f_r} \mathscr{J}_{O(2)}^{\rho_r^* U^{C_r}} \xrightarrow{\mathscr{J}_{\rho_r^*}^{-1}}   \mathscr{J}_{O(2)/C_r}^{\mathcal{U}^{C_r}}   \xrightarrow{Y} \mathscr{T}_{O(2)/C_r} \xrightarrow{ \mathscr{T}_{\rho_r^*}} \mathscr{T}_{O(2)}.\]

\begin{definition}\label{defcy}
Am $O(2)$-cyclotomic spectrum is an orthogonal $O(2)$-spectrum $X$ together with maps of orthogonal $O(2)$-spectra for all $r\geq 0$,
\[T_r: \rho_r^{*} (X^{gC_r})  \to X,\]
such that the composite from the derived geometric fixed point functor
\[ \rho_r^{*} ( X^c)^{gC_r}  \to \rho_r^{*} (X^{gC_r}) \xrightarrow{T_r} X \]
is an $\mathscr{F}$-equivalence. Furthermore we require that for all $s,r\geq 1$ the following diagram commutes:
 \[\begin{tikzpicture}[x=4.5cm,y=1.8cm]
\node (a) at (0,1) {$\rho_{rs}^{*} (X^{gC_{rs}})$};
\node (b) at (1,1) {$\rho_s^{*} ((\rho_r^{*} (X^{gC_r}))^{gC_s})$};
\node (c) at (2,1) {$\rho_s^{*} (X^{gC_s})$};
\node (d) at (0,0) {$\rho_r^{*} ((\rho_s^{*} (X^{gC_s}))^{gC_r})$};
\node (e) at (1,0) {$\rho_r^{*} (X^{gC_r})$};
\node (f) at (2,0) {$X.$};

\draw[->](a) to node [above]{} (b);
\draw[->](b) to node [above]{$\rho_s^{*} (T_r^{gC_s})$} (c);
\draw[->](c) to node [right]{$T_r$} (f);
\draw[->](a) to node [below]{} (d);
\draw[->](d) to node [below]{$\rho_r^{*} (T_s^{gC_r})$} (e);
\draw[->](e) to node [below]{$T_s$} (f);

\end{tikzpicture}\]
\end{definition}
The geometric fixed point functor $X^{gC_r}$ is defined as a left Kan extension via 
\[\phi :  \mathscr{J}_{O(2);C_r}^{U}\to \mathscr{J}_{O(2)/C_r}^{U^{C_r}}, \quad\phi(V)=V^{C_r},\] 
of the functor $\Fix^{C_r} X$. To construct cyclotomic structure maps it therefore suffices to construct $O(2)$-equivariant maps 
\[\tilde{T}_r: \rho_r^*(X(V)^{C_r}) \to X(\rho_r^*(V^{C_r})) ,\]
for each representation $V$ satisfying certain compatibility conditions to ensure commutativity of the diagram above; compare \cite[Lemma 1.2]{HM97}. 

\begin{example}
The $O(2)$-equivariant sphere spectrum $\mathbb{S}$ is an $O(2)$-cyclotomic spectrum with structure maps arising from the identity maps: $\rho_r^*((S^{V})^{C_r}) \to S^{\rho_r^*(V^{C_r})}. $ 
\end{example}

Both in the next example and in the construction of the cyclotomic structure map for $\THR(A,D)$, we will need to pull back the group action on a diagram along a group homomorphism. Let $g: K \to H$ be a group homomorphism.  Let $(X: J \to \mathsf{Sets}, \alpha)$ be a $H$-diagram indexed by $J$. Let $g^*J$ denote the category $J$ with $K$-action defined by $k:=g(k): J \to J$. We have natural transformations 
\[
g^*\alpha_k:=\alpha_{g(k)}: X \Rightarrow X \circ g(k).
\] 
This gives a $K$-diagram indexed by  $ g^*J $ and there is a unique isomorphism of $K$-sets:
\[\underset{g^*J}{\colim} \ g^*X \cong g^*(\underset{J}{\colim} \ X).\]

\begin{example}
Let $\Gamma$ be a topological group. The geometric realization of the dihedral bar construction on $\Gamma$ is an $O(2)$-space which we denote $\Bdi\Gamma$. As explained in the introduction, there is an $\mathscr{F}$-equivalence from  $\Bdi\Gamma$ to the free loop space on the classifying space, $\Map(\mathbb{T}, B\Gamma)$, if we give the free loop space an $O(2)$-action as follows: The group $O(2)$ acts on $\mathbb{T}$ by multiplication and complex cojugation. Taking inverses in the group induces a $G$-action on $B\Gamma$, and we view $B\Gamma$ as an $O(2)$-space with trivial $\mathbb{T}$-action. We let $O(2)$ act on the free loop space by the conjugation action. 

There are $\mathbb{T}$-equivariant homeomorphisms
\[ p_r: \Bdi\Gamma\xrightarrow{\cong} \rho_{r}^* \big( \Bdi\Gamma^{C_r} \big).\]
are constructed in \cite[Sect. 2.1.7]{Ma}. We run trough the construction of $p_r$ to ensure that it is $O(2)$-equivariant. Since we already know that $p_r$ is continuous, we will not keep track of continuity. First we recall the dihedral bar constrcution on $\Gamma$. For $F\in \mathcal{F}$, let $\Bdi[F]= \Gamma^F$. Let $F\subset F'$ be finite subsets of the circle. We define $s_F^{F'}:\Gamma^F \to \Gamma^{F'}$ by repeating the identity element $1 \in \Gamma$ as pictured in the example:
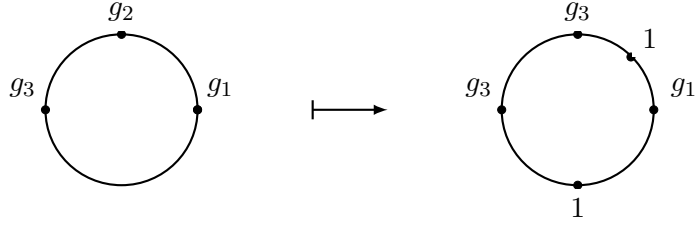
\begin{center}
    \begin{tikzpicture}[scale=1,cap=round,>=latex]
        % draw the unit circle
        \draw[thick] (0cm,0cm) circle(1cm);

        \foreach \x in {0,90,180,360} {
            % dots at each point
                \filldraw[black] (\x:1cm) circle(1.5pt);   }

        % the placement is better this way
        \draw (-1.3cm,0cm) node[above=1pt] {$g_3$}
              (1.3cm,0cm)  node[above=1pt] {$g_1$}
              (0cm,1.3cm)  node[fill=white] {$g_2$};
              
        \draw[thick,|->] (2.5cm,0cm) --  (3.5cm,0cm);
        % draw the unit circle
        \draw[thick] (6cm,0cm) circle(1cm);

    \filldraw[black] (5cm,0cm) circle(1.5pt);   
        \filldraw[black] (7cm,0cm) circle(1.5pt);   
            \filldraw[black] (6cm,1cm) circle(1.5pt);   
              \filldraw[black] (6.7cm,0.7cm) circle(1.5pt);   
                   \filldraw[black] (6cm,-1cm) circle(1.5pt);     
        % the placement is better this way
        \draw (4.7cm,0cm) node[above=1pt] {$g_3$}
              (7.4cm,0cm)  node[above=1pt] {$g_1$}
            (6cm,-1.3cm) node[fill=white] {$1$}
            (6.95cm,0.95cm) node[fill=white] {$1$}
              (6cm,1.3cm) node[fill=white] {$g_3$};

    \end{tikzpicture}
        \end{center}
We define $d_F^{F'}:\Gamma^{F'} \to \Gamma^F$ by pushing the group elements counter clockwise around the circle using the multiplication in $\Gamma$, as pictured in the example:
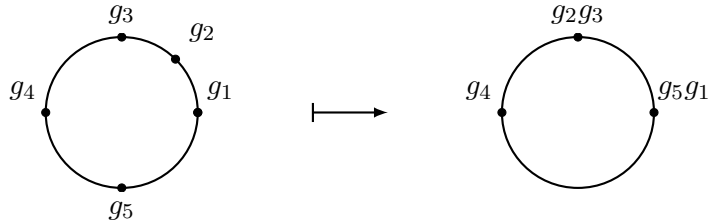
\begin{center}
    \begin{tikzpicture}[scale=1,cap=round,>=latex]
        % draw the unit circle
        \draw[thick] (0cm,0cm) circle(1cm);

        \foreach \x in {0,45,90,180,270,360} {
            % dots at each point
                \filldraw[black] (\x:1cm) circle(1.5pt);   }

        \foreach \x/\xtext in {
            % the coordinates for the first quadrant
            45/{g_2}}
                \draw (\x:1.5cm) node[fill=white] {$\xtext$};
        % the placement is better this way
        \draw (-1.3cm,0cm) node[above=1pt] {$g_4$}
              (1.3cm,0cm)  node[above=1pt] {$g_1$}
              (0cm,-1.3cm) node[fill=white] {$g_5$}
              (0cm,1.3cm)  node[fill=white] {$g_3$};
              
        \draw[thick,|->] (2.5cm,0cm) --  (3.5cm,0cm);
        % draw the unit circle
        \draw[thick] (6cm,0cm) circle(1cm);

    \filldraw[black] (5cm,0cm) circle(1.5pt);   
        \filldraw[black] (7cm,0cm) circle(1.5pt);   
            \filldraw[black] (6cm,1cm) circle(1.5pt);    
        % the placement is better this way
        \draw (4.7cm,0cm) node[above=1pt] {$g_4$}
              (7.4cm,0cm)  node[above=1pt] {$g_5g_1$}
              (6cm,1.3cm) node[fill=white] {$g_2g_3$};
              
    \end{tikzpicture}
        \end{center}
We describe the natural transformations $t_F:\Gamma^F \xrightarrow{} \Gamma^{t(F)}$ and $ \omega_F: \Gamma^F \xrightarrow{} \Gamma^{\omega ( F) }$, where $t\in \mathbb{T}$ and $\omega$ is complex conjugation. They both permute a tuple of group elements accordingly and $\omega$, in addition, takes each label to its inverse:
\begin{align*}
t_F\left((g_z)_{z\in F}\right) = (g_{t^{-1}(y)})_{y\in t(F)}, \quad \omega_F \left( (g_z)_{z\in F} \right)=((g_{\omega^{-1}(y)})^{-1})_{y\in \omega(F)}.
\end{align*}
The geometric realization $\Bdi\Gamma:=\colim_{F\in \mathcal{F}} \Gamma^F$ is an $O(2)$-space.

We proceed to construct the map $p_r$. Let $\overline{\mathcal{F}}$ denote the category of finite subsets of $\mathbb{T}/C_r$ and set inclusions. We have $O(2)/C_r$-equivariant bijections: 
\begin{spreadlines}{0.8em}
\begin{align*}
\left( \underset{ F \in \mathcal{F}}{\colim} \ \Gamma^F \right)^{C_r} &\xleftarrow{\cong}
\left(\underset{C\cdot F \in C_r\mathcal{F}}{\colim} \ \Gamma^{C_r\cdot F}\right)^{C_r} \\ &\xleftarrow{\cong} \underset{C_r\cdot F \in C_r \mathcal{F}}{\colim} \ \left(\Gamma^{C_r\cdot F}\right)^{C_r} \\
&\xleftarrow{\cong} \underset{C_r\cdot F \in C_r \mathcal{F}}{\colim} \ \Gamma^{C_r\cdot F/C_r} \\
&\xrightarrow{\cong} \underset{\overline{F} \in \overline{\mathcal{F}}}{\colim} \ \Gamma^{\overline{F}}.
\end{align*}
\end{spreadlines}
The first map is induced by the inclusion of categories $C_r \mathcal{F} \hookrightarrow \mathcal{F}$, the second map is induced by the inclusion $(\Gamma^{C_r\cdot F})^{C_r} \hookrightarrow \Gamma^{C_r\cdot F}$, the third map is induced by the diagonal isomorphism $\Gamma^{C_r\cdot F/C_r} \to \left(\Gamma^{C_r\cdot F}\right)^{C_r}$, and, finally the last map is induced by the $O(2)/C_r$-equivariant isomorphism of categories $ C_r\mathcal{F} \to \overline{\mathcal{F}}$ given on objects by $C_r\cdot F \mapsto C_r \cdot F /C_r$. If we pull back the diagram $\overline{F} \mapsto \Gamma^{\overline{F}}$ along $\rho_r$, then the isomorphism of categories $\mathcal{F} \to \overline{\mathcal{F}}$ given by $ F \mapsto \rho_r(F)$ induces an $O(2)$-bijection:
\begin{spreadlines}{0.8em}
\begin{align*}
\rho_r^* \left(\underset{\overline{F} \in \overline{\mathcal{F}}}{\colim} \ \Gamma^{\overline{F}} \right) \xleftarrow{\cong}
\underset{F \in \mathcal{F}}{\colim} \ \Gamma^{\rho_r(F)}
\end{align*}
\end{spreadlines}
Finally the isomorphism of dihedral sets $\Gamma^F \to \Gamma^{\rho_r(F)}$ induces an $O(2)$-bijection
\begin{spreadlines}{0.8em}
\begin{align*}
\underset{F \in \mathcal{F}}{\colim} \ \Gamma^{\rho_r(F)} \xleftarrow{\cong} \underset{F \in \mathcal{F}}{\colim} \ \Gamma^{F}.
\end{align*}
\end{spreadlines}
We combine the maps above to obtain $p_r$, which is indeed $O(2)$-equivariant. We give the suspension spectrum $\Sigma^{\infty}_{O(2)} \Bdi\Gamma_+ $ the structure of an $O(2)$-cyclotomic spectrum by letting $\tilde{T}_r$ be the map:
\[
 S^{\rho_r^*(V^{C_r})} \wedge \rho_r^* (\Bdi\Gamma)_+^{C_r}\xrightarrow{\id \wedge p_r^{-1}} S^{\rho_r^*(V^{C_r})} \wedge \Bdi\Gamma_+ . \]
\end{example}

Let $V$ be a finite dimensional $O(2)$-representation and let $r\geq 1$. We run through the construction of the map 
\begin{align*}
\tilde{T_r}:\rho_r^*\big(\THR(A,D)(V)^{C_r}\big) \to \THR(A,D)\big(\rho_r^*(V^{C_r})\big).
\end{align*}
defined in \cite[Sect. 1.5]{HM97}, to check that it is $O(2)$-equivariant. Since it is already known that the map is continuous, we will not keep track of continuity. Consider the $O(2)/C_r \times O(2)/C_r$-diagram, where the second copy of $O(2)/C_r$ acts trivially on the category: 
 \[C_r\mathcal{F}\to \Top_*, \quad {C_r\cdot F} \mapsto \THR(A,D; S^{V^{C_r}})[C_r\cdot F/C_r] \] 
We pull the diagram back along 
\[
D:\triangle(O(2))/\triangle(C_r) \to O(2)/C_r \times O(2)/C_r, \quad D((a,a)\triangle(C_r))=(aC_r,aC_r),
\] 
and define a natural transformation of $\triangle(O(2))/\triangle(C_r)$-diagrams by
 \begin{align*}
 \THR(A,D; S^{V} )&[C_r \cdot F]^{\triangle(C_r)} \xleftarrow{\cong} \underset{I^{C_r \cdot F/C_r}}{\hocolim} \ \left( G^{C_rF}_{S^V} \circ \triangle_r \right)^{\triangle(C_r)}  \\
  &=\underset{I^{C_r \cdot F/C_r}}{\hocolim} \ \Big( \Map \Big( (\bigwedge_{\overline{z}\in C_r \cdot F/C_r} S^{i_{\overline{z}}} )^{\wedge r}, (\bigwedge_{{\overline{z}}\in C_r \cdot F/C_r} A_{i_{\overline{z}}} )^{\wedge r}\wedge S^V\Big) \Big)^{\triangle(C_r)}  \\
   &\xrightarrow{\res} \underset{I^{C_r \cdot F/C_r}}{\hocolim} \ \Map \Big( \bigwedge_{{\overline{z}}\in C_r \cdot F/C_r} S^{i_{\overline{z}}} , \bigwedge_{{\overline{z}}\in C_r \cdot F/C_r} A_{i_{\overline{z}}} \wedge S^{V^{C_r}}\Big) \\
   &=\THR(A,D; S^{V^{C_r}})[C_r\cdot F/C_r] 
 \end{align*}
The map $\res$ is induced by the natural transformation obtained by restricting a map to the fixed point space: $\Map(X,Y)^{C_r} \to \Map(X^{C_r}, Y^{C_r})$. The natural transformation above induces a map of colimits which we also denote $\res$. 

If $d: O(2)/C_r \to \triangle(O(2)) /{\triangle}(C_r)$ denotes the isomorphism $aC_r \mapsto (a,a)\triangle(C_r)$, then there is a commutative diagram of homomorphisms:
\begin{center}
\begin{tikzpicture}
  \matrix (a) [matrix of math nodes,row sep=2em,column sep=2em,minimum width=2em]  {
O(2) & O(2)/C_r & \triangle(O(2))/\triangle(C_r)\\
O(2) \times O(2)  &  & O(2)/C_r \times O(2)/C_r.\\};
  \path[-stealth] (a-1-1) edge node [right] {$\triangle$} (a-2-1)
  (a-1-2) edge node [above] {$d$} (a-1-3)
  (a-1-3) edge node [right] {$D$} (a-2-3)
   (a-2-1) edge node [above] {$\rho_r \times \rho_r$} (a-2-3)
    (a-1-1) edge node [above] {$\rho_r$} (a-1-2);
     .
\end{tikzpicture}
\end{center}
We have a string of $O(2)/C_r$-equivariant set maps:
\begin{spreadlines}{0.8em}
\begin{align*}
   \rho^*_r \Big(\triangle^* \colim_{ F \in \mathcal{F}} \THR(A,&D; S^{V} )[F] \Big)^{C_r}  =
(d \circ \rho_{r})^*    \left(\colim_{ F \in \mathcal{F}} \THR(A,D; S^{V} )[F] \right)^{\triangle(C_r)}  \\
&\xleftarrow{\cong} (d \circ \rho_{r})^*   \Big(\colim_{C_r\cdot F \in C_r\mathcal{F}}\THR(A,D; S^{V} )[C_r \cdot F]^{\triangle(C_r)} \Big)\\ 
&\xrightarrow{\res}  (d \circ \rho_{r})^*\Big( D^*  \Big( \colim_{C_r\cdot F \in C_r\mathcal{F}} \THR(A,D; S^{V^{C_r}})[C_r\cdot F/C_r]  \Big) \Big)  \\
&=   (\triangle^*)\Big(\rho_r^* \times \rho_r^*  \Big( \colim_{C_r\cdot F \in C_r\mathcal{F}} \THR(A,D; S^{V^{C_r}})[C_r\cdot F/C_r]  \Big) \Big)  \\
&\xleftarrow{\cong} (\triangle^*)\Big(\rho_r^* \times \rho_r^* \Big(  \colim_{\overline{F} \in \overline{\mathcal{F}}} \THR(A,D; S^{V^{C_r}})[\overline{F}]   \Big) \Big) \\ 
&\xleftarrow{\cong} \triangle^* \colim_{F \in \mathcal{F}} \THR(A,D; S^{\rho_r^*(V^{C_r})})[ \rho_r(F)]  \\
&\xleftarrow{\cong} \triangle^* \colim_{F \in \mathcal{F}} \THR(A,D; S^{\rho_r^*(V^{C_r})})[ F]. 
\end{align*}
\end{spreadlines}
The maps labelled $\cong$ are bijections.  The map in the second line is induced by the inclusion of categories $C_r \mathcal{F} \hookrightarrow \mathcal{F}$. The map in the fifth line is induced by the isomorphism of $O(2)/C_r$-categories $C_r \mathcal{F} \to \overline{\mathcal{F}}$ given by $C_r\cdot F \mapsto C_r \cdot F /C_r$. The map in the sixth line is induced by the functor $\mathcal{F} \to \overline{\mathcal{F}}$ given by $ F \mapsto \rho_r(F)$ and finally the last map is induced by the obvious isomorphism of dihedral spaces. The composition is the cyclotomic structure map $\tilde{T_r}$, which is indeed $O(2)$-equivariant. 

\begin{theorem}
The composite
\[ \rho_{r}^*(\THR(A,D)^c)^{gC_r} \to \rho_{r}^*\THR(A,D)^{gC_r} \xrightarrow{T_r} \THR(A,D)\]
is an $\mathscr{F}$-equivalence.
\end{theorem}

\begin{proof}
By \cite[Lemma 4.10]{MM} it suffices to show that the map of $\mathscr{J}_{O(2);C_r}^{U}$-spaces
\[ \tilde{T}_r: \Fix^{C_r} \THR(A, D) \to (\rho_{r}^{-1})^* (\THR(A,D) \circ \phi),\]
induces an isomorphism on $\pi^H_q(-)$ for all finite subgroups $H \leq O(2)/C_r$, where these are the homotopy groups of $\mathscr{J}_{O(2);C_r}^{U}$-spaces, see \cite[Def. 4.8]{MM}. We consider the case $q=0$, the general case is similar. More specifically, we must show that the connectivity of the induced $O(2)/C_r$-map on $H$-fixed points
\[ \left(\Omega ^{V^{C_r}} \THR(A, D)(V)^{C_r}\right)^H \to \left(\Omega ^{V^{C_r}} (\rho_r^{-1})^*\THR(A,D)(\rho_r^*V^{C_r})\right)^H,\] 
tends to infinity with $V$ for all finite subgroups $H \leq O(2)/C_r$.

The only non-homeomorphism in the definition of $\tilde{T}_r$,  is the restriction map induced by the natural transformation induced by restricting a map to the fixed point space: $\Map(X,Y)^{C_r} \to \Map(X^{C_r}, Y^{C_r})$. We let $H=D_{rs}/C_r$, the case $H=C_{rs}/C_r$ is analogous. We assume that $2\nmid rs$. The case $2\mid rs$, can be done analogously by restricting to the cofinal subcategory $D_{2sr}\mathcal{F}_*$, compare Lemma \ref{EAL} part \eqref{part 3}. We restrict to the cofinal subcategory $D_{rs}\mathcal{F}_* \subset C_r \mathcal{F}$. The map in question is then 
 \[
 \begin{tikzpicture}
\node(a){$   \Big(\Omega ^{V^{C_r}} \colim_{D_{sr} \mathcal{F}_*}\left(\THR(A,D; S^{V} )[D_{sr}F]  \right)^{C_r}  \Big)^H $};
\node(e)[below of=a, node distance = 1.5cm]{$\Big(\Omega ^{V^{C_r}} \colim_{D_{sr} \mathcal{F}_*} D^* \THR(A,D; S^{V^{C_r}} )[D_{sr}F/C_r]  \Big)^H$};

\draw[->](a) to node [right]{$\res$} (e);

\end{tikzpicture}
\]
where $H=\triangle(D_{rs})/\triangle(C_r)$, $\Omega ^{V^{C_r}}$ is viewed as a $\triangle(O(2))/\triangle(C_r)$-space, $D_{rs}\mathcal{F}_*$ as a category with a $\triangle(O(2))/\triangle(C_r)$-action and $D$ is the homomorphism 
\[
\triangle(O(2))/\triangle(C_r) \to O(2)/C_r \times O(2)/C_r, \quad D(a,a)\triangle(C_r)=(aC_r,aC_r).
\] 
It follows from \cite[Lemma 1.4]{HM97} that we can move $\Omega^{V^{C_r}}$ past the colimit up to weak equivalence. We can then take fixed points before taking the colimit. Thus we are reduced to showing that the connectivity of the map 
\begin{align*}
   \Big(\Omega ^{V^{C_r}} \left(\THR(A,D; S^{V} )[D_{sr}F]  \right)^{C_r}  \Big)^H   \xrightarrow{\res} \Big(\Omega ^{V^{C_r}} D^* \THR(A,D; S^{V^{C_r}} )[D_{sr}F/C_r]  \Big)^H
\end{align*}
tends to infinity with $V$. We can move $\Omega ^{V^{C_r}}$ past the homotopy colimit, up to weak equivalence, and by Lemma \ref{subcat} we can take fixed points before taking homotopy colimits. Let $i \in \ob\big((I^{G})^{\{1,t_{2rs}\}}  \times I^{D_{sr}F/D_{sr}^-}\big)$. By Lemma \ref{infinity}, we are reduced to showing that the connectivity of the map
 \[
 \begin{tikzpicture}
\node(a){$\left( \Omega ^{V^{C_r}} \Map( \bigwedge_{z\in D_{sr}F}S^{i_{\overline{z}}}, \bigwedge_{z\in D_{sr}F}A_{i_{\overline{z}}} \wedge S^V )^{C_r}\right)^H $};
\node(e)[below of=a, node distance = 1.5cm]{$ \left(\Omega ^{V^{C_r}} \Map(( \bigwedge_{D_{sr}F} S^{i_{\overline{z}}})^{C_r} , (\bigwedge_{D_{sr}F} A_{i_{\overline{z}}} \wedge S^{V})^{C_r}) \right)^H. $};

\draw[->](a) to node [right]{$\res$} (e);

\end{tikzpicture}
\]
tends to infinity with $V$ and $i$. We can rewrite the map in question as
 \[
 \begin{tikzpicture}
\node(a){$ \Map_{D_{rs}}\Big( S^{V^{C_r}}\wedge \bigwedge_{z\in D_{sr}F}S^{i_{\overline{z}}} , \bigwedge_{z\in D_{sr}F}A_{i_{\overline{z}}}\wedge S^V \Big)  $};
\node(e)[below of=a, node distance = 1.5cm]{$ \Map_{D_{rs}}\Big(  S^{V^{C_r}} \wedge (\bigwedge_{D_{sr}F} S^{i_{\overline{z}}})^{C_r}   , \bigwedge_{z \in D_{sr}F} A_{i_{\overline{z}}} \wedge S^V \Big).$};

\draw[->](a) to node [right]{$\res$} (e);

\end{tikzpicture}
\]
This is a fibration with fiber 
\[\Map_{D_{rs}}\Big(S^{V^{C_r}}  \wedge (  \bigwedge_{z\in D_{sr}F}S^{i_{\overline{z}}}  / (\bigwedge_{z\in D_{sr}F}S^{i_{\overline{z}}})^{C_r} ), \bigwedge_{z \in D_{sr}F} A_{i_{\overline{z}}} \wedge S^V \Big).\] 
Let $S^W= \bigwedge_{z\in D_{sr}F}S^{i_{\overline{z}}}$. By Lemma \ref{conn of mapping space} the connectivity of the fiber is greater than or equal to 
\[\underset{K\leq D_{rs}}{\min} \Big(\conn \Big( (\bigwedge_{z \in D_{sr}F} A_{i_{\overline{z}}} \wedge S^V )^K\Big)-\dim \Big( (S^{V^{C_r}}  \wedge S^{W}/S^{W^{C_r}})^K \Big)\Big). \]
 Let $t$ divide $rs$. It follows from the connectivity assumptions \ref{concond} on $(A,D)$ that
\begin{align*}
\conn(\bigwedge_{z \in D_{sr}F} A_{i_z} \wedge S^V )^{C_t} &\geq \frac{rs}{t} \Big(i_1+i_{t_{2rs}} + \sum_{\overline{z}\in D_{sr} F/D_{sr}^-} 2i_{\overline{z}}\Big) + \dim(V^{C_t})-1,  \\
\conn(\bigwedge_{z \in D_{sr}F} A_{i_z} \wedge S^V )^{D_{t}}&\geq \frac{rs}{t} \Big(\Big\lceil\frac{i_1}{2}\Big\rceil +\Big\lceil\frac{i_{t_{2rs}}}{2}\Big\rceil + \sum_{\overline{z}\in D_{sr} F/D_{sr}^-} i_{\overline{z}}\Big) + \dim(V^{D_{t}}) -1.
\end{align*} 
If $C_r \leq C_t$ then the dimension of both $(S^{W}/S^{W^{C_r}})^{C_t}$ and $(S^{W}/S^{W^{C_r}})^{D_t}$ is 0, hence the connectivity of the fiber tends to infinity with $i$.  Otherwise, since $C_r$ is normal in $D_{rs}$, we have an splitting $W=W^{C_r} \oplus W'$, and by Lemma \ref{dimension}:
\begin{align*}
\dim\left( S^{V^{C_r}}  \wedge (  S^W/S^{W^{C_r}}) \right)^{C_t} &= 
 \frac{rs}{t} \Big(i_1+i_{t_{2rs}} + \sum_{\overline{z}\in D_{sr} F/D_{sr}^-} 2i_{\overline{z}}\Big) + \dim(V^{C_rC_t}) \\
\dim\left( S^{V^{C_r}}  \wedge (  S^W/S^{W^{C_r}})  \right)^{D_{t}} &= 
\frac{rs}{t} \Big(\Big\lceil\frac{i_1}{2}\Big\rceil +\Big\lceil\frac{i_{t_{2rs}}}{2}\Big\rceil + \sum_{\overline{z}\in D_{sr} F/D_{sr}^-} i_{\overline{z}}\Big) +   \dim(V^{C_rD_t}).
\end{align*}
Thus the connectivity of the fiber is greater than 
\[
{\min}_{C_r \nleq K\leq D_{rs}}(\dim(V^K) -\dim( V^{C_r K}))-1.
\] 
We can find a $D_{rs}$-representation $V'$ which is fixed by $K$ but not by the bigger group $C_rK$. Adding copies of $V'$ to $V$, we can make the connectivity as big as we want.
\end{proof}

\section{Real topological cyclic homology}
The $O(2)$-cyclotomic structure on $\THR(A,D)$ allows us to define $G$-equivariant restriction maps $R: \THR(A, D)^{C_{p^n}} \to \THR(A, D)^{C_{p^{n-1}}}$, and we can define the real topological cyclic homology at a prime $p$ as a $G$-spectrum $\TCR(A, D;p)$ by mimicking the classical definition. Before we do so, we review some constructions from equivariant stable homotopy theory. 

\subsection{The equivariant stable homotopy category} \label{hcat}
Let $H$ be a compact Lie group. We work in the $H$-stable homotopy category which is defined to be the homotopy category of the model category of orthogonal $H$-spectra on a complete universe with the stable model structure, e.g. \cite[Chapter III, 4.1,4.2]{MM}. 

Let $X$ be an orthogonal $H$-spectrum and let $V$ be an $H$-representation. The suspension $\Sigma X$ is defined by $(\Sigma X)(V)=S^1 \wedge X(V)$. The group $O(V) \rtimes H$ acts through the action on $X(V)$, and the structure maps are the suspensions of the structure maps in $X$. The loop spectrum $\Omega X$ is defined by $(\Omega X)(V)=\Map(S^1, X(V))$. The group $O(V) \rtimes H$ acts through the action on $X(V)$. The structure maps are given as the composite
\[
\Map(S^1, X(V))\wedge S^W \to \Map(S^1, X(V)\wedge S^W) \xrightarrow{\Map(\id, \lambda_{V,W})} \Map(S^1, X(V)).
\]
The functors are adjoint and both preserve $\pi_*$-isomorphisms. We let $\varepsilon: \Sigma \Omega\Rightarrow \id$ denote the counit of the adjunction and $\eta: \id \Rightarrow \Omega \Sigma$ denote the unit of the adjunction. Both $\varepsilon$ and $\eta$ are natural isomorphisms on the homotopy category.

Let $\psi: A \xrightarrow{} B$ be a map of pointed $H$-spaces. We define the mapping cone by
\[C_\psi= B \cup_{\psi} ([0,1] \wedge A) ,\] 
where $1\in [0,1]$ is the basepoint for the interval and $H$ acts trivially on the interval.  Let $i: B \to C_\psi$ be the inclusion. Collapsing the image of the inclusion to the basepoint defines a map $\delta: C_\psi \to S^1 \wedge A= \Sigma A$. We define the mapping cone $C_f$ of a map of orthogonal $H$-spectra $f:X \xrightarrow{} Y$ by applying this construction lewelwise. The inclusions and the collapse maps assemble into morphisms of orthogonal $H$-spectra and we obtain a sequence of orthogonal $H$-spectra
\begin{align}\label{triangles}
 X \xrightarrow{f} Y \xrightarrow{i} C_f \xrightarrow{\delta} \Sigma X.
 \end{align}

The collection of all triangles isomorphic to triangles of the form \eqref{triangles} gives the \mbox{$H$-stable} homotopy category the structure of a triangulated category, see \cite[Theorem A.12]{S13}. More precisely, Schwede takes the distinguished triangles to be all triangles isomorphic to triangles of the form 
\[
V \xrightarrow{j} W \to W/V \xrightarrow{\partial} \Sigma W,
\] 
where $j$ is a cofibration and $V$ and $W$ are cofibrant objects and $\partial$ fits in the homotopy commutative diagram 
 \[\begin{tikzpicture}
 \node(v){$V$}; \node(x)[right of = v,node distance = 2cm]{$W$};
\node(y)[right of = x,node distance = 2cm]{$W/V$};
\node(z)[right of = y,node distance = 2cm]{$\Sigma V$};
\node(a)[below of = v,node distance = 1.5cm]{$V$};
\node(b)[below of = x,node distance = 1.5cm]{$W$};
\node(c)[below of = y,node distance = 1.5cm]{$C_j$};
\node(d)[right of = c,node distance = 2cm]{$\Sigma V$};

\draw[->](v) to node [above]{$j$} (x);
\draw[->](x) to node [above]{} (y);
\draw[->](y) to node [above]{$\partial$} (z);
\draw[->](c) to node [above]{$\delta$} (d);
\draw[->](c) to node [right]{$c$} (y);
\draw[->](c) to node [left]{$\sim$} (y);
\draw[->](d) to node [right]{$\id$} (z);
\draw[->](a) to node [right]{$\id$} (v);
\draw[->](b) to node [right]{$\id$} (x);
\draw[->](a) to node [above]{$j$} (b);
\draw[->](b) to node [above]{$i$} (c);

\end{tikzpicture}\]
where $c$ collapses the cone of $V$ to the base point. In order to see that this choice makes the triangles of the form \eqref{triangles} distinguished, we first note that the map $c$ is a weak equivalence by the gluing lemma, and therefore the lower row is a distinguished triangle. Let $f: X \to Y$ be an arbitrary map. We cofibrantly replace $X$ and $Y$ and get a map $X^c \xrightarrow{f^c} Y^c$, which we factor as a cofibration $\tilde{f^c}$ followed by a weak equivalence in the following diagram
 \[\begin{tikzpicture}
 \node(v){$X$}; \node(x)[right of = v,node distance = 2cm]{$Y$};
\node(y)[right of = x,node distance = 2cm]{$C_f$};
\node(z)[right of = y,node distance = 2cm]{$\Sigma X$};
\node(a)[below of = v,node distance = 1.5cm]{$X^c$};
\node(b)[right of = a,node distance = 2cm]{$Y^c$};
\node(c)[right of = b,node distance = 2cm]{$C_{f^c}$};
\node(d)[right of = c,node distance = 2cm]{$\Sigma X^c$};
\node(e)[below of = a,node distance = 1.5cm]{$X^c$};
\node(f)[right of = e,node distance = 2cm]{$Y'$};
\node(g)[right of = f,node distance = 2cm]{$C_{\tilde{f^c}}$};
\node(h)[right of = g,node distance = 2cm]{$\Sigma X^c$};

\draw[->](v) to node [above]{$f$} (x);
\draw[->](x) to node [above]{$i$} (y);
\draw[->](y) to node [above]{$\delta$} (z);

\draw[->](a) to node [above]{$f^c$} (b);
\draw[->](b) to node [above]{$i$} (c);
\draw[->](c) to node [above]{$\delta$} (d);

\draw[->](e) to node [above]{$\tilde{f^c}$} (f);
\draw[->](f) to node [above]{$i$} (g);
\draw[->](g) to node [above]{$\delta$} (h);

\draw[->](a) to node [left]{$\sim$} (v);
\draw[->](b) to node [left]{$\sim$} (x);
\draw[->](a) to node [right]{$j^c$} (v);
\draw[->](b) to node [right]{$j^c$} (x);

\draw[->](c) to node [left]{$\sim$} (y);
\draw[->](d) to node [left]{$\sim$} (z);
\draw[->](d) to node [right]{$\Sigma(j^c)$} (z);

\draw[->](e) to node [right]{$\id$} (a);
\draw[->](f) to node [left]{$\sim$} (b);
\draw[->](g) to node [left]{$\sim$} (c);
\draw[->](h) to node [right]{$\id$} (d);

\end{tikzpicture}\]
 It follows from the long exact sequence induced by the cofiber sequence that the induced map on cones are $\pi_*$-isomorphisms. We note that $Y'$ is cofibrant, 
 %since the unique map $* \to Y'$ factors as the composition of two cofibrations $* \to X^c \xrightarrow{f^c} Y'$, 
 hence the bottom triangle is distinguished, and therefore the top triangle is distinguished.

In the rest of this section we establish some identification of distinguished triangles that we will need later. 

Let $\psi: A \xrightarrow{} B$ be a map of pointed $H$-spaces. We define the homotopy fiber of $\psi$ as the pointed $H$-space 
\[H_\psi= \{(\gamma, a)\in B^{[0,1]} \times A \mid \gamma(0)=\psi(x), \ \gamma(1)=*\},\]
where $(\gamma_{*},*)$ serves as the basepoint and $H$ acts trivially on the interval. Here $\gamma_*$ refers to the constant path at the basepoint.
Let $p: H_\psi \to A$ be the projection $p(\gamma, a)=a$ and let $j: \Omega B \to H_\psi$ be the inclusion $j(\alpha)= (\alpha, *)$. We define the homotopy fiber of a map of orthogonal \mbox{$H$-spectra} $f:X \xrightarrow{} Y$ by applying this construction lewelwise. The projections and inclusions assemble into morphisms of orthogonal \mbox{$H$-spectra} $p: H_f \to X$ and $j: \Omega Y \to H_f$. See \cite[Theorem 7.1.11]{Ho} for the following lemma. 
\begin{lemma}\label{fiber}
The triangle
\[H_f \xrightarrow{p} X \xrightarrow{f} Y \xrightarrow{\Sigma(j) \circ \varepsilon^{-1}}  \Sigma H_f,\]
is distinguished, where $\varepsilon: \Sigma \Omega Y \to Y$ is the counit of the loop-suspension adjunction.
\end{lemma}

Given maps of pointed $H$-spaces $f,g: A \to B$ we define the homotopy equalizer as the pointed $H$-space
\[\HE(f,g)=\{(\gamma, a)\in B^{[0,1]} \times A \mid \gamma(0)=f(a), \ \gamma(1)=g(a)\}, \]
where $(\gamma_{*},*)$ is the basepoint and $H$ acts trivially on the interval. Let $p: \HE(f,g) \to A$ denote the projection $p(\gamma,a)=a$ and $\iota: \Omega B \to \HE(f,g)$ the inclusion $\iota(\alpha)=(\alpha, *)$.  We define the homotopy equalizer of maps of orthogonal $H$-spectra $f,g:X \xrightarrow{} Y$ by applying this construction lewelwise, and the projections and inclusions assemble into morphisms of orthogonal $H$-spectra. 

\begin{lemma}\label{lemma}
The triangle
\[\HE(f,g)  \xrightarrow{p} X \xrightarrow{f-g} Y \xrightarrow{\Sigma(\iota) \circ \varepsilon^{-1}} \Sigma \HE(f,g), \]
is distinguished, where $\varepsilon: \Sigma \Omega Y \to Y$ is the counit of the loop-suspension adjunction.
\end{lemma}

Before we prove the lemma, we introduce some notation. For $\alpha \in \Omega Y$, let $\overline{\alpha} \in \Omega Y$ denote the inverse loop, i.e. $\overline{\alpha}(t)=\alpha(1-t)$. Given two loops $\alpha, \beta \in \Omega Y$ the concatenated loop $\beta \star \alpha \in \Omega Y$ is given by
\[
\beta \star \alpha(t)= \left\{ \begin{array}{rl}
 \alpha(2t) &\mbox{ if $t \leq 1/2$,} \\
  \beta(2(t-\frac{1}{2})) &\mbox{ if $t>1/2$.}
       \end{array} \right. 
\] 
Given a path $\gamma$ in $Y$ and $s\in [0,1]$, let $\gamma_{\leq s}$ and $\gamma_{\geq s}$ denote the paths in $Y$ given by
\[
\gamma_{\leq s}(t)=\gamma(t\cdot s), \quad \gamma_{\geq s}(t)=\gamma(t\cdot (1-s) + s), \quad t\in [0,1].
\]
We have a canonical path in $\Omega Y$ from the loop $\overline{\gamma} \star \gamma$ to the constant loop $*$ given by 
\[ 
t \mapsto \overline{\gamma_{\geq t}} \star \gamma_{\leq (1-t)}.
\]
\begin{proof}
We define the map $F-G: \Omega \Sigma X \to  \Omega \Sigma Y$ to be the composition
		\[
		\Omega \Sigma X \xrightarrow{\triangle} \Omega \Sigma X \times \Omega \Sigma X \xrightarrow{\overline{(\Sigma g \circ-)}\star(\Sigma f \circ -)} \Omega \Sigma Y ,
		\]
		where $\triangle$ is the diagonal map and consider the commutative diagram:
 \[
 \begin{tikzpicture}
\node(a){$\Omega Y$}; \node(b)[right of = a,node distance = 3cm]{$\HE(f,g)$};
\node(c)[right of = b,node distance = 3cm]{$X$};
\node(d)[right of = c,node distance = 3cm]{$Y$};
\node(e)[below of = a,node distance = 1.5cm]{$\Omega ( \Omega \Sigma Y)$};
\node(f)[right of = e,node distance = 3cm]{$H_{F-G}$};
\node(g)[right of = f,node distance = 3cm]{$\Omega \Sigma X$};
\node(h)[right of = g,node distance = 3cm]{$\Omega \Sigma Y.$};

\draw[->](a) to node [above]{$\iota$} (b);
\draw[->](b) to node [above]{$p$} (c);
\draw[->](c) to node [above]{$f-g$} (d);

\draw[->](e) to node [above]{$j$} (f);
\draw[->](f) to node [above]{$P$} (g);
\draw[->](g) to node [above]{$F-G$} (h);

\draw[->](a) to node [left]{$\sim$} (e);
\draw[->](a) to node [right]{$\Omega(\eta)$} (e);
\draw[->](c) to node [right]{$\eta$} (g);
\draw[->](d) to node [right]{$\eta$} (h);
\draw[->](c) to node [left]{$\sim$} (g);
\draw[->](d) to node [left]{$\sim$} (h);

\end{tikzpicture}
\]
The composition  
\[ 
[0,1] \xrightarrow{\gamma} Y \xrightarrow{\eta}  \Omega \Sigma Y \xrightarrow{\overline {\eta( g(x))} \star (-) } \Omega \Sigma Y
 \]
is a path from $  \overline {\eta( g(x))} \star \eta(f(x)) $ to  $ \overline {\eta( g(x))}\star \eta( g(x))$.  We let $\Psi_{\gamma,x}$ be the composition of this path with the canonical path from $\eta( g(x)) \star \overline {\eta( g(x))}$ to the basepoint. We define a map $\HE(f,g) \to H_{F-G}$ by $(\gamma, x) \mapsto ( \Psi_{\gamma,x}, \eta(x))$. The constructed map is a $\pi_*$-isomorphism and completes the commutative diagram above.
\end{proof}

Consider the following diagram of orthogonal $H$-spectra
\[
\begin{tikzpicture}
\node(a){$X$}; \node(b)[right of = a,node distance = 2cm]{$Z$};
\node(c)[right of = b,node distance = 2cm]{$Y$};

\draw[->](a) to node [above]{$f$} (b);
\draw[->](c) to node [above]{$g$} (b);

\end{tikzpicture}
\]
We define the homotopy pull-back spectrum $\HP(f,g)$ lewelwise:
\[\HP(f,g)(V)=\{(x,\gamma,y) \in X(V) \times Z(V)^{[0,1]} \times Y(V) \mid \gamma(0)=f(x), \gamma(1)=g(y)\},\] 
with $H$ acting trivially on the interval. The projections $p_X: \HP(f,g)(V) \to X(V)$, $p_Y: \HP(f,g)(V)\to Y(V)$ assemble into maps of orthogonal $H$-spectra.  The following diagram is homotopy commutative: 
 \[\begin{tikzpicture}
\node(a){$\Omega Y$}; \node(b)[right of = a,node distance = 2cm]{$H_{p_Y}$};
\node(c)[right of = b,node distance = 2cm]{$\HP(f,g)$};
\node(d)[right of = c,node distance = 2cm]{$Y$};
\node(e)[below of = a,node distance = 1.5cm]{$\Omega Z$};
\node(f)[right of = e,node distance = 2cm]{$H_f$};
\node(g)[right of = f,node distance = 2cm]{$X$};
\node(h)[right of = g,node distance = 2cm]{$Z.$};

\draw[->](a) to node [above]{} (b);
\draw[->](b) to node [above]{$p$} (c);
\draw[->](c) to node [above]{$p_Y$} (d);

\draw[->](e) to node [below]{$j$} (f);
\draw[->](f) to node [above]{} (g);
\draw[->](g) to node [below]{$f$} (h);

\draw[->](a) to node [left]{$\Omega g$} (e);
\draw[->](b) to node [left]{$s$} (f);
\draw[->](b) to node [right]{$\sim$} (f);
\draw[->](c) to node [left]{$p_X$} (g);
\draw[->](d) to node [left]{$g$} (h);

\end{tikzpicture}\]
The map $s$ is defined as follows. A point in $H_{{p_Y}}(V)$ is a triple $(x, \gamma,y)\in \HP(f,g)(V)$ along with a path $\beta: I \to Y$ such that $\beta(0)=y$ and $\beta(1)=*$. Then 
\[s(\beta,(x,\gamma,y))=((g\circ \beta)\star \gamma ,x).\]
The following lemma now follows from Lemma \ref{fiber} and the diagram above:
\begin{lemma} \label{triangle HE}
The triangle
\[H_f \xrightarrow{p \circ s^{-1}} \HP(f,g) \xrightarrow{p_Y} Y  \xrightarrow{\Sigma(j) \circ \varepsilon^{-1} \circ g} \Sigma H_f, \]
is distinguished.
 \end{lemma}

Assume we have a sequence of $H$-spectra $X_i$ for $i\geq 0$, maps $f_i: X_i \to X_{i-1}$ and a map $g: X_0 \to X_0$. We recall that the unit of the loop-suspension adjunction 
\[
\eta: X_0  \xrightarrow{\sim} \Omega \Sigma(X_0)
\]
is a $\pi_*$-isomorphism and consider the following diagram:
\[\begin{tikzpicture}
\setlength{\TMP}{3pt}
\node(a){$X_0$}; \node(b)[right of = a,node distance = 6.5cm]{$\Omega \Sigma(X_0) \times \prod_{i=1}^\infty X_i$};
\node(c)[right of = b,node distance = 5cm]{$\prod_{i=1}^\infty X_i$};
\node(f)[below of = a,node distance = 1.7cm]{$X_0$};
\node(g)[right of = f,node distance = 6.5cm]{$\Omega \Sigma(X_0) \times \prod_{i=1}^\infty X_i$};
\node(h)[right of = g,node distance = 5cm]{$\prod_{i=1}^\infty X_i,$};

\draw[->](a) to node [above]{$\incl\circ \eta$} (b);
\draw[->](b) to node [above]{$\proj$} (c);

\draw[->] ([xshift=\TMP]a.south) to node [right]{$\id$} ([xshift=\TMP]f.north);
\draw[->] ([xshift=-\TMP]a.south) to node [left]{$g$} ([xshift=-\TMP]f.north);

\draw[->] ([xshift=\TMP]b.south) to node [right]{$\id$} ([xshift=\TMP]g.north);
\draw[->] ([xshift=-\TMP]b.south) to node [left]{$((\eta \circ f_1\circ \pr_1)\star \Sigma(g),f_2,f_3,\dots)$} ([xshift=-\TMP]g.north);

\draw[->] ([xshift=\TMP]c.south) to node [right]{$\id$} ([xshift=\TMP]h.north);
\draw[->] ([xshift=-\TMP]c.south) to node [left]{$q\circ \prod_{i\geq 1} f_i$} ([xshift=-\TMP]h.north);

\draw[->] (f) to node [below]{$\incl \circ \eta$} (g);
\draw[->] (g) to node [below]{$\proj$} (h);
\end{tikzpicture}\]
where  
\[(\eta \circ f_1\circ \pr_1)\star \Sigma(g): \Omega \Sigma (X_0) \times \prod_{i=1}^\infty X_i \to  \Omega \Sigma (X_0)\] is the map
\[
(\alpha,x) \mapsto\eta \circ f_1\circ \pr_1(x) \star \Sigma(g) \circ \alpha
\]
The diagram gives rise to a distinguished triangle connecting the vertical homotopy equalizers as we now explain. Consider the following diagram where the maps are defined below:
\[
\begin{tikzpicture}
\node(a){$ \Omega \Sigma X_0$}; \node(b)[right of = a,node distance = 3.5cm]{$\Omega \Sigma X_0$};
\node(c)[right of = b,node distance = 5cm]{$ \HE\big(q \circ \prod_{i\geq 1} f_i, \id\big).$};

\draw[->](a) to node [above]{$\scriptstyle \Omega \Sigma (g)-\id$} (b);
\draw[->](c) to node [above]{$\scriptstyle -\eta \circ f_1 \circ \pr_1$} (b);

\end{tikzpicture}
\]
The left hand map takes a loop $\alpha \in \Omega\Sigma X_0$ to the loop $ \Sigma g \circ \alpha \star  \overline{\alpha} \in \Omega\Sigma X_0$. The right hand map takes a pair 
\[
(\gamma, x)\in \HE\big(q \circ \prod_{i\geq 1} f_i, \id\big) \subseteq \prod_{i=1}^\infty X_i^{[0,1]} \times \prod_{i=1}^\infty X_i
\] to the loop $ \overline{\eta \circ f_1 \circ \pr_1(x)}$. A point in the homotopy pull-back of the diagram is a loop $\alpha \in \Omega\Sigma X_0$, a point $x\in  \prod_{i=1}^{\infty} X_i $ and paths 
\[ \gamma: q \circ \prod_{i\geq 1} f_i (x)\sim x, \quad  
\Sigma (g) \circ \alpha \star  \overline{\alpha} \sim  \overline{\eta \circ f_1 \circ \pr_1(x)}.
\]
The notation $x \sim y$ means that the path takes the value $x$ at 0 and the value $y$ at 1.

Next consider the diagram:
 \[\begin{tikzpicture}
\setlength{\TMP}{3pt}
\node(a){$ \Omega \Sigma X_0\times  \prod_{i=1}^{\infty} X_i$};
\node(b)[right of=a, node distance = 7cm]{$\Omega \Sigma X_0\times   \prod_{i=1}^{\infty} X_i .$};
\draw[->] ([yshift=\TMP]a.east) to node [above]{$\scriptstyle  (\eta \circ f_1 \circ \pr_1) \star\Sigma(g)\times (f_2, f_3, \dots)$} ([yshift=\TMP]b.west);
\draw[->] ([yshift=-\TMP]a.east) to node [below]{$\scriptstyle\id$} ([yshift=-\TMP]b.west);

\end{tikzpicture}\]
The top map takes a pair $(\alpha, x)\in \Omega\Sigma X_0 \times  \prod_{i=1}^{\infty} X_i $ to the pair 
\[
\Big( \eta \circ f_1 \circ \pr_1(x)  \star  \Sigma (g) \circ \alpha, q \circ \prod_{i\geq 1} f_i (x)\Big).
\] 
A point in the homotopy equalizer is therefor a loop $\alpha \in \Omega\Sigma X_0$, a point $x\in  \prod_{i=1}^{\infty} X_i $ and paths
\[ 
\eta \circ f_1 \circ \pr_1(x)  \star  \Sigma (g) \circ \alpha   \sim\alpha, \quad q \circ \prod_{i\geq 1} f_i (x) \sim x.
\]
Thus there is a homotopy equivalence from the homotopy pullback to the homotopy equalizer.  It follows from Lemma \ref{triangle HE} that there is a distinguished triangle of the form
\begin{align*}
\HE(\Omega \Sigma g,\id) \xrightarrow{} \HE((\eta \circ f_1 \star \Sigma g,f_2,f_3,\dots),\id) \xrightarrow{} \\ \HE(q \circ \prod f_i, \id) 
 \xrightarrow {- \Sigma(\iota) \Acirc \varepsilon^{-1} \Acirc \eta \Acirc f_1 \Acirc \pr_1}  \HE(\Omega \Sigma g,\id),
\end{align*}
where the first map is induced by the inclusion and the second map is induced by the projection and 
\[
\iota: \Omega \big( \Omega \Sigma (X_0)\big) \to \HE(\Omega \Sigma(g),\id)
\] 
was defined above when we described the construction of homotopy equalizers. Note that we have a commutative square
 \[
 \begin{tikzpicture}
 \setlength{\TMP}{3pt}
\node(a){$\Omega \Sigma (X_0)$}; \node(b)[right of = a,node distance = 4cm]{$ X_0$};
\node(c)[below of = a,node distance = 2cm]{$\Omega \Sigma (X_0)$};
\node(d)[right of = c,node distance = 4cm]{$X_0$};

\draw[->](b) to node [above]{$\eta$} (a);
\draw[->](b) to node [below]{$\sim$} (a);

\draw[->](d) to node [above]{$\eta$} (c);
\draw[->](d) to node [below]{$\sim$} (c);

\draw[->] ([xshift=\TMP]a.south) to node [right]{$\id$} ([xshift=\TMP]c.north);
\draw[->] ([xshift=-\TMP]a.south) to node [left]{$\Omega \Sigma (g)$} ([xshift=-\TMP]c.north);

\draw[->] ([xshift=\TMP]b.south) to node [right]{$\id$} ([xshift=\TMP]d.north);
\draw[->] ([xshift=-\TMP]b.south) to node [left]{$g$} ([xshift=-\TMP]d.north);

\end{tikzpicture}\]
Thus $\eta$ induces an $\pi_*$-isomorphism of homotopy equalizers. The triangle above simplifies to the following triangle under this identification:
\begin{align*}
\HE( g,\id) \xrightarrow{I} \HE((\eta \circ f_1 \star \Sigma g,f_2,f_3,\dots),\id) \xrightarrow{P} \\ \HE(q \circ \prod_{i\geq 1} f_i, \id) 
 \xrightarrow {- \Sigma(\iota) \Acirc \varepsilon^{-1}\Acirc f_1 \Acirc \pr_1}  \HE( g,\id),
\end{align*}
where $I$ is induced by 
\[
\incl \circ  \eta :X_0 \to \Omega \Sigma (X_0) \times \prod_{i=1}^\infty X_i,
\] 
$P$ is induced by the projection 
\[
\proj: \Omega \Sigma (X_0) \times \prod_{i=1}^\infty X_i \to \prod_{i=1}^\infty X_i,
\] 
and $\iota: \Omega X_0 \to  \HE( g,\id)$. Finally we note that 
$ \HE(q \circ \prod_{i\geq 1} f_i, \id)$ is a model for the homotopy limit $\holim_{i\geq 1} X_i $ and we obtain the following theorem:
\begin{theorem} \label{teknisk result}
The triangle
\[\HE(g,\id)  \xrightarrow{I} \HE((\eta \circ f_1 \star \Sigma g,f_2,f_3,\dots),\id) \xrightarrow{ P} \holim_{i\geq 1} X_i \xrightarrow{- \Sigma(\iota) \Acirc \varepsilon^{-1}\Acirc f_1 \Acirc \pr_1} \Sigma \HE(g,\id)\]
is distinguished.
\end{theorem}

\subsection{Real topological cyclic homology}
We fix a prime $p$ and we let $\mathscr{F}_p$ denote the family of $O(2)$-subgroups generated by the subgroups $D_{p^n}$ and $C_{p^n}$ for all $n\geq 0$. Let $X$ be a $\mathscr{F}_p$-fibrant $O(2)$-cyclotomic spectrum. We let $R$ denote the $O(2)$-equivariant composition 
\[ R: \rho_{p}^*X^{C_p} \xrightarrow{\rho_{p}^*(\gamma)} \rho_{p}^* X^{gC_p} \xrightarrow{T_p} X.\]
Let $i: G \hookrightarrow O(2)$ denote the inclusion. We will let $X^{C_{p^n}}$ denote the underlying $G$-spectrum indexed on $i^*\mathcal{U}$. We have an isomorphism of orthogonal $G$-spectra
\[
(\rho_p^* X^{C_p})^{C_{p^{n-1}}} \cong X^{C_{p^n}},
\]
and we define a map of orthogonal $G$-spectra
\[
R_n: X^{C_{p^n}} \cong (\rho_p^* X^{C_p})^{C_{p^{n-1}}} \xrightarrow{R^{C_{p^{n-1}}}} X^{C_{p^{n-1}}}.
\]

\begin{remark}\label{splitting fib}
We let $X$ be an $O(2)$-cyclotomic spectrum and let $j_f: X \xrightarrow{\sim} X_f$ be a fibrant replacement and consider the following diagram:

 \[\begin{tikzpicture}
 
 \node(c)[left of = a,node distance = 4cm]{$\rho_{p}^*(X^c)^{gC_p}$};
\node(d)[below of = c,node distance = 2cm]{$\rho_{p}^*((X_f)^c)^{gC_p}$};
 \node(a){$\rho_{p}^*X^{gC_p}$};
\node(b)[below of = a,node distance = 2cm]{$\rho_{p}^*(X_f)^{gC_p}$};
\node(e)[right of = a,node distance = 4cm]{$X$};
\node(f)[below of = e, node distance = 2cm]{$X_f$};

\draw[->](c) to node [left]{$\rho_{p}^*((j_f)^c)^{gC_p}$} (d);
\draw[->](c) to node [right]{$\sim$} (d);

\draw[->](a) to node [left]{$\rho_{p}^*(j_f)^{gC_p}$} (b);
\draw[->](a) to node [right]{$\sim$} (b);

\draw[->](a) to node [above]{$T_p$} (e);

\draw[->](e) to node [left]{$j_f$} (f);
\draw[->](e) to node [right]{$\sim$} (f);

\draw[->,dashed](b) to node [below right]{$\hat{T}_p$} (f);

\draw[->](c) to node [above]{} (a);
\draw[->](d) to node [below]{} (b);

\end{tikzpicture}\]
Even though $X_f$ is not cyclotomic in the sense of Definition \ref{defcy}, we do have maps $\hat{T}_p: \rho_{p}^*(X_f)^{gC_p} \to X_f$ in the $O(2)$-stable homotopy category. We let $R$ denote the composition
\[
R: \rho_{p}^*(X_f)^{C_p}  \xrightarrow{\rho_p^*(\gamma)} \rho_{p}^*(X_f)^{gC_p} \xrightarrow{\hat{T}_p} X
\] 
We get maps in the $O(2)$-stable homotopy category
 \begin{align*} 
R_n &: X_f^{C_{p^n}} \to X_f^{C_{p^{n-1}}}
\end{align*}
by mimicking the above construction.
\end{remark}

In order to define the real topological cyclic homology at a prime $p$, we let 
\[
\TRR^n(A,D;p)=\THR(A, D)^{C_{p^n}},
\]
and define the $G$-spectrum $\TRR(A,D;p)$ to be the homotopy limit over the $R_n$ maps,  
\[\TRR(A, D;p):=\holim_{n,R_n} \TRR^n(A, D;p).\]
The inclusions of fixed points
\[F_n: X^{C_{p^n}} \to X^{C_{p^{n-1}}},\]
which we refer to as the Frobenius maps, are $G$-equivariant and induce a self map of $\TRR(A,D;p)$. In order to describe this map, we let $N_0$ denote the category
\[
\cdots \to n \to n-1 \to \cdots \to 2 \to 1 \to 0,
\]
and we let $\TRR^{(-)}(A,D;p): N_0 \to \Top_*$ denote the functors which maps $n \to n-1$ to $\THR(A, D)^{C_{p^n}} \xrightarrow{R_n} \THR(A, D)^{C_{p^{n-1}}} $. Let $\tau: N_0 \to N_0$ denote the translation functor $\tau(n)=n+1$. The Frobenius maps $F_n$ assemble into a natural transformation 
\[
F: \TRR^{(-)}(A, D;p) \circ \tau \Rightarrow \TRR^{(-)}(A, D;p),\] 
and we let $\varphi$ denote the composite
\[
\holim_{N_0} \TRR^n(A, D;p) \xrightarrow{\ind_{\tau}} \holim_{N_0} \TRR^{n+1}(A, D;p) \circ \tau  \xrightarrow{F} \holim_{N_0} \TRR^n(A, D;p).
\] 
\begin{definition}
The real topological cyclic homology  at $p$, $\TCR(A, D;p)$, is the homotopy equalizer, $\HE(\varphi,\id)$, of the diagram
 \[\begin{tikzpicture}
\setlength{\TMP}{2.5pt}
\node(a){$\TRR(A,D;p)$};
\node(b)[right of=a, node distance = 3.5cm]{$\TRR(A,D;p).$};
\draw[->] ([yshift=\TMP]a.east) to node [above]{$\scriptstyle \varphi$} ([yshift=\TMP]b.west);
\draw[->] ([yshift=-\TMP]a.east) to node [below]{$\scriptstyle\id$} ([yshift=-\TMP]b.west);

\end{tikzpicture}\]
\end{definition}

We conclude this section by describing the homotopy fiber of the restriction maps. Let $\mathcal{R}$ denote the family of subgroups of $O(2)$ consisting of the trivial subgroup and all order 2 subgroups generated by a reflection of the plane in a line trough the origin;
\[
\mathcal{R}=\{1, \left\langle t\omega\right\rangle \mid t\in \mathbb{T}, \omega \in G \}. 
\]
Write $E\mathcal{R}$ for the classifying space of this family, thus $E \mathcal{R}$ is an $O(2)$-CW-complex such that
\[E\mathcal{R}^H= \left\{ \begin{array}{rl}
 * &\mbox{ if $H\in \mathcal{R}$} \\
  \emptyset &\mbox{ if $H\notin \mathcal{R}.$}
       \end{array} \right. \]
The space $S(\mathbb{C}^{\infty})$ is a model for $E\mathcal{R}$, with $O(2)$-acting on $S(\mathbb{C})$ by multiplication and complex conjugation and on $S(\mathbb{C}^{\infty})$ via the homeomorphism $S(\mathbb{C}) \star S(\mathbb{C}) \star \cdots \cong S(\mathbb{C}^{\infty})$. We let $\tilde{E\mathcal{R}}$ denote the mapping cone of the map $E\mathcal{R}_+ \xrightarrow{} S^0$ which collapses $E\mathcal{R}$ to the non-basepoint. We have a cofibration sequence of based $O(2)$-spaces,
\[E\mathcal{R}_+ \xrightarrow{} S^0 \to \tilde{E\mathcal{R}} \to \Sigma E\mathcal{R}_+.\]
and we obtain a distinguished triangle of orthogonal $G$-spectra by smashing the sequence with an $O(2)$-spectrum $X$ and taking derived $C_{p^n}$-fixpoints.
\begin{align}\label{cof}
(E\mathcal{R}_+ \wedge X)_f^{C_{p^n}} \xrightarrow{} X_f^{C_{p^n}} \xrightarrow{\lambda^{C_{p^n}}} (\tilde{E\mathcal{R}} \wedge X)_f^{C_{p^n}} \to \Sigma (E\mathcal{R}_+ \wedge X)_f ^{C_{p^n}}.
\end{align}
We denote the left hand spectrum by $\operatorname{H}_{\boldsymbol{\cdot}}(C_{p^n}; X)$ and refer to it as the $G$-equivariant homology spectrum of the subgroup $C_{p^n}$ acting on the $O(2)$-spectrum $X$. The third term in the triangle identifies with the $G$-equivariant derived $C_{p^n}$-geometric fixed points; see \cite[Prop. 4.17]{MM}, giving the well-known isotropy separation sequence. 

\begin{lemma}\label{cofiber seq}
Let $X$ be a $O(2)$-cyclotomic spectrum. The triangle
\[\operatorname{H}_{\boldsymbol{\cdot}}(C_{p^n}; X) \xrightarrow{} X_f^{C_{p^n}} \xrightarrow{R_n} X_f^{C_{p^{n-1}}} \to \Sigma \operatorname{H}_{\boldsymbol{\cdot}}(C_{p^n}; X) ,\]
is distinguished in the $G$-stable homotopy category.
\end{lemma}

\begin{proof}
The result follows immediately using the distinguished triangle (\ref{cof}), the identification of $(\tilde{E\mathcal{R}} \wedge X)_f^{C_{p^n}}$ with the $G$-equivariant derived $C_{p^n}$-geometric fixed points and the $\pi_*$-isomorphism of $G$-spectra  $(X^c)^{gC_p} \to X^{gC_p} \xrightarrow{T_p} X $.
\end{proof}

\begin{remark}\label{G-equi}
Let $f: X \to Y$ be a map of $O(2)$-spectra. If both $X$ and $Y$ are cyclotomic and $f$ commutes with the cyclotomic structure maps, then $f$ commutes with $R$. If $f$ restricts to a $\pi_*$-isomorphism of $G$-spectra, then by induction using the distinguished triangle above, $f$ is a $\mathscr{F}_p$-equivalence. 
\end{remark}

\section{Spherical group rings}
 In this section we determine the real topological Hochschild homology of the spherical group ring $\mathbb{S}[\Gamma]$ of a topological group $\Gamma$ with anti-involution $\id[\Gamma]$ induced by taking inverses in the group. We then determine the $G$-homotopy type of $\TCR(\mathbb{S}[\Gamma], \id[\Gamma]; p)$ where $p$ is a prime. This is a generalization of results by Bökstedt-Hsiang-Madsen in \cite[Section 5]{BHM}; see also \cite[Section 4.4]{Ma}. 

\begin{theorem} \label{THR(sphere)}
Let $\Gamma$ be a topological group. There is a map of $O(2)$-orthogonal spectra
\[
i: \Sigma^\infty_{O(2)} \Bdi\Gamma_+ \to \THR(\mathbb{S}[\Gamma], \id[\Gamma]),
\]
commuting with the cyclotomic structures, which induces isomorphisms on $\pi_*^{C_{p^n}}(-)$ and $\pi_*^{D_{p^n}}(-)$ for all $n\geq 0$ and all primes $p$. 
\end{theorem}

\begin{proof}
Let $V$ be an $O(2)$-representation and let $F\subset \mathbb{T}$ be a finite subset. We define the map $i_V[F]$ to be the composition
\begin{align*}
\underset{z\in F}{\wedge} \Gamma_+ \wedge S^V &\cong \Map\Big(\underset{z\in F}{\wedge} S^0, \big(\underset{z\in F}{\wedge} S^0\big) \wedge \big(\underset{z\in F}{\wedge} \Gamma_+\big) \wedge S^V \Big) \\ 
&\to \underset{I^F}{\hocolim} \ \Map\Big(\underset{z\in F}{\wedge} S^{i_z}, \big(\underset{z\in F}{\wedge} S^{i_z} \big) \wedge \big(\underset{z\in F}{\wedge} \Gamma_+\big) \wedge S^V \Big)\\
&\cong \underset{I^F}{\hocolim} \ \Map\Big(\underset{z\in F}{\wedge} S^{i_z}, \big(\underset{z\in F}{\wedge} (S^{i_z} \wedge \Gamma_+)\big)  \wedge S^V \Big), 
\end{align*}
where the last map is induced by the natural transformation given by permuting the smash factors of the target. These maps commute with the dihedral structure and the $O(2)$-action on $V$, thus we obtain an $O(2)$-equivariant map on realizations:
\[i_V: \Bdi\Gamma_+  \wedge S^V \to \THR(\mathbb{S}[\Gamma], \id[\Gamma])(V).  \]
The map $i$ commutes with the cyclotomic structure, hence by Remark \ref{G-equi}, it suffices to check that $i$ restricts to a $\pi_*$-isomorphism of $G$-spectra. It follows from \cite[Chapter XVI, Thm. 6.4]{May96}, that it suffices to show that $i$ induces an isomorphism on $\pi_*((-^c)^{gC_p})$  and $\pi^e_*(-)$, hence we must show that the connectivity of the induced map 
\[(i_V)^H: (\Bdi\Gamma_+  \wedge S^V)^H \to (\THR(\mathbb{S}[\Gamma], \id[\Gamma])(V))^H,  \]
is $(\dim(V^H) + \epsilon(V))$-connected, where $H\in \{e,G\}$ and $\epsilon(V)$ tends to infinity with $V$. Let $i\in \ob\big((I^{G})^{\{1, -1\}}\times I^{GF/F^-}\big) $. After we restrict to the cofinal subcategory $G\mathcal{F}_*$ , then the map $i_V$ in simplicial level $G\cdot F$ is equal to the composite: 
\[
\begin{tikzpicture}
\node(a){$ \underset{z\in G \cdot F}{\wedge} \Gamma_+ \wedge S^V$};
\node(c)[below of= a,node distance = 3cm]{$ \underset{I^{G \cdot F}}{\hocolim} \ \Map\Big(\underset{z\in G \cdot F}{\wedge} S^{i_{\overline{z}}}, \underset{z\in G \cdot F}{\wedge} S^{i_{\overline{z}}} \wedge S^V \wedge \underset{z\in G \cdot F}{\wedge} \Gamma_+\Big).$};
\node(e)[below of=a, node distance = 1.5cm]{$\Map\Big(\underset{z\in G \cdot F}{\wedge} S^{i_{\overline{z}}}, \underset{z\in G \cdot F}{\wedge} S^{i_{\overline{z}}} \wedge S^V \wedge \underset{z\in G \cdot F}{\wedge} \Gamma_+\Big)$};

\draw[->](a) to node [right]{$\eta$} (e);

\draw[->](e) to node [above]{} (c);

\end{tikzpicture}\]
The top map is the adjunction unit, which by the Equivariant Suspension Theorem \ref{suspension} is at least $2\cdot \dim(V)-1$ connected as a map of non-equivariant spaces and $2 \cdot \dim(V^G)-1$ connected on $G$-fixed points. By Lemma \ref{subcat} we can make the lower map as connected as desired as a map of $G$-spaces by choosing $i\in \ob\big ( (I^{G})^{\{1, -1\}}\times I^{GF/F^-}\big)$ big enough. Thus the composite has the desired connectivity.
\end{proof}

The calculation of $\TCR(\mathbb{S}[\Gamma], \id[\Gamma];p)$ relies on the fact that the restriction map
\[
R_n: \THR(\mathbb{S}[\Gamma], \id[\Gamma])^{C_{p^n}} \to \THR(\mathbb{S}[\Gamma], \id[\Gamma])^{C_{p^{n-1}}}
\]
splits. In \cite[Lemma 6.2.5.1]{DGM}, a splitting of $R_n$ is constructed and it is straight forward to check that the splitting is indeed $G$-equivariant. We denote this section $S_n$.
Recall that we defined $O(2)$-equivariant homeomorphisms $p_r: \Bdi\Gamma \to \rho_r^*\Bdi\Gamma^{C_r}$. We let $\Delta_r$ denote the $G$-equivariant composition
\[
\Bdi\Gamma \xrightarrow{p_r} \Bdi\Gamma^{C_r} \hookrightarrow{}\Bdi\Gamma.
\]
There is a commutative diagram in the $G$-stable homotopy category where $i$ is the $\mathscr{F}_p$-equivalence of Theorem \ref{THR(sphere)}; see \cite[Corollary 4.4.11]{Ma}
 \[
 \begin{tikzpicture}
\node(a){$ \THR(\mathbb{S}[\Gamma], \id[\Gamma])$}; 
\node(b)[right of = a,node distance = 4cm]{$\Sigma^{\infty}_{O(2)} \Bdi\Gamma_+ $};
\node(e)[below of = a,node distance = 1.8cm]{$\THR(\mathbb{S}[\Gamma], \id[\Gamma])$};
\node(f)[right of = e,node distance = 4cm]{$\Sigma^{\infty}_{O(2)} \Bdi\Gamma_+ $};

\draw[->](b) to node [above]{$\sim$} (a);
\draw[->](b) to node [below]{$i$} (a);

\draw[->](f) to node [above]{$\sim$} (e);
\draw[->](f) to node [below]{$i$} (e);

\draw[->] (a) to node [left]{$F_1 \circ S_1$} (e);

\draw[->] (b) to node [right]{$\Sigma^{\infty} {\Delta_p}_+$} (f);

\draw[->](f) to node [above]{$\sim$} (e);

\end{tikzpicture}
\]
We set $T:=\THR(\mathbb{S}[\Gamma], \id[\Gamma])$ to ease notation and we recall the notation for the $G$-equivariant homology spectrum
\[\mathrm{H}_{\boldsymbol{\cdot}}(C_{p^{n}}; \Bdi\Gamma)=
(E\mathcal{R}_+ \wedge \Sigma^{\infty}_{O(2)} \Bdi\Gamma_+ )_f^{C_{p^{n}}}, \quad n\geq 1. \] 
If $n=0$, then we let 
\[\mathrm{H}_{\boldsymbol{\cdot}}(1; \Bdi\Gamma) := 
\Sigma^{\infty}_{G} \Bdi\Gamma_+.\]
Let $c: E\mathcal{R}_+ \wedge \Sigma^{\infty}_{O(2)} \Bdi\Gamma_+ \xrightarrow{} \Sigma^{\infty}_{O(2)} \Bdi\Gamma_+$ denote the map that collapses $E\mathcal{R}$ to a point. We let $c_n$ denote the composition
\[
\mathrm{H}_{\boldsymbol{\cdot}}(C_{p^{n}}; \Bdi\Gamma) \xrightarrow{(c_f)^{C_{p^n}}} \big(\Sigma^{\infty}_{O(2)} \Bdi\Gamma_+\big)_f^{C_{p^n}}  \xrightarrow{} T^{C_{p^n}}.
\]
where the last map is induced by the $\mathscr{F}_p$-equivalence $i$ of Theorem \ref{THR(sphere)} and fibrant replacement. By Lemma \ref{cofiber seq} we have distinguished triangles of $G$-spectra
\begin{align*}
\operatorname{H}_{\boldsymbol{\cdot}}(C_{p^{n}}; \Bdi\Gamma)\xrightarrow{ c_n} T^{C_{p^{n}}} \xrightarrow{R_n} T^{C_{p^{n-1}}} \to \Sigma \operatorname{H}_{\boldsymbol{\cdot}}(C_{p^{n}}; \Bdi\Gamma)
\end{align*}
which split and provide an isomorphism in the $G$-stable homotopy category
\[
S_n \circ \cdots \circ S_1 \circ i \vee   \cdots \vee S_n \circ c_{n-1} \vee c_n:\bigvee_{j=0}^{n} \operatorname{H}_{\boldsymbol{\cdot}}(C_{p^j}; \Bdi\Gamma) \xrightarrow{\sim}T^{C_{p^{n}}}.
\]
There is a commutative diagram, where the projection maps collapse the $n$th summand to the basepoint:
\[
 \begin{tikzpicture}
\node(a){$ T^{C_{p^{n}}}$}; 
\node(b)[right of = a,node distance = 4cm]{$ \bigvee_{j=0}^{ n} \operatorname{H}_{\boldsymbol{\cdot}}(C_{p^j}; \Bdi\Gamma) $};
\node(c)[right of = b,node distance = 5cm]{$ \prod_{j=0}^{ n} \operatorname{H}_{\boldsymbol{\cdot}}(C_{p^j}; \Bdi\Gamma) $};
\node(e)[below of = a,node distance = 2cm]{$ T^{C_{p^{n-1}}} $};
\node(f)[right of = e,node distance = 4cm]{$ \bigvee_{j=0}^{ n-1} \operatorname{H}_{\boldsymbol{\cdot}}(C_{p^j}; \Bdi\Gamma) $};
\node(g)[right of = f,node distance = 5cm]{$ \prod_{j=0}^{ n-1} \operatorname{H}_{\boldsymbol{\cdot}}(C_{p^j}; \Bdi\Gamma) $};

\draw[->](b) to node [above]{$\sim$} (a);

\draw[->] (a) to node [right]{$R_n$} (e);

\draw[->] (b) to node [right]{$ \proj$} (f);

\draw[->](f) to node [above]{$\sim$} (e);

\draw[->] (b) to node [above]{$\sim$} (c);
\draw[->] (f) to node [above]{$\sim$} (g);
\draw[->] (c) to node [right]{$\proj$} (g);
\end{tikzpicture}
\]
Combined with the canonical map from the limit to the homotopy limit we obtain a canonical isomorphism in the $G$-stable homotopy category  
\[
 \TRR(\mathbb{S}[\Gamma], \id[\Gamma];p) \sim  \prod_{j=0}^{\infty} \operatorname{H}_{\boldsymbol{\cdot}}(C_{p^j}; \Bdi\Gamma).
 \] 
 We proceed to identify the Frobenius maps $F_n$. We have inclusions
 \[
 \incl: (E\mathcal{R}_+ \wedge \Sigma^{\infty}_{O(2)} \Bdi\Gamma_+ )_f^{C_{p^{n}}}\to (E\mathcal{R}_+ \wedge \Sigma^{\infty}_{O(2)} \Bdi\Gamma_+ )_f^{C_{p^{n-1}}}.
 \] 
 We abuse notation slightly and let $c$ denote the composite
\[
 (E\mathcal{R}_+ \wedge \Sigma^{\infty}_{G} \Bdi\Gamma_+ )_f \xrightarrow{j_f^{-1}}  E\mathcal{R}_+ \wedge  \Sigma^{\infty}_{G} \Bdi\Gamma_+ \xrightarrow{c} \Sigma^{\infty}_{G} \Bdi\Gamma_+,
\] 
where $j_f$ is fibrant replacement. There is a commutative diagram:
 \[
           \hspace*{2.4cm}
 \begin{tikzpicture}
\node(a){$ T^{C_{p^{n}}}$}; \node(b)[right of = a,node distance = 5cm]{$ \Sigma^{\infty}_{G} \Bdi\Gamma_+  \vee  \bigvee_{j=1}^{ n} \operatorname{H}_{\boldsymbol{\cdot}}(C_{p^j}; \Bdi\Gamma) $};
\node(e)[below of = a,node distance = 2cm]{$ T^{C_{p^{n-1}}} $};
\node(f)[right of = e,node distance = 5cm]{$ \Sigma^{\infty}_{G} \Bdi\Gamma_+  \vee  \bigvee_{j=1}^{n-1} \operatorname{H}_\bullet(C_{p^j}; \Bdi\Gamma). $};

\draw[->](b) to node [above]{$\sim$} (a);

\draw[->] (a) to node [left]{$F_n$} (e);

\draw[->] (b) to node [right]{$ \Sigma^{\infty}{\Delta_p}_+  \vee c \circ \incl \vee \incl \vee \cdots \vee \incl$} (f);

\draw[->](f) to node [above]{$\sim$} (e);

\end{tikzpicture}
\]
We can now identify the self-map $\varphi$ of $\TRR(\mathbb{S}[\Gamma], \id[\Gamma];p)$ induced by the Frobenius maps $F_n$. We have the following commutative diagram, where $\TCR(\mathbb{S}[\Gamma], \id[\Gamma];p)$ is the homotopy equalizer of the two middle maps:
\adhoc{.97}{
\begin{tikzpicture}
\setlength{\TMP}{3pt}
\node(a){$\Sigma^{\infty}_{G} \Bdi\Gamma_+$}; 
\node(b)[right of = a,node distance = 5.5cm]{$ \displaystyle \Omega \Sigma \big(\Sigma^{\infty}_{G} \Bdi\Gamma_+\big) \times   \prod_{j=1}^{\infty}\operatorname{H}_{\boldsymbol{\cdot}}(C_{p^j}; \Bdi\Gamma)$};
\node(c)[right of = b,node distance = 5.8cm]{$ \displaystyle  \prod_{j=1}^{\infty}\operatorname{H}_{\boldsymbol{\cdot}}(C_{p^j}; \Bdi\Gamma)$};

\node(f)[below of = a,node distance = 2.5cm]{$\Sigma^{\infty}_{G} \Bdi\Gamma_+$};
\node(g)[right of = f,node distance = 5.5cm]{$\displaystyle  \Omega \Sigma \big(\Sigma^{\infty}_{G} \Bdi\Gamma_+\big)   \times   \prod_{j=1}^{\infty}\operatorname{H}_{\boldsymbol{\cdot}}(C_{p^j}; \Bdi\Gamma)$};
\node(h)[right of = g,node distance = 5.8cm]{$  \displaystyle  \prod_{j=1}^{\infty}\operatorname{H}_{\boldsymbol{\cdot}}(C_{p^j}; \Bdi\Gamma)$};

\draw[->] ([xshift=\TMP]a.south) to node [right]{$\id$} ([xshift=\TMP]f.north);
\draw[->] ([xshift=-\TMP]a.south) to node [left]{$\Sigma^{\infty}{\Delta_p}_+$} ([xshift=-\TMP]f.north);

\draw[->] ([xshift=\TMP]b.south) to node [right]{$\id$} ([xshift=\TMP]g.north);
\draw[->] ([xshift=-\TMP]b.south) to node [left]{$X$} ([xshift=-\TMP]g.north);

\draw[->] ([xshift=\TMP]c.south) to node [right]{$\id$} ([xshift=\TMP]h.north);
\draw[->] ([xshift=-\TMP]c.south) to node [left]{$q\circ \prod_{i\geq 1} \incl$} ([xshift=-\TMP]h.north);

\draw[->](a) to node [above]{$\incl \circ \eta$} (b);
\draw[->](b) to node [above]{$\proj$} (c);

\draw[->] (f) to node [below]{$\incl \circ \eta$} (g);
\draw[->] (g) to node [below]{$\proj$} (h);
\end{tikzpicture}
}
where $X=((\eta \circ c \circ \incl \circ \pr_1) \star\Sigma(\Sigma^{\infty}{\Delta_p}_+), \incl, \dots , \incl)$ and the first map in the tuple takes a pair
\[
(\alpha, x) \in \Omega \Sigma \big(\Sigma^{\infty}_{G} \Bdi\Gamma_+\big)   \times   \prod_{j=1}^{\infty}\operatorname{H}_{\boldsymbol{\cdot}}(C_{p^j}; \Bdi\Gamma)
\]
to the loop
\[
\big(\eta \circ c \circ \incl \circ \pr_1(x)\big) \star \big( \Sigma(\Sigma^{\infty}{\Delta_p}_+) \circ \alpha \big) \in \Omega \Sigma \big(\Sigma^{\infty}_{G} \Bdi\Gamma_+\big) .
\]
We let $\Sigma^{\infty}_{G} \Bdi\Gamma_+^{\eqq{\Delta_p=\id}} $ denote the homotopy equalizer of $\Sigma^{\infty}{\Delta_p}_+$ and the identity,
\[
\Sigma^{\infty}_{G} \Bdi\Gamma_+^{\eqq{\Delta_p=\id}} := \HE(\Sigma^{\infty}{\Delta_p}_+, \id)
.\]
Recall that we constructed a map $\iota: \Omega (\Sigma^{\infty}_{G} \Bdi\Gamma_+) \to \Sigma^{\infty}_{G} \Bdi\Gamma_+^{\eqq{\Delta_p=\id}}$. The map $\incl \circ \eta$ and the projection induce the maps $I$ and $P$ in the following theorem, which follows directly from Theorem \ref{teknisk result}.

\begin{theorem}\label{result}
The triangle
\begin{align*}
\Sigma^{\infty}_{G} \Bdi\Gamma_+^{\eqq{\Delta_p=\id}}  \xrightarrow{I} \TCR(\mathbb{S}[\Gamma], \id[\Gamma]; p) \xrightarrow{P} \holim_{j\geq 1}\operatorname{H}_{\boldsymbol{\cdot}}(C_{p^j};  \Bdi\Gamma) \\ 
\xrightarrow{- \Sigma(\iota) \Acirc \varepsilon^{-1}\Acirc c \Acirc \incl \Acirc \pr_1} \Sigma \left(\Sigma^{\infty}_{G} \Bdi\Gamma_+^{\eqq{\Delta_p=\id}}\right),
\end{align*}
is distinguished in the $G$-stable homotopy category, where $\varepsilon: \Sigma \Omega Y \to Y$ is the counit of the loop-suspension adjunction.
\end{theorem}

\begin{remark}
The inclusion of index categories induces a canonical $\pi_*$-isomorphism 
\[
\holim_{j\geq 0} \operatorname{H}_{\boldsymbol{\cdot}}(C_{p^j};  \Bdi\Gamma)   \xrightarrow{\text{can}} \holim_{j\geq 1} \operatorname{H}_{\boldsymbol{\cdot}}(C_{p^j};  \Bdi\Gamma) 
\]
and the composite $\incl \circ \pr_1$ in the triangle can be replaced by $\pr_0 \circ \text{can}^{-1}$.
\end{remark}

A non-equivariant identification of the homotopy limit after $p$-completion appears in \cite{BHM} and in more generality in \cite[Lemma 4.4.9]{Ma}. The result generalizes immediately to the equivariant setting. Let $X$ be an $O(2)$-spectrum. The $O(2)$-spectrum $E\mathcal{R}_+ \wedge X$ is $C_{p^n}$-free. It follows from the generalized Adams isomorphism \cite[Chapter XVI, Thm. 5.4]{May96} that there are isomorphisms in the $G$-stable homotopy category
\[E\mathcal{R}_+ \wedge_{C_{p^n}} X \xrightarrow{\sim} (E\mathcal{R}_+ \wedge X)_f^{C_{p^n}}\] 
and under this isomorphism the inclusions of fixed points on the right hand side correspond to the $G$-equivariant transfers on the left hand side 
 \[
 \begin{tikzpicture}
\node(a){$E\mathcal{R}_+ \wedge_{C_{p^n}} X$}; \node(b)[right of = a,node distance = 4cm]{$ (E\mathcal{R}_+ \wedge X)_f^{C_{p^n}} $};
\node(e)[below of = a,node distance = 2cm]{$E\mathcal{R}_+ \wedge_{C_{p^{n-1}}} X$};
\node(f)[right of = e,node distance = 4cm]{$ (E\mathcal{R}_+ \wedge X)_f^{C_{p^{n-1}}}.$};

\draw[->](a) to node [above]{$\sim$} (b);

\draw[->] (a) to node [left]{$\trf$} (e);

\draw[->] (b) to node [right]{$\incl$} (f);

\draw[->](e) to node [above]{$\sim$} (f);

\end{tikzpicture}
\]
We obtain an isomorphism in the $G$-stable homotopy category
\[
\holim_{\trf} E\mathcal{R}_+ \wedge_{C_{p^i}} X \xrightarrow{\sim} \holim_{i\geq 1} \operatorname{H}_{\boldsymbol{\cdot}}(C_{p^i}; X),
\]
and we can identify the left hand homotopy limit after $p$-completion:

\begin{theorem}\label{transfer}
Let $X$ be an orthogonal $O(2)$-spectrum. The $\mathbb{T}$-transfer
\[\trf_{\mathbb{T}}: \Sigma^{1,1}  E\mathcal{R}_+ \wedge_{\mathbb{T}} X  \to  \holim_{\trf} E\mathcal{R}_+ \wedge_{C_{p^i}} X,\]
induces an isomorphism in the $G$-stable homotopy category after $p$-completion, where $\Sigma^{1,1}$ denotes suspension with the sign representation of $G$.
\end{theorem}

\begin{proof}
Recall that $S(\mathbb{C}^{\infty})$ is a model for $E\mathcal{R}$ and that $O(2)$ acts on $S(\mathbb{C})$ by multiplication and complex conjugation and on $S(\mathbb{C}^{\infty})$ via the homeomorphism 
\[
S(\mathbb{C}) \star S(\mathbb{C}) \star \cdots \cong S(\mathbb{C}^{\infty}).
\] 
We filter $S(\mathbb{C}^{\infty})$ by the $O(2)$-subspaces $S(\mathbb{C}^{k})$ and obtain the diagram
\[\begin{tikzpicture}
\node(a){$\Sigma^{1,1} S(\mathbb{C}^{k})_+ \wedge_{\mathbb{T}} X$}; \node(b)[right of = a,node distance = 5cm]{$\Sigma^{1,1} S(\mathbb{C}^{k+1})_+ \wedge_{\mathbb{T}} X$};
\node(c)[right of = b,node distance = 5cm]{$\Sigma^{1,1} S(\mathbb{C}^{k+1})/S(\mathbb{C}^{k}) \wedge_{\mathbb{T}} X$};
\node(f)[below of = a,node distance = 1.5cm]{$S(\mathbb{C}^{k})_+ \wedge_{C_{p^n}} X$};
\node(g)[right of = f,node distance = 5cm]{$S(\mathbb{C}^{k+1})_+ \wedge_{C_{p^n}} X$};
\node(h)[right of = g,node distance = 5cm]{$S(\mathbb{C}^{k+1})/S(\mathbb{C}^{k}) \wedge_{C_{p^n}} X$};

\draw[->](a) to node [above]{} (b);
\draw[->](b) to node [above]{} (c);

\draw[->](a) to node [left]{$\trf_{\mathbb{T}}$} (f);
\draw[->](b) to node [left]{$\trf_{\mathbb{T}}$} (g);
\draw[->](c) to node [left]{} (h);

\draw[->] (f) to node [below]{} (g);
\draw[->] (g) to node [below]{} (h);

\end{tikzpicture}\]
By induction it suffices to show that the outer right vertical map induces isomorphism after $p$-completion on the homotopy limit of the transfers
\[ \trf: S(\mathbb{C}^{k+1})/S(\mathbb{C}^{k}) \wedge_{C_{p^n}} X \to S(\mathbb{C}^{k+1})/S(\mathbb{C}^{k}) \wedge_{C_{p^{n-1}}} X.
\]
In order to identify the right vertical map, we first note the following general fact. Let $Y$ be a pointed $O(2)$-space and let $i: G \to O(2)$ be the inclusion. If we let $i(Y)$ denote the $\mathbb{T}$-trivial $O(2)$-space whose underlying $G$-space is $i^*Y$, then there is an $O(2)$-equivariant homeomorphism 
\[S(\mathbb{C})_+\wedge Y \xrightarrow{\cong} S(\mathbb{C})_+ \wedge i(Y), \quad (z,y) \mapsto (z,z^{-1}y).\]
Lemma \ref{dimension} provides a homeomorphism of $O(2)$-spaces 
\[
S(\mathbb{C}^{k+1})/S(\mathbb{C}^{k})\cong  S(\mathbb{C})_+ \wedge \Sigma S(\mathbb{C}^{k})\] 
and it follows that the right vertical map in the diagram identifies with 
\[\begin{tikzpicture}
\node(a){$\Sigma^{1,1} S(\mathbb{C})/\mathbb{T}_+ \wedge i(\Sigma S(\mathbb{C}^{k}) \wedge X)$}; 
\node(f)[right of = a,node distance = 7cm]{$S(\mathbb{C})/C_{p^n_+} \wedge  i(\Sigma S(\mathbb{C}^{k}) \wedge X),$};

\draw[->](a) to node [above]{$ \tau_{\infty} \wedge \id$} (f);

\end{tikzpicture}
\]
where $\tau_{\infty}:  \Sigma^{1,1}  S(\mathbb{C})/\mathbb{T}_+ \to   S(\mathbb{C})/C_{p^n +}$ is the $G$-equivariant transfer. Likewise the transfer map
\[ \trf: S(\mathbb{C}^{k+1})/S(\mathbb{C}^{k}) \wedge_{C_{p^n}} X \to S(\mathbb{C}^{k+1})/S(\mathbb{C}^{k}) \wedge_{C_{p^{n-1}}} X,\]
identifies with 
\[
\begin{tikzpicture}
\node(a){$S(\mathbb{C})/C_{p^n_+} \wedge  i(\Sigma S(\mathbb{C}^{k}) \wedge X) $}; 
\node(f)[right of = a,node distance = 7cm]{$S(\mathbb{C})/C_{p^{n-1}_+} \wedge  i(\Sigma S(\mathbb{C}^{k}) \wedge X),$};

\draw[->](a) to node [above]{$ \tau_{n} \wedge \id$} (f);

\end{tikzpicture}
\]
where $\tau_{n}:  S(\mathbb{C})/C_{p^n +} \to  S(\mathbb{C})/C_{p^{n-1} +}$ is the $G$-equivariant transfer.

As an $O(2)$-space $S(\mathbb{C})_+=S^0 \vee S(\mathbb{C})$, with $O(2)$ acting trivially on $S^0$. If we identify $S(\mathbb{C})/C_{p^n}$ with $S(\mathbb{C})$ via the root isomorphism $\rho_{p^n}$, then $\tau_n=\id \vee p$. The result now follows exactly as in \cite[Lemma 4.4.9]{Ma} by first arguing that there is a distinguished triangle of the form  
\[\Sigma^{1,1}  i(\Sigma S^{2k-1} \wedge X) \xrightarrow{\tau}  \holim_{\tau_n} S(\mathbb{C})/C_{p^n +} \wedge i(\Sigma S^{2k-1} \wedge X) \to \holim_{p} i(\Sigma S^{2k-1} \wedge X) \to \Sigma.\]
The mod $p$ homotopy groups of the right hand term vanish. Hence the map $\tau$ induces an isomorphism after $p$-completion as desired.
\end{proof}

When $\Gamma$ is trivial, the distinguished triangle of Theorem \ref{result} simplifies to
\[\begin{tikzpicture}
\node(a){$\Omega \mathbb{S} \vee \mathbb{S}$}; \node(b)[right of = a,node distance = 2.5cm]{$\TCR(\mathbb{S}, \id ;p) $};
\node(c)[right of = b,node distance = 3.3cm]{$ \displaystyle \holim_{i\geq 0} \operatorname{H}_{\boldsymbol{\cdot}}(C_{p^i}; \mathbb{S}) $};
\node(d)[right of = c,node distance = 4.8cm]{$\Sigma \Omega \mathbb{S}  \vee \Sigma \mathbb{S}. $};

\draw[->](a) to node [above]{} (b);
\draw[->](b) to node [above]{} (c);
\draw[->](c) to node [above]{$-(\varepsilon^{-1 } \circ \pr_0)$} (d);

\end{tikzpicture}\]
If we rotate the distinguished triangle arising from the homotopy fiber of the projection $\pr_0: \holim_{i\geq 0 }  \operatorname{H}_{\boldsymbol{\cdot}}(C_{p^i}; \mathbb{S}) \xrightarrow{} \mathbb{S}$ given in Lemma \ref{fiber} and add the distinguished triangle $\mathbb{S} \xrightarrow{\id} \mathbb{S}\xrightarrow{} * \xrightarrow{} \Sigma \mathbb{S}$, then we obtain the distinguished triangle:
\[\begin{tikzpicture}
\node(f){$\Omega \mathbb{S} \vee \mathbb{S}$};
\node(g)[right of = f,node distance = 2.5cm]{$H_{\pr_0}\vee \mathbb{S}$};
\node(h)[right of = g,node distance = 3.3cm]{$\displaystyle \holim_{i\geq 0} \operatorname{H}_{\boldsymbol{\cdot}}(C_{p^i}; \mathbb{S})$};
\node(i)[right of = h,node distance = 4.8cm]{$\Sigma \Omega \mathbb{S}  \vee \Sigma \mathbb{S}. $};

\draw[->] (f) to node [below]{${}$} (g);
\draw[->] (g) to node [below]{} (h);
\draw[->] (h) to node [above]{$-(\varepsilon^{-1 } \circ \pr_0)$} (i);
\end{tikzpicture}\]
It follows from the axioms of a triangulated category that there is a non-canonical isomorphism in the $G$-stable homotopy category $\TCR(\mathbb{S}, \id ;p) \sim H_{ pr_0}  \vee\mathbb{S}$. If we let 
$\mathbb{P}^{\infty}(\mathbb{C})$ be the infinite complex projective space with $G$ acting by complex conjugation, then we can identify the projection $\pr_0: \holim_{i\geq 0 }  \operatorname{H}_{\boldsymbol{\cdot}}(C_{p^i}; \mathbb{S}) \xrightarrow{} \mathbb{S}$  with the $G$-equivariant $\mathbb{T}$-transfer $\Sigma^{\infty}_{G} \Sigma^{1,1}\mathbb{P}^{\infty}(\mathbb{C})\to \mathbb{S}$ after $p$-completion by Theorem \ref{transfer}. If we let $\Sigma^{1,1} \mathbb{P}^{\infty}_{-1}(\mathbb{C})$ denote the $G$-equivariant homotopy fiber of the $\mathbb{T}$-transfer above, then we obtain the following corollary generalizing the classical calculation:

\begin{corollary}\label{TCR(S)}
After $p$-completion, there is an isomorphism in the $G$-stable homotopy category 
\[\TCR(\mathbb{S},\id ;p) \sim \Sigma^{1,1} \mathbb{P}^{\infty}_{-1}(\mathbb{C}) \vee \mathbb{S}.\]
\end{corollary}
\newpage
\appendix

\section{Equivariant homotopy theory}
This appendix recalls some results from equivariant homotopy theory, which we need in the paper.  Let $H$ be a finite group and let $f:  A\to B$ be a map of pointed $H$-spaces. We call the map $f$ $n$-connected respectively a weak-$H$-equivalence if $f^K:~A^K \to B^K$ is $n$-connected respectively a weak equivalence for all subgroups $K\leq H$. Let $\Map_H(A,B)$ denote the space of pointed $H$-equivariant maps.

The following lemma can be found in \cite[Prop. 2.7]{Ad}.
 \begin{lemma}\label{A2}  
Let $f: B \to C$ be a map of pointed $H$-spaces and let $A$ be a pointed $H$-CW-complex. The induced map
\[f_*: \Map_H(A,B)  \to \Map_H(A,C)\]
is $n$-connected with $n\geq \underset{K\leq H}{\min} \{\conn(f^K)-\dim(A^K)\}$, where $K$ runs through all subgroups of $H$. 
\end{lemma}
Let $i: A' \to A$ be an $H$-cofibration. For any pointed $H$-space $B$, the induced map
\[i^*:\Map_H(A,B) \xrightarrow{} \Map_H(A',B)\]
is a fibration with fiber $\Map_H(A/A',B)$. The above Lemma estimates the connectivity of mapping spaces like the fiber by considering the map $f: B \to *$. The estimate amounts to the following Lemma: 
\begin{lemma} \label{conn of mapping space}
Let $A$ be a pointed $H$-CW-complex and let $B$ be a pointed $H$-space. Then 
\[\conn(\Map_H(A,B)) \geq \underset{K\leq H}{\min} (\conn(B^K)-\dim(A^K)), \]
where $K$ runs through all subgroups of $H$. 
\end{lemma}
 
Throughout this paper we will make use of the Equivariant Suspension Theorem, a proof can be found in \cite[Theorem 3.3]{Ad}.

\begin{theorem}[Equivariant Suspension Theorem] \label{suspension}
Let $V$ be a finite dimensional orthogonal $H$-representation and let $A$ be a based $H$-space. The adjunction unit 
\[
\eta: A \to \Omega^V \Sigma^V A
\]
is $n$-connected, where 
\[n \geq \min \{2\cdot \conn(A^H)+1,\conn(A^K)  \mid  K\leq H\text{ with } \dim \ V^K >\dim \ V^H\}.\]
\end{theorem}

Finally, we will need the following lemma.

\begin{lemma} \label{dimension} 
Let $V$ and $W$ be finite dimensional orthogonal $H$-representations. There is a canonical $H$-equivariant homeomorphism
\[ S^{V \oplus W}/S^W \cong \Sigma S^W \wedge S(V)_+.\]
\end{lemma}

\begin{proof}
First note that we have canonical $H$-equivariant homeomorphisms
\begin{align*}
S(V\oplus W) \cong S(V) \star S(W) \cong S(V) \times D(W) \cup_{S(V) \times S(W)} D(V) \times S(W),
\end{align*} 
which gives a canonical $H$-equivariant homeomorphism
\begin{align*}
S(V\oplus W)/S(W) \cong S^W \wedge S(V)_+.
\end{align*}
We have a commutative diagram
\[\begin{tikzpicture}
\node(a){$ S(W)_+$}; \node(b)[right of = a,node distance = 2.5cm]{$S^0 $};
\node(c)[right of = b,node distance = 2.5cm]{$S^W$};
\node(f)[below of = a,node distance = 1.5cm]{$ S(V \oplus W)_+$};
\node(g)[right of = f,node distance = 2.5cm]{$S^0$};
\node(h)[right of = g,node distance = 2.5cm]{$S^{V \oplus W}$};
\node(i)[below of = f,node distance = 1.5cm]{$ S(V \oplus W)/S(W)$};
\node(j)[right of = i,node distance = 2.5cm]{$*$};
\node(k)[right of = j,node distance = 2.5cm]{$S^{V \oplus W}/S^W$};

\draw[->](a) to node [above]{$c$} (b);
\draw[->](b) to node [above]{} (c);

\draw[->](f) to node [above]{$c$} (g);
\draw[->](g) to node [above]{} (h);

\draw[->](i) to node [above]{} (j);
\draw[->](j) to node [above]{} (k);

\draw[right hook->](a) to node [above]{} (f);
\draw[->](b) to node [right]{$\id$} (g);
\draw[right hook->](c) to node [above]{} (h);

\draw[->](f) to node [above]{} (i);
\draw[->](g) to node [left]{} (j);
\draw[->](h) to node [above]{} (k);

\end{tikzpicture}\]
where $c$ collapses the unit-sphere to the non-basepoint, and we can identify
\[S^{V \oplus W}/S^W \cong \Sigma S^W \wedge S(V)_+.\].
\end{proof}

\end{document}